\documentclass[11pt]{article}

\usepackage{geometry,float}
\usepackage{graphicx}
\usepackage{amssymb,mathtools}
\usepackage{amsmath}
\usepackage{amsthm,bm,bbm}
\usepackage{color}
\usepackage{tabularx}
\usepackage{enumitem}
\usepackage[authoryear]{natbib}
\usepackage{dsfont}
\usepackage{algorithm}
\usepackage{algorithmic}
\usepackage{multirow}
\usepackage{JASA_manu}
\let\hat\widehat
\let\tilde\widetilde

\usepackage{url}
\usepackage[colorlinks,citecolor=blue,urlcolor=blue,linkcolor=blue,linktocpage=true]{hyperref}
\usepackage[nameinlink]{cleveref}
\crefformat{equation}{(#2#1#3)}
\crefrangeformat{equation}{(#3#1#4) to~(#5#2#6)}
\crefname{equation}{}{}
\Crefname{equation}{}{}

\newcommand{\single}{\renewcommand{\baselinestretch}{1.2}\normalsize}
\newcommand{\double}{\renewcommand{\baselinestretch}{1.63}\normalsize}

\renewenvironment{thebibliography}[1]{
  \begin{oldthebibliography}{#1}
    \setlength{\parskip}{0ex}
    \setlength{\itemsep}{0ex}
}
{
  \end{oldthebibliography}
}

\allowdisplaybreaks

\newtheorem{theorem}{Theorem}
\newtheorem{lemma}[theorem]{Lemma}
\newtheorem{corollary}[theorem]{Corollary}

\newtheorem{remark}{Remark}
\newtheorem{definition}{Definition}
\newtheorem{example}{Example}
\newtheorem{assumption}{Assumption}
\crefname{definition}{\textbf{definition}}{definitions}
\Crefname{definition}{Definition}{Definitions}
\crefname{assumption}{\textbf{assumption}}{assumptions}
\Crefname{assumption}{Assumption}{Assumptions}

\newcommand{\argmin}{\mathop{\mathrm{argmin}}}

\newtheorem{manassumptioninner}{Assumption}
\newenvironment{manassumption}[1]{%
  \renewcommand\themanassumptioninner{#1}%
  \manassumptioninner
}{\endmanassumptioninner}

\title{Bias-adjusted spectral clustering in multi-layer stochastic block models}
  \author{Jing Lei \\
    Department of Statistics and Data Science, Carnegie Mellon University, USA\\
    and \\
    Kevin Z. Lin \\
    Department of Statistics, Wharton School of Business, University of Pennsylvania, USA}

\begin{document}

\begin{titlepage}
\thispagestyle{empty}
\single
\maketitle

\begin{center}
  \textbf{Abstract}
\end{center}
We consider the problem of estimating common community structures in multi-layer stochastic block models, where each single layer may not have sufficient signal strength to recover the full community structure.  In order to efficiently aggregate signal across different layers, we argue that the sum-of-squared adjacency matrices contain sufficient signal even when individual layers are very sparse. Our method uses a bias-removal step that is necessary when the squared noise matrices may overwhelm the signal in the very sparse regime.  The analysis of our method
relies on several novel tail probability bounds for matrix linear combinations with matrix-valued coefficients and matrix-valued quadratic forms, which may be of independent interest.  The performance of our method and the necessity of bias removal is demonstrated in synthetic data and in microarray analysis about gene co-expression networks.

\vspace*{.1in}


\end{titlepage}

\setcounter{page}{2}
\double

\section{Introduction} \label{sec:intro}
A network records the interactions among a collection of individuals, such as  gene co-expression, functional connectivity among brain regions, and friends on social media platforms.  In the simplest form, a network can be represented by a binary symmetric matrix $A\in\{0,1\}^{n\times n}$ where each row/column represents an individual and the $(i,j)$-entry of $A$ represents the presence/absence of interaction between the two individuals.  In the more general case, $A_{ij}$ may take values in $\mathbb R^1$ to represent different magnitudes or counts of the interaction.  We refer to \cite{Kolaczyk09}, \cite{Newman09}, and \cite{Goldenberg10} for general introduction of statistical analysis of network data.

In many applications, the interaction between individuals are recorded multiple times, resulting in multi-layer network data. For example, in this paper, we study
the temporal gene co-expression networks in the medial prefrontal cortex of rhesus monkeys at ten different developmental stages \citep{bakken2016comprehensive}.  The medial prefrontal cortex is believed to be related to developmental brain disorders, and many of the genes  we study 
are suspected to be associated with autism spectrum disorder at different stages of development. 
Other examples of multi-layer network data are brain imaging, where we may infer one set of interactions among different brain regions from electroencephalography (EEG), and another set of interactions using resting-state functional magnetic resonance imaging (fMRI) measures.  
Similarly, one may expect the brain regions to form groups in terms of connectivity.
The wide applicability and rich structures of multi-layer networks make it an active research area in the statistics, machine learning, and signal processing community. See \cite{TangLD09}, \cite{DongFVN12}, \cite{KivelaABGMP14}, \cite{XuH14}, \cite{HanXA15}, \cite{ZhangC17}, \cite{MatiasM17} and references within.

In this paper, we study multi-layer network data through the lens of multi-layer stochastic block models, where we observe many simple networks on a common set of nodes. The stochastic block model (SBM) and its variants \citep{Holland83,BickelC09,KarrerN11,Airoldi08} are an important prototypical class of network models that allow us to mathematically describe the community structure and understand the performance of popular algorithms such as spectral clustering \citep{McSherry01,RoheCY11,Jin12,LeiR14} and other methods \citep{Latouche12,Peixoto13,AbbeS15}. Roughly speaking, in an SBM, the nodes in a network are partitioned into disjoint communities (i.e., clusters), and nodes in the same community have similar connectivity patterns with other nodes.  A key inference problem in the study of SBM is estimating the community memberships given an observed network.

Compared to an individual layer, a multi-layer network contains more data and hopefully enables us to extract salient structures, such as communities, more easily. On the other hand, new methods must be developed in order to efficiently combine the signal from individual layers.
To demonstrate the necessity for these methods, we plot the observed gene co-expression networks collected from \cite{bakken2016comprehensive} in \Cref{fig:postred}.
The three networks correspond to gene co-expression patterns within the medial prefrontal cortex tissue of rhesus monkeys collected at different stages of development.
We plot only the sub-network formed by a small collection genes for simplicity.
A quick visual inspection across the three networks suggests that the genes can be approximately divided into four common communities (i.e., clusters that persist throughout all three networks), where genes in the same community exhibit similar connectivity patterns. 
However, different gene communities are more visually apparent in different layers.
For example, in the layer labeled as ``E40'' (for tissue collected 40 days of development in the embryo), the last three communities are indistinguishable.
In contrast, in the layer labeled as ``E90,'' the first community is less distinguishable,
and in the layer labeled ``48M'' (for the tissue collected 48 months after birth),
nearly all of the communities are indistinguishable.
These qualitative observations are of scientific interest since these time-dependent densely-connected communities are evidence of ``gene coordination,''  a biological concept that describes when a community of genes
is synchronized in ramping up or down in gene expression at certain stages of
development \citep{paul2012developmental,werling2020whole}.
Hence, we can infer two potential advantages of analyzing such multi-layer network data in an aggregated manner. First, an aggregated analysis is able to reveal global structures that are not exhibited by any individual layer. Second, the common structure across different layers can help us to better filter out the noise, which allows us to obtain more accurate inference results. We describe the analysis in more detail and return to analyze the full dataset in \Cref{sec:data}.


 \begin{figure}[t]
 \begin{center}
           \includegraphics[width=\textwidth]{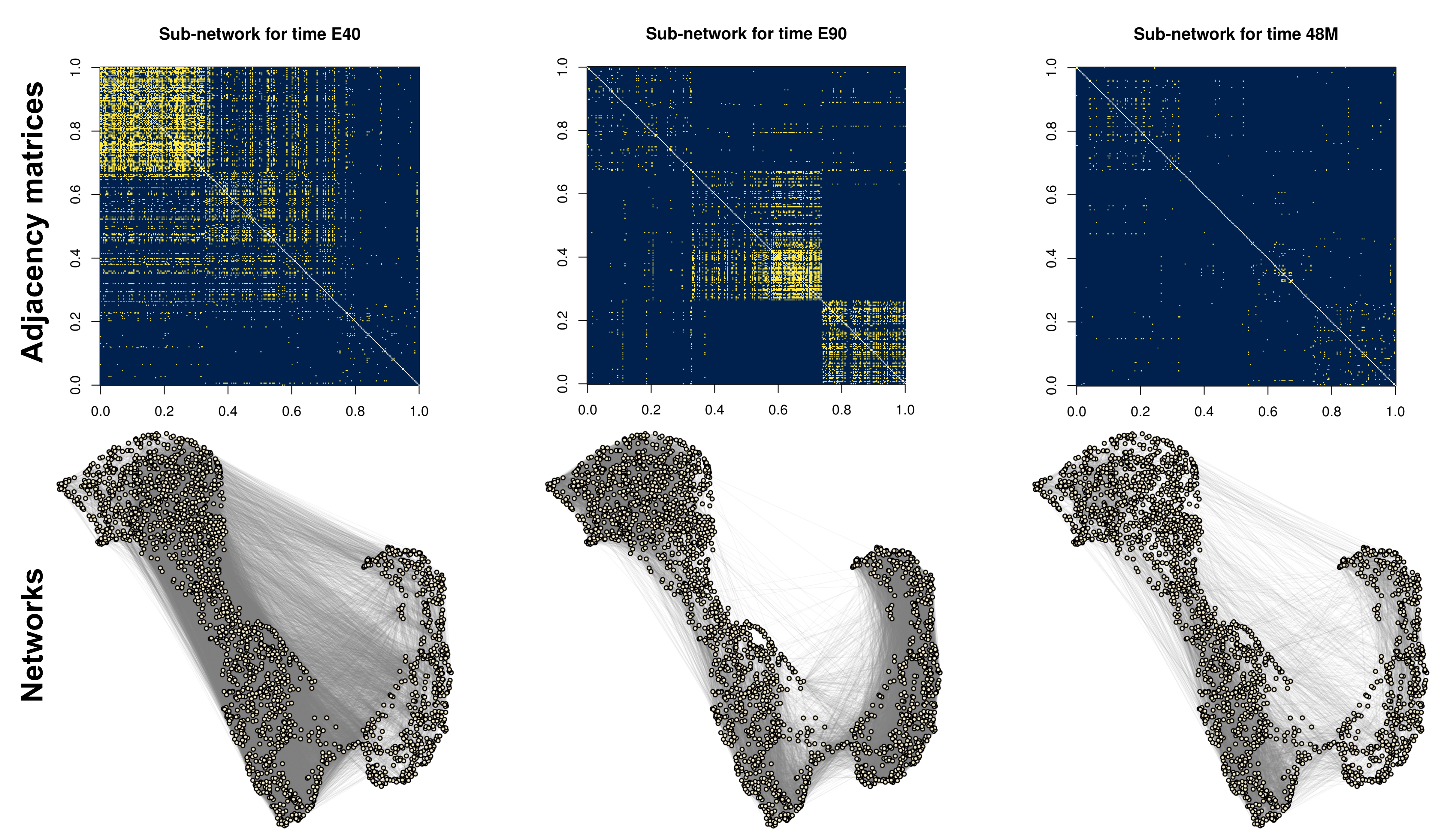}
    \end{center}
    \vspace{-2em}
    \caption{\small The adjacency matrix (top row, yellow denoting the presence of an edge and blue denoting the lack of) and the corresponding network (bottom row)  for three different developmental times of the rhesus monkey's gene co-expression in
    the medial prefrontal cortex based on selected set of genes to visually demonstrate the varying network structures. Likewise, the ordering of the genes in the adjacency matrices is chosen to visually demonstrate the clustering structure, and persist throughout all three adjacency matrices. The three developmental times are E40, E90 (for 40 or 90 days in the embryo) and 48M (for 48 months after birth), corresponding to the pair of plots on the left to the pair of plots on the right. The full dataset is analyzed in \Cref{sec:data}.}
   \label{fig:postred}
\end{figure}

The theoretical understanding of estimating common communities in multi-layer SBMs is relatively limited compared to those in single-layer SBMs. \cite{Sharmo18} and \cite{PaulC17} studied variants of spectral clustering for multi-layer SBMs, but the strong theoretical guarantee requires a so-called layer-wise positivity assumption, meaning each matrix encoding the probability of an edge among the communities must have only positive eigenvalues bounded away from zero.  In contrast, \cite{PenskyZ19} studied a different variant of spectral clustering, but established estimation consistency under conditions similar to those for single-layer SBMs. These results only partially describe the benefits of multi-layer network aggregation. Alternatively, \cite{LeiCL19} considered a least-squares estimator, and proved consistency of the global optima for general block structures without imposing the positivity assumption for individual layers, but that method is computationally intractable in the worst case.

The first main contribution of this paper is a simple, novel, and computationally-efficient aggregated spectral clustering method for multi-layer SBMs, described in \Cref{sec:sbm}.  The estimator applies spectral clustering to the sum of squared adjacency matrices after removing the bias by setting the diagonal entries to $0$. In addition to its simplicity, this estimator has two appealing features.  First, summing over the squared adjacency matrices enables us to prove its consistency without requiring a layer-wise positivity assumption.  Second, compared with single-layer SBMs, the consistency result reflects a boost of signal strength by a factor of $L^{1/2}$, where $L$ is the number of layers.  Such a $L^{1/2}$ signal boost is comparable to that obtained in \cite{LeiCL19}, but is now achieved by a simple and computationally tractable algorithm. 
The removal of the diagonal bias in the squared matrices is shown to be crucial in both theory (\Cref{sec:consistency}) and simulations (\Cref{sec:simu}), especially in the most interesting regime where the network density is too low for any single layer to carry sufficient signal for community estimation.  Interestingly, similar diagonal-removal techniques have also been discovered and studied in other contexts, such as Gaussian mixture model clustering \citep{ndaoud2018sharp}, principal components analysis \citep{zhang2018heteroskedastic}, and centered distance matrices \citep{szekely2014partial}. 

Another contribution of this paper is a collection of concentration inequalities for matrix-valued linear combinations and quadratic forms. 
These are described in \Cref{sec:concentration}, which are an important ingredient for the aforementioned theoretical results.
Specifically, an important step in analyzing our matrix-valued data is to understand the behavior of the matrix-valued measurement errors.  Towards this end, many powerful concentration inequalities have been obtained for matrix operator norms under various settings, such as random matrix theory \citep{BaiS10}, eigenvalue perturbation and concentration theory \citep{FeigeO05,ORourke18,LeiR14,LeLV17,CapeTP17}, and matrix deviation inequalities \citep{BandeiraV16,Vershynin11}.  The matrix Bernstein inequality and related results \citep{Tropp12} are also applicable to linear combinations of noise matrices with scalar coefficients.
In order to provide technical tools for our multi-layer network analysis, we extend 
these matrix-valued concentration inequalities in two directions.  First, we provide upper bounds for linear combinations of noise matrices with matrix-valued coefficients. This can be viewed as an extension of the matrix Bernstein inequality to allow for matrix-valued coefficients.  Second, we provide concentration inequalities for sums of matrix-valued quadratic forms, extending the scalar case known as the Hanson--Wright inequality \citep{HansonW71,RudelsonV13} in several directions. A key intermediate step in relating linear cases to quadratic cases is deriving a deviation bound for matrix-valued $U$-statistics of order two.

\section{Community Estimation in Multi-Layer SBM} \label{sec:sbm}
Throughout this section, we describe the model, theoretical motivation, and our estimator for clustering nodes in a multi-layer SBM.
Motivated by such multi-layer network data with a common community structure as demonstrated in \Cref{fig:postred}, we consider the $L$-layer SBM containing $n$ nodes assigned to $K$ different communities,
\begin{equation}\label{eq:gen_mlsbm}
A_{\ell,ij}\sim {\rm Bernoulli}(\rho B_{\ell,\theta_i \theta_j})\quad \text{for}\quad 1\le i<j\le n,\quad 1\le\ell\le L\,,
\end{equation}
where $\ell$ is the layer index, $\theta_i\in\{1,...,K\}$ is the membership index of node $i$ for $i\in \{1,...,n\}$, $\rho \in (0,1]$ is an overall edge density parameter, and $B_{\ell}\in[0,1]^{K\times K}$ is a symmetric matrix of community-wise edge probabilities in layer $\ell$. We assume $A_\ell$ is symmetric and $A_{\ell,ii}=0$ for all $\ell \in \{1,...,L\}$ and $i \in \{1,...,n\}$.

Our statistical problem is to estimate the membership vector $\theta = (\theta_1,...,\theta_n) \in \{1,...,K\}^n$ given the observed adjacency matrices $A_1,...,A_L$.  Let $\hat\theta\in\{1,...,K\}^n$ be an estimated membership vector, and the estimation error is the number of mis-clustered nodes based on the Hamming distance,
\begin{equation} \label{eq:hamming}
d(\hat\theta,\theta)=\min_{\pi} \sum_{i=1}^n \mathbbm{1}(\theta_i\neq \pi(\hat\theta_i))\,,
\end{equation}
for the indicator function $\mathbbm{1}(\cdot)$,
where the minimum is taken over all label permutations $\pi:\{1,...,K\}\mapsto\{1,...,K\}$. An estimator $\hat\theta$ is consistent if $n^{-1}d(\hat\theta,\theta)=o_P(1)$.

The assumption of a fixed common membership vector $\theta$ can be relaxed to each layer having its own membership vector but close to a common one. The theoretical consequence of this relaxation is discussed in \Cref{rem:varying_membership}, after the main theorem in \Cref{sec:consistency}.
We assume that $K$ is known.  The problem of selecting $K$ from the data is an important problem and will not be pursued in this paper. Further discussion will be given in \Cref{sec:disc}.

When $L=1$, the community estimation problem for single-layer SBM is well-understood \citep{BickelC09,LeiR14,Abbe17}.  If $K$ is fixed as a constant while $n\rightarrow\infty$, $\rho\rightarrow 0$ with balanced community sizes lower bounded by a constant fraction of $n$, and $B$ is a constant matrix with distinct rows, then the community memberships can be estimated with vanishing error when $n\rho\rightarrow\infty$.  Practical estimators include variants of spectral clustering, message passing, and likelihood-based estimators.

As mentioned in \Cref{sec:intro}, in the multi-layer case, consistent community estimation has been studied in some recent works.  The theoretical focus is to understand how the number of layers $L$ affects the estimation problem.  \cite{PaulC17} and \cite{Sharmo18} show that consistency can be achieved if $L n \rho$ diverges, but under the aforementioned positivity assumption, meaning that each $B_{\ell}$ is positive definite with a minimum eigenvalue bounded away from zero. Such assumptions are plausible in networks with strong associativity patterns where nodes in the same communities are much more likely to connect to one another than nodes in different communities.  But there are networks observed in practice that do not satisfy this assumption, such as 
those in  \cite{newman2002assortative} and \cite{litvak2013uncovering}. See \cite{lei2018network} and the references within for additional discussion on such positivity assumptions in a more general context. 
To remove the positivity assumption, \cite{LeiCL19} considered a least-squares estimator, and proved consistency when $L^{1/2} n\rho$ diverges (up to a small poly-logarithmic factor) and the smallest eigenvalue of $\sum_\ell B_\ell^2$ grows linearly in $L$.  A caveat is that the least-squares estimator is computationally challenging, and in practice, one may only be able to find a local minimum using greedy algorithms. 

In the following subsections, we will motivate a spectral clustering method from the least-squares perspective, investigate its bias, and derive our estimator with a data-driven bias adjustment.

\subsection{From least squares to spectral clustering}
In this subsection, we motivate how least-squares estimators is well-approximated by spectral clustering, which lays down the intuition of our estimator in \Cref{sec:method}.
Let $\psi \in \{1,...,K\}^n$ be a membership vector and $\Psi=[\Psi_1,...,\Psi_K]$ be the corresponding $n\times K$ membership matrix where each $\Psi_k=(\Psi_{1,k},...,\Psi_{n,k})^T$ is an $n\times 1$ vector with $\Psi_{i,k}=\mathbbm{1}(\psi_i=k)$.  Let $I_k(\psi)=\{i \in \{1,...,n\}:\psi_i=k\}$ and $n_k(\psi)=|I_k(\psi)|$,
the size of the set $I_k(\psi)$.

The least-squares estimator of \cite{LeiCL19} seeks to minimize the residual sum of squares,
\begin{align}\label{eq:LS}
  \hat \theta = \argmin_{\psi \in \{1,...,K\}^n} \sum_{\ell=1}^L \sum_{1\le i < j\le n}(A_{\ell,ij}-\hat B_{\ell,\psi_i \psi_j}(\psi))^2
\end{align}
where
$$
\hat B_{\ell,kl}(\psi)=\left\{\begin{array}{cc}
  \frac{\sum_{i,j\in I_k(\psi)} A_{\ell,ij}}{n_k(\psi)(n_k(\psi)-1)} &\text{when}\quad k=l,\\
  \frac{\sum_{i\in I_k(\psi),j\in I_l(\psi)} A_{\ell,ij}}{n_k(\psi)n_l(\psi)} &\text{when}\quad k\neq l,
\end{array}\right.
$$
is the sample mean estimate of $B_\ell$ under a given membership vector $\psi$. 
Recall that the total-variance decomposition implies the equivalence between minimizing within-block sum of squares and maximizing between-block sum of squares.
Hence, if we accept the approximation $n_k(\psi)(n_k(\psi)-1)\approx n_k^2(\psi)$, then 
after multiplying the least-squares objective function \eqref{eq:LS} by $2$ and using the total-variance decomposition, the objective function becomes
  \begin{align*}
    \max_{\psi \in \{1,...,K\}^n}\sum_{\ell=1}^L \sum_{1\le k,l\le K}\frac{(\Psi_k^T A_\ell \Psi_l)^2}{n_k(\psi)n_l(\psi)}\,,
  \end{align*}
%
which is equivalent to
 \begin{align*}
  \max_{\psi \in \{1,...,K\}^n} \sum_{\ell=1}^L \sum_{1\le k,l\le K} \left(\tilde \Psi_k^T A_{\ell} \tilde \Psi_l\right)^2=\max_\psi \sum_{\ell=1}^L \left\|\tilde \Psi^T A_{\ell} \tilde \Psi\right\|_F^2\,,
\end{align*}
where $\|\cdot\|_F$ denotes the matrix Frobenius norm, and $\tilde \Psi=[\tilde \Psi_1,...,\tilde \Psi_K]$ with $\tilde \Psi_k=\Psi_k/\sqrt{n_k(\psi)}$ is the column-normalized version of $\Psi$ where each column of $\tilde \Psi$ has norm $1$.
This means $\tilde \Psi$ is orthonormal, i.e., $\tilde \Psi^T\tilde \Psi=I_K$. 
The benefit of considering orthonormal matrices is that
for any orthonormal matrix $U\in\mathbb R^{n\times K}$ and symmetric matrix $A \in \mathbb{R}^{n\times n}$,
$$
\|U^T A U\|_F^2 = {\rm tr}(U^T A UU^T A U) \le {\rm tr}(U^T A^2 U)\,.
$$
The right-hand side of the above inequality is maximized by the leading $K$ eigenvectors of $A$, where the eigenvalues ordered by absolute value. For this $U$, the inequality becomes equality.  Additionally, under the multi-layer SBM, the expected values of adjacency matrices $\{P_1, ..., P_L\}$ (where $P_\ell=\mathbb E A_\ell$ for $\ell \in \{1,...,L\}$) share roughly the same leading principal subspace as determined by the common community structure. Putting all these facts together, we intuitively expect $U=\tilde \Theta$ to correspond to an approximate solution of the original least-squares problem, where $\tilde\Theta$ is the column-normalized version of the true membership matrix $\Theta$.

Therefore, a relaxation of the approximate version of the original problem \eqref{eq:LS}  is
\begin{align} 
  \max_{U\in \mathbb R^{n\times K}: U^TU=I_K} {\rm tr} \left[U^T\left(\sum_{\ell=1}^L A_\ell^2\right)U\right]\,,\label{eq:SSS_bias}
\end{align}
which is a standard spectral problem. For this reason, we often call $U$ the ``spectral embedding.'' The community estimation is then obtained by applying a clustering algorithm to the rows of $\hat U$, a solution to \eqref{eq:SSS_bias}.

\subsection{The necessity of bias adjustment} \label{sec:necessity}
Let $P_\ell=\mathbb E A_\ell$ denote the expected adjacency matrix, meaning that $P_\ell$ is the matrix obtained by zeroing out the diagonal entries of $\tilde P_\ell = \rho\Theta B_\ell \Theta^T$. We now show that $\sum_{\ell}A^2_{\ell}$ is a biased estimate of $\sum_{\ell}P^2_{\ell}$, and that we can correct for this bias by simply removing its diagonal entries.  Let $X_\ell=A_\ell-P_\ell$ be the noise matrix.
Then
\begin{align}
  \sum_{\ell=1}^{L} A_\ell^2 = \Big(\sum_{\ell=1}^{L}  P_\ell^2\Big) + \Big(\sum_{\ell=1}^{L} (X_\ell P_\ell+ P_\ell X_\ell)\Big)+S\,,\label{eq:decomp}
\end{align}
where $S=\sum_\ell X_\ell^2$.
The first term is the signal term, with each summand close to $\tilde P_\ell^2=\rho^2\Theta B_\ell^2\Theta^T$, and will add up over the layers, because each matrix $B_\ell^2$ is positive semi-definite. The second term is a mean-$0$ noise matrix, which can be controlled using matrix concentration inequalities developed in \Cref{sec:concentration} below. The third term $S=\sum_\ell X_\ell^2$ is a squared error matrix and will also add up over the layers, which may introduce bias if the overall edge density parameter $\rho$ is too small.

\begin{figure}[t]
  \begin{center}
    \includegraphics[width=300px]{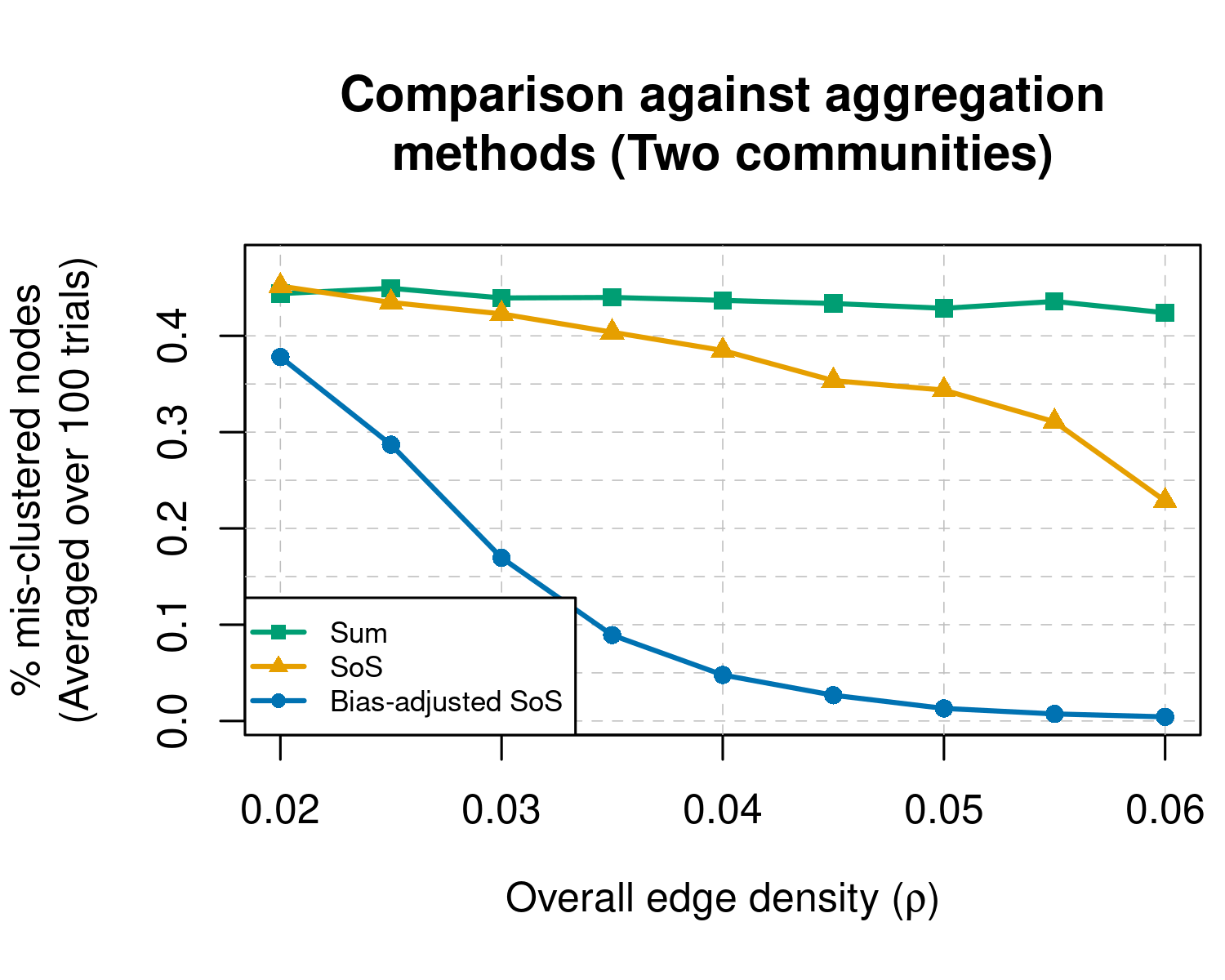}
    \vspace{-2em}
    \caption{\small The average proportion of mis-clustered nodes for three methods (measured via Hamming distance $n^{-1}d(\hat\theta,\theta)$ shown in \eqref{eq:hamming}, averaged over 100 trials), with $n=200$ and two equal-sized communities among overall edge densites ranging from $\rho \in [0.02,0.06]$ and $L=30$ layers. Three methods' performance are shown: Sum (green squares), SoS (orange triangles), and Bias-adjusted SoS (blue circles).} \label{fig:simu}
  \end{center}
\end{figure}

We use a simple simulation study to illustrate the necessity of bias adjustment in spectral clustering applied to the sum of squared adjacency matrices.  We set $K=2$ and consider two edge-probability matrices,
$$
B^{(1)}=\left[\begin{array}{cc}
  3/4 & \sqrt{3}/8 \\ \sqrt{3}/8 & 1/2
\end{array}\right]\,,\quad\text{and}\quad
B^{(2)}=\left[\begin{array}{cc}
  7/8 & 3\sqrt{3}/8 \\ 3\sqrt{3}/8 & 1/8
\end{array}\right]\,.
$$
These two matrices are chosen such that spectral clustering applied to the sum of the adjacency matrices and the sum of squared adjacency matrices would be either sub-optimal or inconsistent in the very sparse regime. We set $n=200$ nodes with $100$ nodes in each community, the number of layers to be $L=30$, and for each layer $\ell$, $B_{\ell}$ is randomly and independently chosen from $B^{(1)}$ and $B^{(2)}$ with equal probability. We use five different values of the overall edge density parameter $\rho$ between $0.02$ and $0.06$. For each value of $\rho$, we generate a multi-layer SBM according to \eqref{eq:gen_mlsbm} and apply spectral clustering to three matrices: (1) the sum of adjacency matrices without squaring (i.e., ``Sum''), (2) the sum of squared adjacency matrices (i.e., ``SoS''), and (3) a bias-adjusted sum of squared adjacency matrices (i.e., ``Bias-adjusted SoS''), which will be introduced in the next subsection.
The results across 100 trials are  reported in \Cref{fig:simu}. 
By construction, the ``Sum'' method  performs poorly since the sum of adjacency matrices has only one significant eigen-component, meaning the result is sensitive to noise when $K=2$ eigenvectors are used for spectral clustering. 
In fact, as described in \Cref{exa:1} below, it is also easy to generate cases in which the sum of adjacency matrices carries no signal at all.
The ``SoS'' method also performs poorly. This is because although the sum of squared adjacency matrices contains signal for clustering, the aforementioned bias is large when $\rho$ is small.  In contrast, our method 
``Bias-adjusted SoS'' performs the best.  A more detailed simulation study is presented in \Cref{sec:simu}.

\subsection{Bias-adjusted sum-of-squared spectral clustering} \label{sec:method}
We are now ready to quantify the amount of bias, and to describe our aforementioned bias-adjusted sum-of-squared method to cluster nodes in a multi-layer SBM.
From \eqref{eq:decomp}, we see that the diagonal entries of the squared error term $S$ have positive expected value and hence may cause systematic bias in the principal subspace of $\sum_\ell A_\ell^2$.  Now consider a further decomposition $S=S_1+S_2$ where $S_1$ and $S_2$ correspond to the off-diagonal and diagonal parts of $S$, respectively.  
Observe that only the diagonal entries of $S$ have positive expected value,
so our effort will focus on removing the bias caused by $S_2$. 
Towards this end, observe that by construction, we have
\begin{align}
(S_2)_{ii}=&S_{ii}=\sum_{\ell=1}^{L}\sum_{j=1}^{n} X_{\ell,ij}^2  \nonumber\\
=& \sum_{\ell=1}^{L}  \sum_{j=1}^{n} P_{\ell,ij}^2\mathbbm{1}(A_{\ell,ij}=0)  + (1-P_{\ell,ij})^2\mathbbm{1}(A_{\ell,ij}=1)\nonumber\\
\le & Ln \max_{\ell,ij}P_{\ell,ij}^2 + \sum_{\ell=1}^{L} d_{\ell,i}\,\label{eq:S2_diag_decomp}
\end{align}
where $d_{\ell,i}=\sum_{j}A_{\ell,ij}$ is the degree of node $i$ in layer $\ell$.  The expected value of $\sum_\ell d_{\ell,i}$ is $\sum_{\ell,j}P_{\ell,ij}\asymp Ln \max_{\ell,ij}P_{\ell,ij}$. In the very sparse regime, $\max_{\ell,ij}P_{\ell,ij}$ is very small so $\sum_\ell d_{\ell,i}$ is the leading term in $(S_2)_{ii}$.  

Combining this calculation with a key observation that $\sum_\ell d_{\ell,i}$ can be computed from the data, we arrive at the following bias-adjusted sum-of-squared  spectral clustering algorithm.
Let $D_\ell$ be the diagonal matrix consisting of the degrees of $A_{\ell}$ where $(D_\ell)_{ii}=d_{\ell,i}$. The bias-adjusted sum of squared adjacency matrices is 
\begin{equation}\label{eq:S0}
S_0=\sum_{\ell=1}^{L}(A_\ell^2- D_\ell)\,.
\end{equation}
The community membership is estimated by applying a clustering algorithm to the rows of the matrix whose columns are the leading $K$ eigenvectors of $S_0$ given in \eqref{eq:S0}.

\section{Consistency of bias-adjusted sum-of-squared spectral clustering} \label{sec:consistency}
We now describe our theoretical result characterizing how multi-layer networks benefit community estimation.
The hardness of community estimation is determined by many aspects of the problem, including number of communities, community sizes, number of nodes, separation of communities, and overall edge density.  Here, we need to consider all of these aspects jointly across the $L$ layers.  To simplify the discussion, we primarily focus on the following setting but discuss additional settings in later remarks.
\begin{assumption}\label{ass:signal}
\begin{enumerate}
  \item [(a)] The number of communities $K$ is fixed and community sizes are balanced.
  That is, there exists  a constant $c$ such that each community size is in $[c^{-1}n/K,cn/K]$.
  \item [(b)] The relative community separation is constant.  That is, $B_\ell=\rho B_{\ell,0}$ where $B_{\ell,0}$ is a $K\times K$ symmetric matrix with constant entries in $[0,1]$. 
  Furthermore, the minimum eigenvalue of $\sum_\ell B_{\ell,0}^2$ is at least $cL$ for some constant $c>0$.
\end{enumerate}
\end{assumption}
Part (a) simplifies the effect of the community sizes and the number of communities. This setting has been well-studied in the SBM literature for $L=1$ \citep{LeiR14}.  Part (b) puts the focus on the effect of the overall edge density parameter $\rho$, and requires a linear growth of the aggregated squared edge-probability matrices in terms of the minimum eigenvalue. This is much less restrictive than the layer-wise positivity assumption used in other work mentioned in \Cref{sec:sbm} which require each $B_{\ell,0}$ to be positive definite.   We give two examples in which \Cref{ass:signal}(b) is satisfied but the layer-wise positivity is not.

\begin{example}[Identicially distributed random layers]\label{exa:1}
    Consider a theoretical scenario in which the $B_{\ell,0}$'s have i.i.d. $\text{Uniform}(0,1)$ entries subject to symmetry.  It is easy to verify that the expected sum matrix $\mathbb E \sum_\ell B_\ell$ is a constant matrix with each entry being $L\rho/2$.  Therefore it is impossible to reconstruct the block structure from the sum of adjacency matrices $\sum_\ell A_\ell$ when $\rho$ is small.
\end{example}

\begin{example}[Community merge and split]
  Consider a more realistic scenario in which for $\{B_\ell:1\le\ell\le L\}$,
  some layers $\ell$ and community indices $k,k'$ have $B_{\ell,kj}=B_{\ell,k'j}$ for all $j$.  This can be interpreted as the merge of communities $k$ and $k'$ at layer $\ell$.  In such cases, each layer may not contain full community information, and we must aggregate the layers to recover the full community structure.  In our real data example, we actually observe that in most layers,  all but one or two communities merge with a large, null community, and each non-null community is active in one or two layers.
\end{example}

Based on these assumptions, in the asymptotic regime $n\rightarrow\infty$ and $\rho\rightarrow 0$, it is well-known that consistent community estimation is possible for $L=1$ when $n\rho\rightarrow\infty$. Hence, in the multi-layer setting when $L\rightarrow\infty$, one should expect a lower requirement on overall density as we aggregate information across layers. This is shown in our following result.




\begin{theorem}\label{thm:consistency}
  Under \Cref{ass:signal}, if $L^{1/2}n\rho\ge C_1\log^{1/2}(L+n)$ and $n\rho\le C_2$ for a large enough positive constant $C_1$ and a positive constant $C_2$, then spectral clustering with a constant factor approximate K-means clustering algorithm applied to $S_0$, the bias-adjusted sum of squared adjacency matrices  in \eqref{eq:S0},
  correctly estimates the membership of all but a
  $$
  C\left(\frac{1}{n^2}+\frac{\log(L+n)}{Ln^2\rho^2}\right)
  $$
proportion of nodes for some constant $C$ with probability at least $1-O((L+n)^{-1})$.
\end{theorem}
An immediate consequence of \Cref{thm:consistency} is the Hamming distance consistency of the bias-adjusted sum-of-squared spectral clustering, provided that $L^{1/2}n\rho/\log^{1/2} (L+n)\rightarrow\infty$. This demonstrates the boost of signal strength by a factor of $L^{1/2}$ made possibly due to aggregating layers (up to a poly-logarithmic factor) that we alluded to in \Cref{sec:intro}.

The proof of \Cref{thm:consistency} is given in \Cref{app:consistency}, where the main effort is to establish sharp operator norm bounds for the linear noise term $\sum_\ell X_\ell P_\ell$ and the quadratic noise term $\sum_\ell (X_\ell^2- D_\ell)$.  A refined operator norm bound for the off-diagonal part of $\sum_\ell (X_\ell^2- D_\ell)$ plays an important role (\Cref{thm:s1_sparse_bernoulli}).  Once the operator norm bound is established, the clustering consistency follows from a standard analysis of the K-means algorithm (\Cref{lem:k-means}).  These concentration inequalities indeed hold for more general classes of matrices, and we provide a systematic development in the next section.

\Cref{thm:consistency} is stated in a simple form for brevity.  It can be generalized in several directions to better suit practical scenarios with more careful bookkeeping in the proof.  We describe some important extensions in the remarks below, where $\|\cdot \|$ denotes the operator norm (i.e., largest singular value).

\begin{remark}[Varying membership across layers]\label{rem:varying_membership}
    \Cref{thm:consistency} can be extended to accommodate varying membership across the layers. In particular, assume that the $\ell$th layer has membership matrix $\Psi_\ell \in \{0,1\}^{n\times K}$, such that each $\Psi_\ell$ is close to a common membership matrix $\Psi  \in \{0,1\}^{n\times K}$,
    \begin{equation}\label{eq:varying_membership_condition}
\|\Psi_\ell-\Psi\|\le \epsilon_\ell \sqrt{n}\,,
    \end{equation}
    for some positive constant $\epsilon_\ell$.
Then we have the following generalization of \Cref{thm:consistency}.
\end{remark}
\begin{corollary}[Consistency under varying membership]\label{cor:varying_membership}
Assume the multilayer adjacency matrices $A_1,\ldots,A_L$ are generated from individual membership matrices $\Psi_1,\ldots,\Psi_L$ satisfying \eqref{eq:varying_membership_condition} for some sequence $\epsilon_1,\ldots,\epsilon_L$ and common membership matrix $\Psi$.  Under the same condition as in \Cref{thm:consistency}, if in addition $\bar\epsilon:= L^{-1}\sum_\ell\epsilon_\ell \le C_3$ for some positive constant $C_3$, then the error bound of the bias-adjusted sum of squared spectral clustering is no more than $$C\left(\frac{1}{n^2}+\bar\epsilon^2+\frac{\log(L+n)}{Ln^2\rho^2}\right)$$
with high probability.
\end{corollary}

\begin{remark}[Other regimes of network density]\label{rem:n_rho>1}
     The condition $L^{1/2}n\rho \geq C_1 \log^{1/2}(L+n)$ is required in order for the error bound in \Cref{thm:consistency} to imply consistency, and is suitable for the linear squared signal accumulation assumed in Part (b) of \Cref{ass:signal}.  If we assume a different growth speed of the minimum eigenvalue of $\sum_{\ell}B^2_{\ell,0}$, this requirement needs to be changed accordingly.
Second, the condition $n\rho\lesssim 1$ is used for notational simplicity.  The regime $n\rho\gg 1$ would allow for consistent community recovery even when $L=1$. For multilayer models, if $n\rho \geq C_2$ for some constant $C_2$, the error bound in \Cref{thm:consistency} becomes $$
C\left(\frac{1}{n^2}+\frac{\log(L+n)}{Ln\rho}\right)\,.
$$
for some constant $C$ with high probability.
Detailed explanations of this claim are given in \Cref{app:consistency}.
\end{remark}

\begin{remark}[More general conditions on community sizes]\label{rem:K_n_min}
Let $n_{\min}=\min_{1\le k\le K}\|\Psi_{\cdot k}\|_1$ be the size of the smallest community, and denote $\alpha=n_{\min}/n$.
Our analysis can also allow the number of communities, $K$, and $\alpha$ to change with other model parameters $(n,L, \rho)$.  In particular, the lower bound of the signal term in \eqref{eq:decomp} will be multiplied by $\alpha$ since the operator norm of $\Psi$ is proportional to $\alpha$.  All the matrix concentration results, such as \Cref{thm:s1_sparse_bernoulli} and \Cref{lem:counting} still hold as they do not rely on any block structures.  Therefore under the same setting as \Cref{thm:consistency}, if we allow $K$ and $\alpha$ to vary with $(n,L,\rho)$, but have $\alpha L^{1/2}n\rho \ge C_1\log^{1/2}(L+n)$ for some constant $C_1$, then with high probability,\Cref{thm:consistency} holds with error bound
$$
C K\alpha^{-2}\left(\frac{1}{n^2} +\frac{\log(L+n)}{L n^2\rho^2}\right)\,.
$$
\end{remark}


\section{Matrix Concentration Inequalities} \label{sec:concentration}
We generically consider a sequence of independent matrices $X_1,...,X_L\in \mathbb R^{n\times r}$ with independent mean-0 entries. The goal is to provide upper bounds for operator norms 
of linear combinations of the form $\sum_\ell X_\ell 
H_\ell$ with $H_\ell \in \mathbb R^{r\times m}$ for $\ell \in \{1,...,L\}$, and quadratic 
forms $\sum_\ell X_\ell G_\ell X_\ell^T$ with $G_\ell\in \mathbb R^{r\times 
r}$ for $\ell \in \{1,...,L\}$. Here, $H_\ell$ and $G_\ell$ are non-random.  To connect with the notations in previous sections, let $H_\ell=P_\ell$, then an operator norm bound of $\sum_\ell X_\ell P_\ell$ will help control the second term in \eqref{eq:decomp}.  Let $G_\ell=I_r$ be the $r \times r$ identity matrix, then $\sum_\ell X_\ell G_\ell X_\ell^T$ corresponds to the third term in \eqref{eq:decomp}.  Our general results cover both the symmetric and asymmetric cases, as well as more general entries of $X_\ell$ beyond the Bernoulli case.

Concentration inequalities usually require tail conditions on the entries of $X_\ell$.  A standard tail condition for scalar random variables is the Bernstein tail condition.
\begin{definition}\label{def:bern}
  We say a random variable $Y$ satisfies a $(v,R)$-Bernstein tail condition (or is $(v,R)$-Bernstein), if $\mathbb E[|Y|^k]\le \frac{v}{2}k! R^{k-2}$ for all integers $k\ge 2$.
\end{definition}

%
%

The Bernstein tail condition leads to concentration inequalities for sums of independent random variables \citep[][Chapter 2]{vanderVaartAndWellner}.  
Since we are interested not only in linear combinations of $X_\ell$'s, but also the quadratic forms involving $X_\ell G_\ell X_\ell^T$, we need the Bernstein condition to hold for the squared entries of $X_1,...,X_L$.  Specifically we consider the following three assumptions.

\begin{assumption}\label{ass:bern1}
  Each entry $X_{\ell,ij}$ is $(v_1,R_1)$-Bernstein, for all $\ell \in \{1,...,L\}$ and $i,j \in \{1,...,n\}$.
\end{assumption}

\begin{assumption}\label{ass:bern2}
  Each squared entry $X_{\ell,ij}^2$ is $(v_2,R_2)$-Bernstein, for all $\ell \in \{1,...,L\}$ and $i,j \in \{1,...,n\}$.
\end{assumption}

\begin{manassumption}{\ref*{ass:bern2}'}\label{ass:bern2p}
  The product $X_{\ell,ij}\tilde X_{\ell,ij}$ is $(v_2',R_2')$-Bernstein, for all $\ell \in \{1,...,L\}$ and $i,j \in \{1,...,n\}$, where $\tilde X_\ell$ is an independent copy of $X_\ell$.
\end{manassumption}

There are two typical scenarios in which such a squared Bernstein condition in \Cref{ass:bern2} holds. The first is the sub-Gaussian case: If a random variable $Y$ satisfies the sub-Gaussian condition $\mathbb E e^{Y^2/\sigma^2}\le 2$ for some $\sigma> 0$, then we have $\mathbb E Y^{2k}\le 2\sigma^4 (\sigma^2)^{k-2}k!$, and hence $Y^2$ is $(4\sigma^4,\sigma^2)$-Bernstein.  The second scenario is centered Bernoulli: If a random variable $Y$ satisfies $\mathbb P(Y=1-p)=1-\mathbb P(Y=-p)=p$ for some $p\in [0,1/2]$, then we have $\mathbb EY^{2k}=p(1-p)^{2k}+(1-p)p^{2k}\le p$, and hence $Y^2$ is $(2p,1)$-Bernstein.  Our proof will also use the fact that if $Y^2$ is $(v_2,R_2)$-Bernstein, then the centered version $Y^2-\mathbb E (Y^2)$ is also $(v_2,R_2)$-Bernstein \citep[Lemma 3]{wang2016average}.

We require \Cref{ass:bern2p} in order to use a decoupling technique in establishing concentration of quadratic forms.  One can show that if \Cref{ass:bern2} holds then \Cref{ass:bern2p} holds with $(v_2',R_2')=(v_2,R_2)$. However, when $X_{\ell,ij}$'s are centered Bernoulli random variables with parameters bounded by $p\le 1/2$, then \Cref{ass:bern2p} holds with $v_2'=2p^2$ and $R_2'=1$, while \Cref{ass:bern2} holds with $v_2=2p$ and $R_2=1$, so that $v_2'$ can potentially be much smaller than $v_2$.  We will explicitly keep track of the Bernstein parameters in our results for the sake of generality.

\subsection{Linear combinations with matrix coefficients}

\begin{theorem}\label{lem:mat_prod}
  Let $X_1,...,X_L$ be a sequence of independent $n\times r$ matrices with mean-$0$ independent entries satisfying \Cref{ass:bern1}, and $H_\ell$ be any sequence of $r\times m$ non-random matrices. Then for all $t>0$,
  \begin{align}
     &\mathbb P\left[\left\|\sum_{\ell=1}^L X_\ell H_\ell\right\|\ge t\right]\nonumber\\
    \le & 2(m+n)\exp\left(-\frac{t^2/2}{v_1\left(n\left\|\sum_{\ell}H_\ell^T H_\ell\right\|\vee \sum_\ell \|H_\ell\|_F^2\right)+R_1\max_\ell\|H_\ell\|_{2,\infty}t}\right)\,.\label{eq:mat_prod}
  \end{align}%
A similar result holds, with $t^2/2$ replaced by $t^2/8$ and $2(m+n)$ replaced by $4(m+n)$ in \eqref{eq:mat_prod}, for symmetric $X_\ell$'s of size $n\times n$ with independent $(v_1,R_1)$-Bernstein diagonal and upper-diagonal entries and $H_\ell$ of size $n\times m$.
\end{theorem}

The proof of \Cref{lem:mat_prod}, given in \Cref{app:1}, combines the matrix Bernstein inequality \citep{Tropp12} for symmetric matrices and a rank-one symmetric dilation trick (\Cref{lem:rank1_dil}) to take care of the asymmetry in $X_\ell H_\ell$.

%

%
%
%


\begin{remark}
  If $n=m=r=1$, then \Cref{lem:mat_prod} recovers the well-known Bernstein's inequality as a special case with a different pre-factor.

  If $n\ge \min\{m,Lr\}$, then $n\|\sum_\ell H_\ell^T H_\ell\|\ge \sum_\ell \|H_\ell\|_F^2$ and the probability upper bound in \Cref{lem:mat_prod} reduces to
  \begin{align}
    \mathbb P\left[\left\|\sum_{\ell=1}^L X_\ell H_\ell\right\|\ge t\right]\le 2(m+n)\exp\left(-\frac{t^2/2}{v_1 n\left\|\sum_{\ell}H_\ell^T H_\ell\right\|+R_1\max_\ell\|H_\ell\|_{2,\infty}t}\right)\,.\label{eq:sum_prod_large_n}
  \end{align}

  If $n=1$ then $n\|\sum_\ell H_\ell^T H_\ell\|\le \sum_\ell \|H_\ell\|_F^2$
  and the probability bound reduces to
  \begin{align}
    \mathbb P\left[\left\|\sum_{\ell=1}^L X_\ell H_\ell\right\|\ge t\right]\le 2(m+n)\exp\left(-\frac{t^2/2}{v_1 \sum_{\ell}\left\| H_\ell\right\|_F^2+R_1\max_\ell\|H_\ell\|_{2,\infty}t}\right)\,.\label{eq:sum_prod_small_n}
  \end{align}
\end{remark}

\begin{remark}
  When $L=1$, the setting is similar to that considered in \cite{Vershynin11}.
  In the constant variance case (e.g., sub-Gaussian), $v_1^{1/2}\asymp R_1\asymp 1$, \Cref{lem:mat_prod} implies a high probability upper bound of $C\sqrt{\log(m+n)}(\sqrt{n}\|H\|+\|H\|_F)$, which agrees with Theorem 1.1 of \cite{Vershynin11}.  The extra $\sqrt{\log(n+m)}$ factor in our bound is because  our result is a tail probability bound while \cite{Vershynin11} provides upper bounds on the expected value.  However, in the sparse Bernoulli setting, where $v_1\ll R_1=1$, the upper bound in \Cref{lem:mat_prod} is better because it correctly captures the $\sqrt{v_1}$ factor multiplied by $\sqrt{n}\|H\|+\|H\|_F$, whereas the result in \cite{Vershynin11} leads to $v_1^{1/4}(\sqrt{n}\|H\|+\|H\|_F)$.  
\end{remark}

\subsection{Matrix $U$-statistics and quadratic forms}
Let
\begin{align}
S= &\sum_{\ell=1}^{L} X_\ell G_\ell X_\ell^T= \sum_{\ell=1}^{L}\sum_{(i,j),(i',j')} X_{\ell,ij}X_{\ell,i'j'}e_i e_{i'}^T G_{\ell,jj'}\label{eq:total_decomp_S}
\end{align}
where the summation is taken over all pairs $(i,j),(i',j')\in \{1,...,n\}^2$ and $e_i$ is the canonical basis vector in $\mathbb R^n$ with a 1 in the $i$th coordinate. In this subsection, we will focus on the symmetric case because the bookkeeping is harder compared to the asymmetric case.  The treatment for the asymmetric case is similar and the corresponding results are stated separately in \Cref{sec:asymmetric} for completeness.

Because $X_\ell$ has centered and independent diagonal and upper diagonal entries, a term in \eqref{eq:total_decomp_S} has non-zero expected value only if $(i,j)=(i',j')$ or $(i,j)=(j',i')$ since this would imply $X_{\ell,ij}X_{\ell,i'j'}=X_{\ell,ij}^2$.
This motivates the following decomposition of $S$ into a quadratic component with non-zero entry-wise mean value
\begin{align}
S_2=&\Big[\sum_{\ell=1}^{L}\sum_{1\le i < j\le n} X_{\ell,ij}^2\left(e_ie_i^TG_{\ell,jj}+e_je_j^T G_{\ell,ii}+e_ie_j^TG_{\ell,ji}+e_je_i^T G_{\ell,ij}\right)\Big]\nonumber\\
&+\Big[\sum_{\ell=1}^{L}\sum_{1\le i\le n}X_{\ell,ii}^2 e_i e_i^T G_{\ell,ii}\Big]\,,\label{eq:S2}
\end{align}
 and a cross-term component with entry-wise mean-0 value
\begin{equation}
S_1= S-S_2\,. \label{eq:S1}
\end{equation}

It is easy to check that $\mathbb E S_2=\mathbb E S$ and $\mathbb E S_1=0$.  Intuitively, the spectral norm of $S_1$ should be small since it is the sum of many random terms with zero mean and small correlation, which can be viewed as a $U$-statistic with a centered kernel function of order two.  This $U$-statistic perspective is  a key component of the analysis and will be made clearer in the proof.  
For a similar reason, $S_2-\mathbb E S_2$ should also be small.  Hence, the main contributing term in $S$ should be the deterministic term $\mathbb E S_2$.
%
%
To formalize this, define the following quantities,
\begin{align*}
\sigma_1^2=&\sum_{\ell=1}^{L}\|G_\ell\|^2,\\
\sigma_2=&\max_\ell\max\left\{\|G_\ell\|_{2,\infty},\|G_\ell^T\|_{2,\infty}\right\}
\\
(\sigma_2')^2=&\sum_{\ell=1}^{L}\sum_{j=1}^{n}G_{\ell,jj}^2,\\
\sigma_3=&\max_\ell\|G_\ell\|_\infty\,,
\end{align*}
where $\|\cdot\|_{2,\infty}$ is the maximum $L_2$-norm of each row, and $\|\cdot\|_\infty$ is the maximum entry-wise absolute value.
The following theorem quantifies the random fluctuations of $S_1$, $S_2$ and $S$ around
their expectations.

\begin{theorem}\label{thm:quad_sum}
  If $X_1,...,X_L$ are independent $n\times n$ symmetric matrices with independent diagonal and upper diagonal entries satisfying \Cref{ass:bern1} and \Cref{ass:bern2p}. Let $G_1,...,G_L$ be $n\times n$ matrices. Define $S=\sum_\ell X_\ell G_\ell X_\ell^T$ and $S_1, S_2$ as in \eqref{eq:S2} and \eqref{eq:S1}.
  Then there exists a universal constant $C$ such that with probability at least $1-O((n+L)^{-1})$,
  \begin{align}
    \|S_1\|
    \le & C\bigg[
    v_1n \log(L+n)\sigma_1+\sqrt{v_1}R_1\sqrt{Ln}\log^{3/2}(L+n)\sigma_2\nonumber\\
&\quad+\sqrt{v_2'}\log(L+n)(\sqrt{L}\sigma_2+\sigma_2')+(R_1^2+R_2')\log^2(L+n)\sigma_3
    \bigg]\,.\label{eq:bound_on_S1}
  \end{align}
If in addition \Cref{ass:bern2} holds, then with probability at least $1-O((L+n)^{-1})$,
\begin{align}
    \|S_2-\mathbb E S_2\|
    \le & C\bigg[
\sqrt{v_2}\log(L+n)(\sqrt{L}\sigma_2+\sigma_2')+R_2\log(L+n)\sigma_3
    \bigg]\,.\label{eq:bound_on_S2-ES2}
  \end{align}
and consequently,
\begin{align}
\|S-\mathbb E S\| \le & C\bigg[
    v_1n \log(L+n)\sigma_1+\sqrt{v_1}R_1\sqrt{Ln}\log^{3/2}(L+n)\sigma_2\nonumber\\
&\quad+\sqrt{v_2'+v_2}\log(L+n)(\sqrt{L}\sigma_2+\sigma_2')+(R_1^2+R_2+R_2')\log^2(L+n)\sigma_3
    \bigg]\,.\label{eq:bound_on_quad_dev}
  \end{align}
  \end{theorem}

The proof of \Cref{thm:quad_sum} is given in \Cref{app:1}, where the main effort is to control $\|S_1\|$.  Unlike the linear combination case, the complicated dependence caused by the quadratic form needs to be handled by viewing $S_1$ as a matrix-valued $U$-statistic indexed by the pairs $(i,j)$, and using a decoupling technique due to \cite{PenaM95}.  This reduces the problem of bounding $\|S_1\|$ to that of bounding $\|\sum_\ell X_\ell G_\ell \tilde X_\ell^T\|$, where $\tilde{X}_1,...,\tilde{X}_L$ are i.i.d. copies of $X_1,...,X_L$.

The upper bounds in \Cref{thm:quad_sum} look complicated. This is because we do not make any assumption about the Bernstein parameters or the matrices $G_\ell$.  The bound can be much simplified or even improved in certain important special cases.
 In the sub-Gaussian case, where $R_1\asymp v_1^{1/2}\asymp R_2^{1/2}\asymp v_2^{1/4}$, the first term $v_1 n \log (L+n)\sigma_1$ in \eqref{eq:bound_on_S1} dominates. This reflects the $L^{1/2}$ effect for sums of independent random variables. For example, in the case $G_\ell=G_0$ for all $\ell$ and $X_\ell$ are i.i.d., we have $\|\mathbb E S\|\approx L\|X_1 G_0 X_1^T\|\asymp v_1 n L \|G_0\|$, but when we consider the fluctuations contributed by $S_1$, we have $\|S_1\|\lesssim v_1 n  L^{1/2}\|G_0\|$ ignoring logarithmic factors.  In other words, the signal is contained in $\mathbb E S_2$ whose operator norm may grow linearly as $L$, while the fluctuation in the operator norm of $S_1$ only grows at a rate of $L^{1/2}$.

Additionally, in the Bernoulli case, the situation becomes more complicated when the variance $v_1$ is vanishing, meaning that  $v_1\asymp v_2\asymp (v')^{1/2}_2 \ll R_1\asymp R_2$. In the simple case of $G_\ell=I_n$, we have $\sigma_1=L^{1/2}$, $\sigma_2=\sigma_3=1$. Thus the second term $(v_1Ln)^{1/2}\sigma_2$ in \eqref{eq:bound_on_S1} may dominate the first term when $nv_1\ll 1$.  In this case, we also have $\sigma_2'=(Ln)^{1/2}$.  Therefore, it is also possible that the term $v_2^{1/2}\sigma_2$ in \eqref{eq:bound_on_S2-ES2} may be large.
It turns out that in such very sparse Bernoulli cases, the bound on the fluctuation term $\|S_1\|$ can be improved by a more refined and direct upper bound for $\|\sum_\ell X_\ell X_\ell^T\|=\|S\|$.  The details are presented in the next subsection.


\subsection{Sparse Bernoulli matrices}\label{sec:sparse_bernoulli}
In this section, we focus on the case where $G_\ell=I_n$ for all $\ell$, and the $X_\ell$'s are symmetric with centered Bernoulli entries whose probability parameters are bounded by $\rho$. Here, $\rho$ can be very small.
In this case, Assumptions \ref{ass:bern1}, \ref{ass:bern2} and \ref{ass:bern2p} hold with $v_1=v_2=2\rho$, $R_1=R_2=R_2'=1$, $v_2'=2\rho^2$, and the matrices $G_\ell$ satisfy $\sigma_1=L^{1/2}$, $\sigma_2=\sigma_3=1$, $\sigma_2'=(Ln)^{1/2}$.

Ignoring logarithmic factors, the first part of \Cref{thm:quad_sum} becomes
\begin{align*}
  \|S_1\|\lesssim C\left[L^{1/2}n\rho + (Ln\rho)^{1/2}+1\right]\,,
\end{align*}
where the second term $(Ln\rho)^{1/2}$ can be dominating when $n\rho$ is small and $Ln\rho$ is large.  This is suboptimal since intuitively we expect that the main variance term $L^{1/2}n\rho$ is the leading term as long as its value is large enough, which only requires $n\rho \gg L^{-1/2}$. 
To investigate the cause of this suboptimal bound, observe that $(Ln\rho)^{1/2}$
originates from the second term $R_1(v_1Ln)^{1/2}\sigma_2$ in \eqref{eq:bound_on_S1}.
Investigating the proof of \Cref{thm:quad_sum}, this term is derived by bounding $\sum_{\ell}\|H_\ell^T H_\ell\|$ by $\sum_\ell \|H_\ell\|^2$, which is suboptimal in this sparse Bernoulli case when applying the decoupling technique. The following result shows a sharper bound in this setting using a more refined argument.

\begin{theorem}\label{thm:s1_sparse_bernoulli}
  Assume $G_\ell=I_n$ for all $\ell \in \{1,...,L\}$ and $X_1,...,X_L$ are symmetric with centered Bernoulli entries whose parameters are bounded by $\rho$.  If $L^{1/2}n\rho\ge C_1\log^{1/2}(L+n)$ and $n\rho\le C_2$ for some constants $C_1$, $C_2$, then
  with probability at least $1-O((n+L)^{-1})$,
  \begin{equation}\label{eq:s1_sparse_bernoulli}
    \|S_1\|\le C L^{1/2}\rho n \log^{1/2}(L+n)
  \end{equation}
   for some constant $C$\,.
\end{theorem}

The proof of \Cref{thm:s1_sparse_bernoulli} is given in \Cref{app:2} where we modify our usage of the decoupling technique. At a high level, the decoupling technique reduces the problem to controlling the operator norm of $\tilde S=\sum_\ell X_\ell \tilde X_\ell^T$ where $\tilde X_\ell$ is an i.i.d. copy of $X_\ell$.  Instead of directly applying \Cref{lem:mat_prod} with $H_\ell=\tilde X_\ell$, we instead shift $\tilde X_\ell$ back to the original Bernoulli matrix by considering $\tilde S= \sum_\ell X_\ell\tilde A_\ell - \sum_\ell X_\ell P_\ell$, where $\tilde A_\ell$ is the original uncentered binary matrix and $P_\ell=\mathbb E \tilde A_\ell$.  Then \Cref{lem:mat_prod} is applied to $\sum_\ell X_\ell P_\ell$ and $\sum_\ell X_\ell \tilde A_\ell$ separately, where the entry-wise non-negativity of $\tilde A_\ell$ allows us to use the Perron--Frobenius theorem to obtain a sharper bound for $\|\sum_\ell\tilde A_\ell^2\|$.


\section{Further simulation study}\label{sec:simu}

In the following simulation study, we show that bias-adjusting sum of squared adjacency matrices constructed in \eqref{eq:S0} 
has a measurable impact on the downstream
spectral clustering accuracy, and that our method performs favorably against other competing methods. This builds upon the simulation initially shown in \Cref{sec:necessity}.

\paragraph{Data-generating process.} We design the following simulation setting to highlight the importance of
bias adjustment for $\sum_{\ell} A_{\ell}^2$. We consider $n=500$ nodes per network across $K=3$ communities, 
with imbalanced sizes $n_1= 200, ~ n_2 = 50,~  n_3 = 250$.
We construct two edge-probability matrices that share the same eigenvectors,
\begin{equation} \label{eq:eigenvector}
W= \begin{bmatrix}
1/2 & 1/2 & -\sqrt{2}/2\\
1/2 & 1/2 & \sqrt{2}/2\\
 \sqrt{2}/2 & - \sqrt{2}/2 & 0
\end{bmatrix}.
\end{equation}
The two edge-probability matrices are
\begin{align*} \label{eq:connectivity}
B^{(1)} &= W \begin{bmatrix}
1.5 & 0 & 0\\
0 & 0.2 & 0\\
0 & 0 & 0.4
\end{bmatrix}W^\top \approx
\begin{bmatrix}
0.62 & 0.22 & 0.46\\
0.22 & 0.62 & 0.46\\
0.46 & 0.46 & 0.85
\end{bmatrix} 
,\\
B^{(2)} &= W \begin{bmatrix}
1.5 & 0 & 0\\
0 & 0.2 & 0\\
0 & 0 & -0.4
\end{bmatrix}W^\top\approx
\begin{bmatrix}
0.22 & 0.62 & 0.46\\
0.62 & 0.22 & 0.46\\
0.46 & 0.46 & 0.85
\end{bmatrix}. \nonumber
\end{align*}
We then generate $L=100$ layers of adjacency matrices, where each layer is drawn by setting the edge-probability matrices $B_{\ell} = \rho B^{(1)}$ for $\ell \in \{1,...,L/2\}$ and $B_{\ell} = \rho B^{(2)}$ for $\ell \in \{L/2+1,...,L\}$.
Using this, we generate the adjacency matrices via \eqref{eq:gen_mlsbm}, with $\rho$ varying from $0.025$ to $0.2$.

We choose this particular simulation setting for two reasons. First, the first two eigenvectors in $W$ are not sufficient to distinguish between the first two communities.
Hence, methods based on $\sum_{\ell} A_{\ell}$ are not expected to perform well since the third eigen-component cancels out in the summation. Second, the average degrees among the three communities are drastically different, which are $251\rho$, $191\rho$ and $327\rho$ respectively. This means the variability of degree matrix $D_{\ell}$'s diagonal entries will be high, helping demonstrating the effect of our method's bias adjustment.

\paragraph{Methods we consider.} We consider the following four ways to aggregate information across all $L$ layers, three of which were used earlier in \Cref{fig:simu}: 1) the sum of adjacency matrices without squaring (i.e., considering $M = \sum_{\ell} A_{\ell}$, ``Sum''), 2) the sum of squared adjacency matrices (i.e., considering $M = \sum_{\ell} A^2_{\ell}$, ``SoS''), 3) our proposed bias-adjusted sum of squared adjacency matrices (i.e., considering \eqref{eq:S0}, or equivalently $M = \sum_{\ell} A^2_{\ell}$ and then zeroing out the diagonal entries, ``SoS-Debias''), and 4)
column-wise concatenating the adjacency matrices together, specifically, considering
 \[
M = \begin{bmatrix}
A_1 & A_2 & ... & A_L
\end{bmatrix} \in \mathbb{R}^{n \times (Ln)}.
\]
(i.e., ``Tensor matricization''). This method is commonly-used in the tensor literature (see \cite{zhang2018tensor} for example), where the $L$ adjacency matrices can be
viewed as a $n \times n \times L$ tensor, and the column-wise concatenation converts the tensor into a matrix.  
Then, using one of the four construction of the aggregated matrix $M$, we then apply spectral clustering onto $M$, meaning we first compute the matrix containing the leading $K$ left singular vectors of $M$ and perform K-means on its rows.

 Additionally,  we consider two methods that developed in \cite{PaulC17} called Linked Matrix Factorization (i.e., ``LMF'') and  Co-regularized Spectral Clustering (i.e., ``Co-reg''). These two methods fall outside the framework of the four methods discussed above. Instead, 
they use optimization procedures designed with different so-called fusion techniques to solve for an appropriate low-dimensional embedding shared among all $L$ layers, and then perform K-means clustering on its rows.


\begin{figure}[t]
  \centering
   \includegraphics[width=350px]{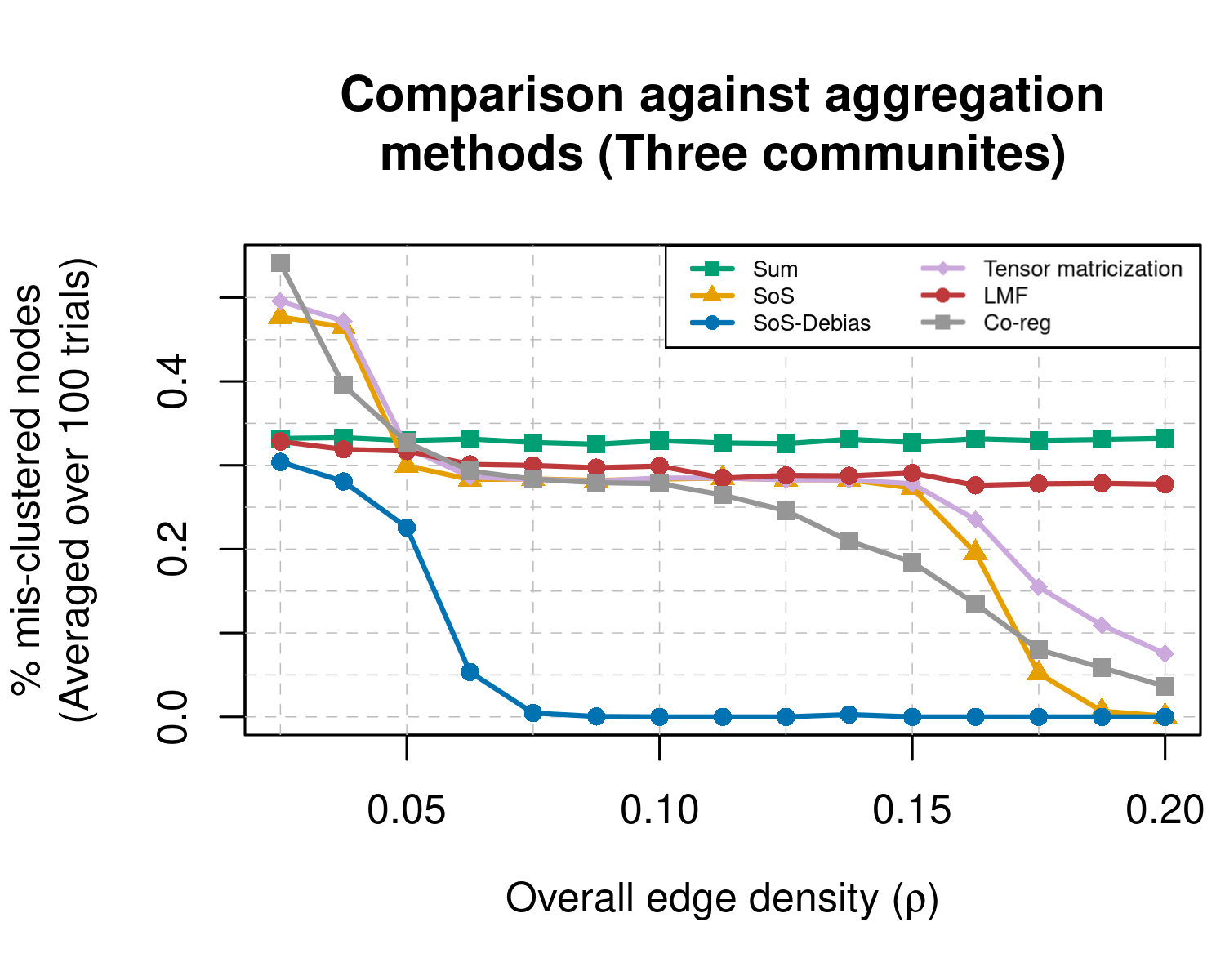} 
   \vspace{-2em}
    \caption{ \small 
   The average proportion of mis-clustered nodes for eight methods (measured via Hamming distance $n^{-1}d(\hat\theta,\theta)$ shown in \eqref{eq:hamming}, averaged over 100 trials), with $n=500$ with three unequally-sized communities among overall edge densities ranging from $\rho \in [0.025,0.2]$ and $L=100$ layers. Six methods' performance are shown: ``Sum'' (green squares), ``SoS'' (orange triangles), ``Bias-adjusted SoS'' (blue circles), ``Tensor matricization'' (purple diamonds), ``LMF'' (red circles), and ``Co-reg'' (gray squares).
   }
    \label{fig:simulation_rho}
\end{figure}

\paragraph{Results.} The results shown in \Cref{fig:simulation_rho} demonstrate that bias-adjusting the diagonal entries of $\sum_{\ell} A^2_{\ell}$ has a noticeable impact on the clustering accuracy.
Using the aforementioned simulation setting and methods, we vary $\rho$ from 0.025 to $0.2$ in 15
equally-spaced values, and compare the methods for each setting of $\rho$ across 100 trials by measuring the average Hamming distance (i.e., $n^{-1}d(\hat{\theta}, \theta)$ defined in \eqref{eq:hamming}) between the true memberships in $\theta$ and the estimated membership $\hat\theta$.
We observe phenomenons in \Cref{fig:simulation_rho} which all agree with our intuition and theoretical results.
Specifically, summing the adjacency matrices hinders our ability to
cluster the nodes due to the cancellation of positive and negative eigenvalues
(green squares), and the diagonal bias induced by squaring the adjacency matrices has a 
profound effect in the range of $\rho \in [0.08, 0.17]$, which our bias-adjusted sum-of-squared method removes (purple diamonds verses blue circles). We also
see that our bias-adjusted sum-of-squared method out-performs Linked Matrix Factorization (red circles) and Co-regularized Spectral Clustering (gray squares). 
While the LMF method and Co-reg method show some improvements over the Sum and SoS methods, respectively, they still behave qualitatively similar.  This observation suggests that these two methods may have similar difficulty in aggregating layers without positivity or removing the diagonal bias.

\paragraph{Intuition behind results.} We provide additional intuition behind the results shown in \Cref{fig:simulation_rho}
by visualizing the impact
of the diagonal terms on the overall spectrum and quantifying the loss of population signal due to the bias.

First, we demonstrate in \Cref{fig:simulation_bias} that the third leading eigenvalue of $\sum_{\ell} A_{\ell}^2$ when $\rho = 0.15$
is indistinguishable from the remaining bulk ``noise'' eigenvalues if the diagonal
bias is not removed (left), but becomes well-separated if so (middle). Recall by construction \eqref{eq:eigenvector}, all three eigenvectors are needed for recovering the communities. Hence, if the third eigenvalue of $\sum_{\ell} A_{\ell}^2$ is indistinguishable from fourth through last eigenvalues (i.e., the ``noise''),
then we should expect many nodes to be mis-clustered. This is exactly what
\Cref{fig:simulation_bias} (left) shows, where the third eigenvalue (denoted by the left-most red vertical line) is not separated from the remaining eigenvalues.
However, when we appropriately bias-adjust $\sum_{\ell} A_{\ell}^2$ via \eqref{eq:S0}, then \Cref{fig:simulation_bias} (middle) shows that the third eigenvalue is now well-separated from the remaining eigenvalues. This demonstrates the importance of bias-adjustment for community estimation in this regime of $\rho$.

\begin{figure}[t]
  \centering
   \includegraphics[width=\textwidth]{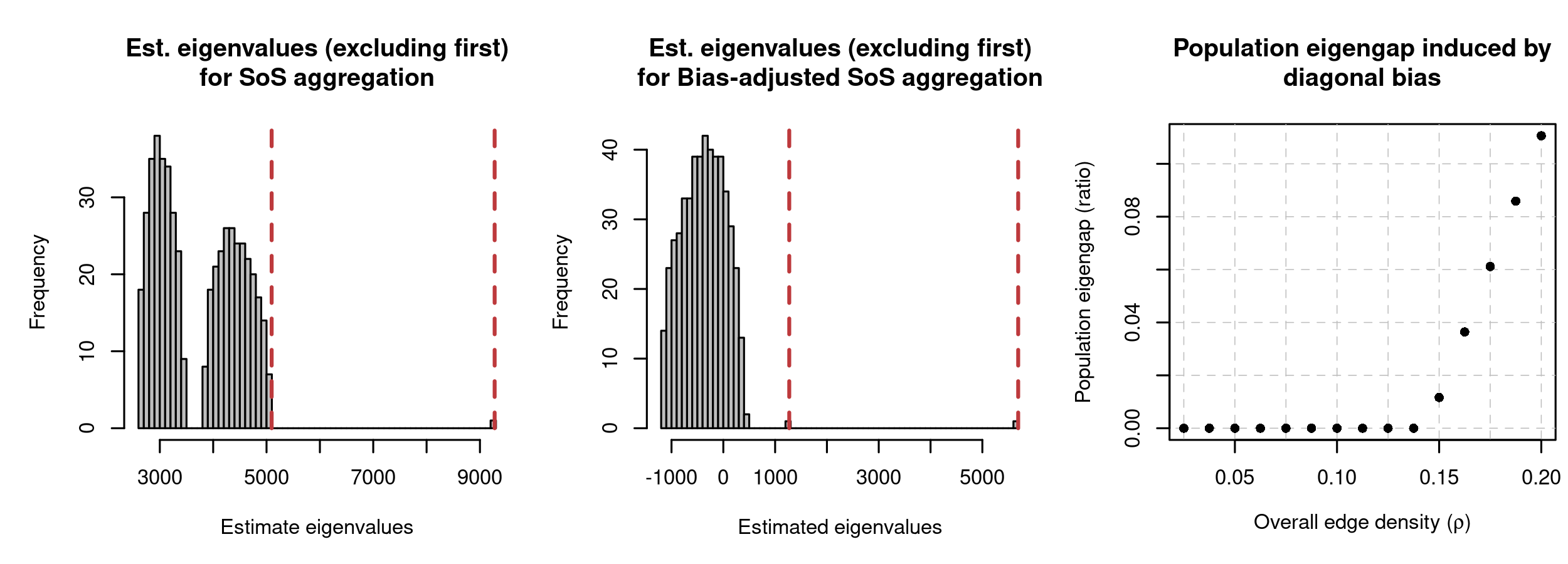}  
   \caption
   { \small (Left): For one realization of $A_1,...,A_L$ given the setup described in the simulation with $\rho=0.15$, a histogram of all 500 eigenvalues of $\sum_{\ell}A_{\ell}^2$, where the red vertical dashed lines denote the second and third eigenvalues. (The first eigenvalue is too large to be shown.) (Middle): Similar to the left plot, but showing the 500 eigenvalues of the bias-adjusted variant of $\sum_{\ell}A_{\ell}^2$ (i.e., setting the diagonal to be all 0's). (Right):  The population eigengap $(\lambda_3-\lambda_4)/\lambda_3$ computed from $\sum_{\ell}P_{\ell}^2+\tilde{D}_{\ell}$ for varying values of $\rho$. }
    \label{fig:simulation_bias}
\end{figure}

Next, in \Cref{fig:simulation_bias} (right), we show that this lack-of-separation
between the third eigenvalue and the noise can be observed on the population
level. Specifically, we show that the population counterpart of $\sum_{\ell} A_{\ell}^2$ has considerable diagonal bias that makes the accurate estimation of the third eigenvector nearly impossible when $\rho$ is too small. 
To show this, for a particular value of $\rho$, recall from our theory that the population counterpart of $\sum_{\ell} A_{\ell}^2$ is
\[
\sum_{\ell=1}^{L} (P_{\ell}^2 + \tilde{D}_\ell), \text{ for a diagonal matrix }\tilde{D}_{\ell}  \text{ where } 
\tilde{D}_{\ell,ii} = \sum_{j=1}^{n}P_{\ell,ij} \text{ for } 1 \leq i \leq n,
\]
and $P_{\ell} = \mathbb{E}A_{\ell}$. Let $\lambda_1,...,\lambda_n$ denote the $n$ eigenvalues of the above matrix, dependent on $\rho$. We then plot $(\lambda_3-\lambda_4)/\lambda_4$ against $\rho$ in \Cref{fig:simulation_bias} (right).
This plot demonstrates that when $\rho$ is too small, the diagonal entries (represented by $\tilde{D}_\ell$'s) add
a disproportionally large amount of bias that makes it impossible to accurately distinguish between the third and fourth eigenvectors.
Additionally, the raise in the eigengap in \Cref{fig:simulation_bias} (right) at $\rho = 0.15$ corresponds to when
``SoS'' starts to improve in \Cref{fig:simulation_rho} (orange triangles).
This means starting at $\rho = 0.15$, the effect of the diagonal bias starts to diminish, and at larger values of $\rho$, the sum of squared adjacency matrices contains accurate information for community estimation (both with and without bias adjustment). We report additional results in \Cref{app:additional_simu}, where we report the time needed for each method, visualize the lack of concentration in the nodes' degrees in sparse graphs and its effect on the spectral embedding, and also 
report that the qualitative trends in \Cref{fig:simulation_rho} remain the same when we either consider the varying-membership setting (described in \Cref{cor:varying_membership}) or an additional variant of spectral clustering where the eigenvectors are reweighted by its corresponding eigenvalues.

\section{Data application: Gene co-expression patterns in developing monkey brain}\label{sec:data}

We analyze the microarray dataset of developing rhesus monkeys' tissue from the medial prefrontal cortex introduced in \Cref{sec:intro} that was originally
collected in \cite{bakken2016comprehensive} to demonstrate the utility of our bias-adjusted sum-of-squared spectral clustering method.
As described in other work that analyze this data (\cite{LiuCXR18} and \cite{LeiCL19}), this is a suitable
dataset to analyze as other work have well-documented that the gene co-expression patterns in monkeys' tissue from this brain region
change dramatically over development.
Specifically, the data from \cite{bakken2016comprehensive} consists of the gene co-expression network of ten different developmental times (starting from 40 days in the embryo to 48 months after birth) derived from microarray data, where each of the developmental time points corresponds to post-mortem tissue samples of multiple unique rhesus monkeys. With this data, we aim to show 
that our bias-adjusted sum-of-squared spectral clustering method produces insightful gene communities.


\paragraph{Preprocessing procedure.}
The microarray dataset from \cite{bakken2016comprehensive} contains $n=9173$ genes measured among many samples across the $L=10$ layers, which we preprocess into ten adjacency matrices in the following way in line with other work like \cite{langfelder2008wgcna}. First, for each layer $\ell \in\{1,...,L\}$, we
construct the Pearson 
correlation matrix. 
Then, we convert each correlation 
matrix into adjacency matrix by hard-thresholding at $0.72$ in absolute value, resulting in ten adjacency matrices $A_1, ..., A_{L}$. We choose this particular threshold since it yields sparse and scale-free networks that have many disjoint connected components individually but have one connected component after aggregation, as reported in \Cref{app:additional_data}. Lastly, we remove all the genes corresponding to nodes whose total degree 
across all ten layers is less than 90. This value is chosen
since the median total degree among all nodes that do not have any neighbors in five or more of the layers (i.e., a degree of zero in more than half the layers) is 89.  In the end, we have ten adjacency matrices $A_1, ..., A_{L} \in \{0,1\}^{7836 \times 7836}$, each representing a network corresponding to 7836 genes. 
We note that the above procedure of transforming correlation matrices into adjacency matrices is unlikely to procedure networks that severely violate the layer-wise positivity assumption commonly required by other methods -- this hypothetically could happen
if many pairs of genes display high negative correlations, but this is not typical in genomic data. Nonetheless, we are interested in what insights the bias-adjusted sum-of-squared spectral clustering method can reveal for this dataset.

\paragraph{Results and interpretation.} The following results show that bias-adjusted sum-of-squared spectral clustering
finds meaningful gene communities. Prior to using our method, we select the dimensionality and number of communities to be $K=8$ based on a scree plot of the singular values of the bias-adjusted variant of $\sum_{\ell}A^2_{\ell}$.
%
We perform our bias-adjusted spectral clustering on this matrix with $K=8$, and visualize
three out of the ten adjacency matrices using the estimated communities in \Cref{fig:pnas_adj} (which are the full adjacency matrices corresponding to the
three adjacency matrices shown in \Cref{fig:postred}). 
We see that as development occurs from 40 days in the embryo to 48 months after birth, there are different gene communities that are most-connected. This visually demonstrates different biological processes in brain tissue that are
most active at different stages of development. Labeling the communities 1 through 8 from top left to  bottom right,
our results show that starting at 40 days in the embryo, Community 1 is highly coordinated (i.e., densely connected),
and ending at 48 months after birth, Community 7 is highly coordinated. 
All the genes in Community 8 are sparsely connected
throughout all ten adjacency matrices, suggesting that these genes 
are not strongly correlated with
many other genes throughout development. 

\begin{figure}[t]
  \centering
   \includegraphics[width=0.32\textwidth]{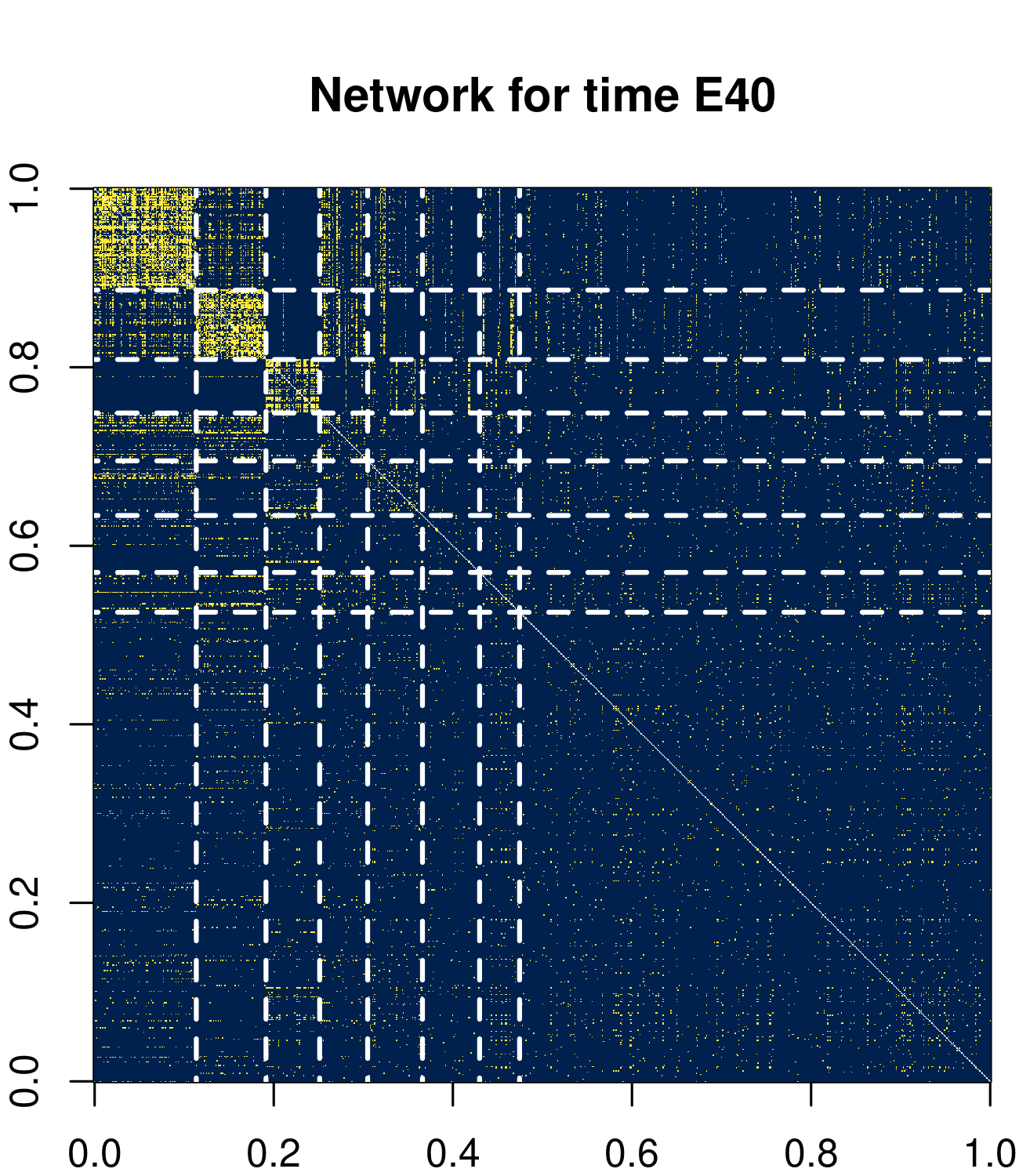}  
      \includegraphics[width=0.32\textwidth]{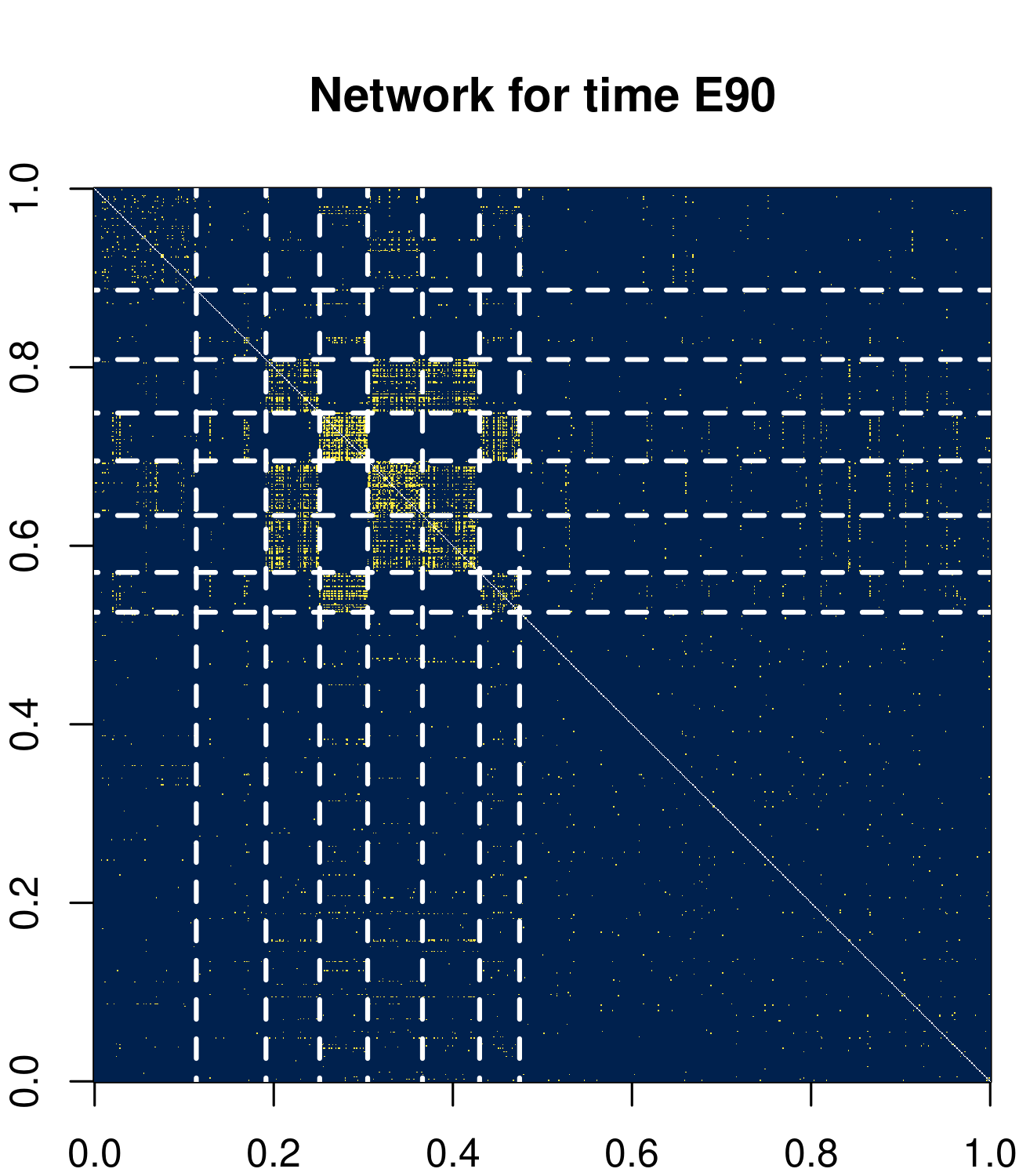}
         \includegraphics[width=0.32\textwidth]{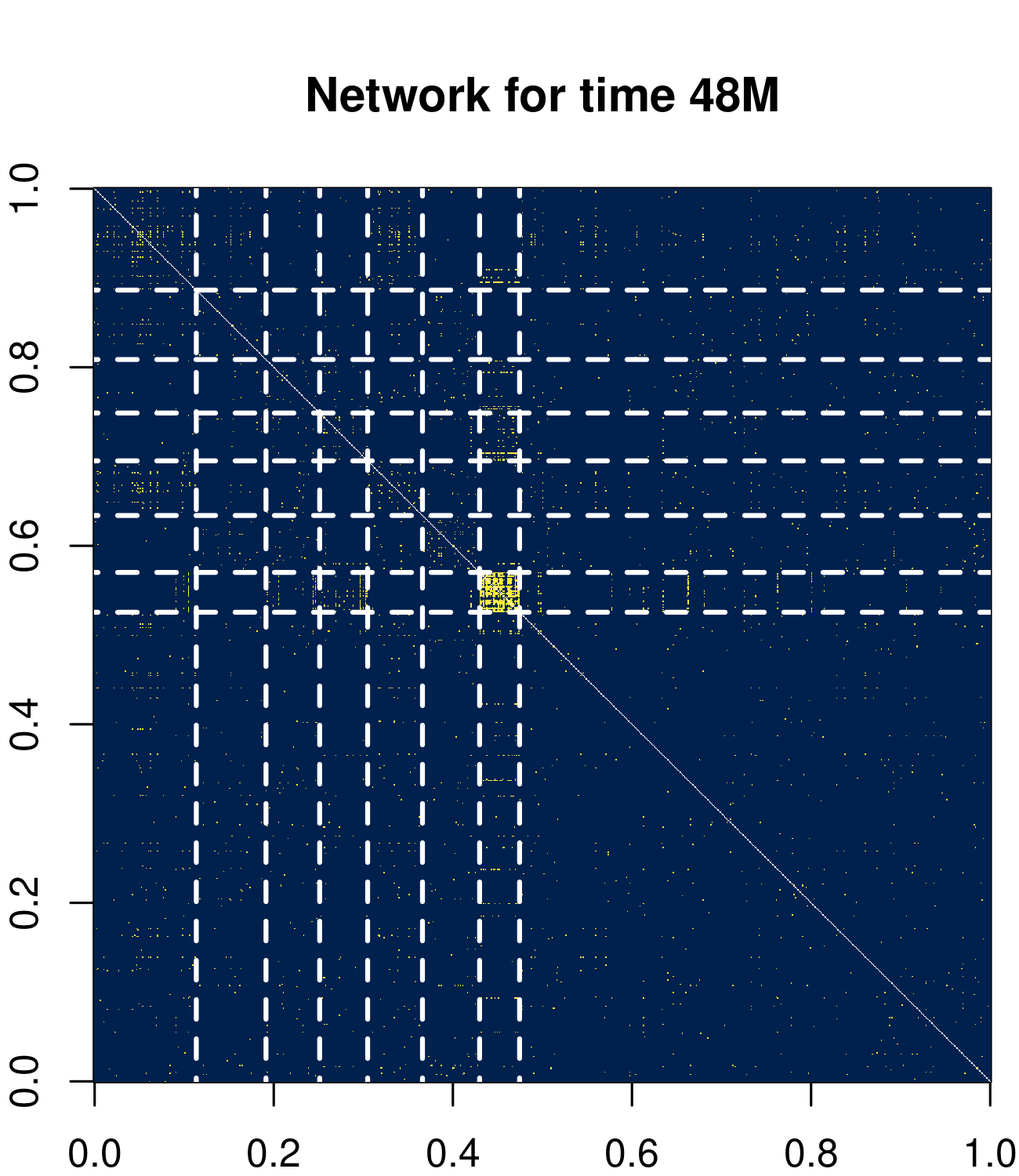}   \caption
   { \small Three of ten adjacency matrices where the genes are ordered according to the estimated $K=8$ communities. Blue pixels correspond to the absence of an edge between the corresponding genes in $A_{\ell}$'s, while yellow pixels correspond to
   an edge. The dashed white lines denote the separation among the $K=8$ gene communities. The adjacency matrices shown in \Cref{fig:postred} correspond to the same three developmental times 
  (from left to right), and are formed by selecting only the genes in Communities 1, 4, 5, and 7.
    }
    \label{fig:pnas_adj}
\end{figure}

To interpret these $K=8$ communities, we perform a gene ontology analysis, using the \texttt{cluster}- \texttt{Profiler::enrichGO} function on the gene annotation in the Bioconductor
package \texttt{org.Mmu.eg.db} to analyze the scientific
interpretation of each of the $K$ communities of genes within rhesus monkeys.  \Cref{tab:go} shows the results. We see the first seven communities are highly enriched for cell processes closely related to
brain development -- we can interpret \Cref{fig:pnas_adj} and \Cref{tab:go} together as which biological systems are most active in a coordinated fashion at different developmental stage. 
Since genes in the eighth community are sparsely-connected across all developmental time and is not enriched for any cell processes, we infer that these genes are unlikely to be coordinated to drive any process related to brain development.
Together, these results demonstrate that the bias-adjusted sum-of-squared spectral clustering is able to find meaningful gene communities.
Visualizations of all ten adjacency matrices, beyond those shown in \Cref{fig:pnas_adj} and explicit reporting of the edge densities,
as well as stability analyses that demonstrate how the results vary when different tuning parameters are used, are included in \Cref{app:additional_data}.

\begin{table}[t]
\centering
 \begin{tabular}{|c | c c c|} 
 \hline
Community & Description & GO ID  & p-value \\ [0.5ex] 
 \hline\hline
1 & RNA splicing & GO:0008380 & $1.07\times 10^{-11}$  \\ 
 2 & Nuclear transport & GO:0051169 & $3.15\times 10^{-5}$ \\
 3 & Neuron development & GO:0048666 & $2.08\times 10^{-8}$  \\
 4 & Chromosome segregation & GO:0007059 & $1.31 \times 10^{-8}$  \\
 5 & Neuron projection development & GO:0031175 & $1.51 \times 10^{-5}$ \\ 
  6 & Regulation of transporter activity & GO:0032409 & $5.68\times 10^{-6}$ \\ 
   7 & Anchoring junction & GO:0070161 & $8.86 \times 10^{-5}$ \\ 
     8 & None &  &  \\ 
 [1ex] 
 \hline
\end{tabular}
 \caption{\small \label{tab:go} Gene ontology of the estimated $K=8$ communities of genes. Here, ``GO'' denotes the gene ontology ID, and ``p-value''
denotes the Fisher's exact test to denote an enrichment (i.e., significance or over-representation) of a particular GO for the genes in said community compared to all other genes.} 
\end{table}

\section{Discussion}\label{sec:disc}
While we establish community estimation consistency in this paper,
there are two major additional theoretical directions we hope our results will help
shed light into for future work.
First, an important theoretical question in the study of stochastic block models is the critical threshold for community estimation.  This involves finding a critical rate of the overall edge density and/or the separation between rows of $B_{\ell,0}$, and proving achievability of certain community estimation accuracy when the density and/or separation are above this threshold, as well as impossibility for non-trivial community recovery below this threshold. For single-layer SBMs, this problem has been studied by many authors, such as \cite{Massoulie14}, \cite{AbbeS15}, \cite{ZhangZ16}, and \cite{Mossel18}.  The case of multi-layer SBMs is much less clear, especially for generally structured layers.  The upper bounds proved in \cite{PaulC17} and \cite{Sharmo18} imply achievability of vanishing error proportion when $Ln\rho\rightarrow\infty$ under a layer-wise positivity assumption.  Our results requires a stronger $L^{1/2}n\rho/\log^{1/2}(L+n)\rightarrow\infty$ condition, but does not require a layer-wise positivity assumption. Ignoring logarithmic factors, is a rate of $L^{1/2}$ the right price to pay for not having the layer-wise positivity assumption?  The error analysis in the proof of \Cref{thm:consistency} seems to suggest a positive answer, but a rigorous claim will require a formal lower bound analysis. We note that the simplified constructions such as that in \cite{ZhangZ16} designed for
single-layer SBMs are unlikely to work, since they do not reflect the additional hardness brought to the estimation problem by unknown layer-wise structures.

Second, the consistency result for multi-layer SBMs also makes it possible to extend other inference tools developed for single-layer data to multi-layer data.  One such example is model selection and cross-validation \citep{ChenL18,LiLZ16}.
The probability tools developed in this paper, such as \Cref{lem:mat_prod,thm:quad_sum,thm:asym_quad}, may be useful for other statistical inference problems involving matrix-valued measurements and noise.  For example, our theoretical analyses could  refine the theoretical analyses for multilayer graphs that go beyond SBMs, such as degree-corrected SBMs or random dot-product graphs in general \citep{nielsen2018multiple,arroyo2019inference}.
Alternatively, in dynamic networks where the network parameters change smoothly over time, one may use nonparametric kernel smoothing techniques in \cite{PenskyZ19} and the matrix concentration inequalities developed in this paper to control the aggregated noise and perhaps obtain more refined analysis in those settings.

\bibliographystyle{asa}
\bibliography{SumSqSpec}

\newpage

\appendix

\section{Proofs for general concentration results}\label{app:1}

\textbf{Notation.}
For a matrix $M$, let $M_{j\cdot}$ denote its $j$th row in the form of a column vector. Also, we define $\|M\|_{q,\infty}=\max_j\|M_{j\cdot}\|_q$ for $q\in [1,\infty)$, and $\|M\|_\infty$ is the maximum entry-wise absolute value. When $M$ is symmetric with eigen-decomposition $\sum_j \lambda_j u_j u_j^T$, let $|M|=\sum_j |\lambda_j| u_j u_j^T$. Let $e_i$ be the $i$-th coordinate unit vector, the length of $e_i$ will depend on the context. For two symmetric matrices $A$ and $B$, $A\preceq B$ means that $B-A$ is positive semidefinite.  In the statement of the theorems and their proofs, we use $C$ to denote a universal constant whose value may vary from line to line but does not depend on any of the model parameters. 
Throughout this entire paper, we reserve $i,j \in \{1,\ldots,n\}$ as indices for individual nodes, while we reserve $\ell \in \{1,\ldots,L\}$ as the index for individual layers. For a square matrix $A$, let $A_{ij}^2$ and $[A^2]_{ij}$ refer to the square of the $(i,j)$-entry of $A$ and $A^2$ respectively. For two random sequences $X_n$ and $Y_n$, we write $X_n = O_P(Y_n)$ and $X_n=o_P(Y_n)$ to denote $X_n/Y_n$ is asymptotically bounded in probability or converging to 0 in probability respectively. Let $I_K$ denote an identity matrix of size $K$.

\subsection{Rank-one dilation}
Our proof of \Cref{lem:mat_prod} uses a rank-one dilation trick to handle the asymmetry in $X_\ell H_\ell$.

\begin{definition}[Symmetric dilation]
  For an $n\times m$ matrix $A$, the symmetric dilation of $A$, denoted by $\mathcal D(A)$, is the $(n+m)\times (n+m)$ symmetric matrix
  $$
  \mathcal D(A)=\left[\begin{array}{cc}
    0 & A \\ A^T & 0
  \end{array}\right]\,.
  $$
\end{definition}

The symmetric dilation is a convenient tool to reduce singular values and singular vectors of asymmetric matrices to eigenvalues and eigenvectors of symmetric matrices.  See Exercise II.1.15 of \cite{Bhatia} and Section 2.6 of \cite{Tropp12} for example.  Here, we will use a special case of rank-one dilations whose proof is elementary and omitted.
\begin{lemma}[Rank-one dilation]\label{lem:rank1_dil}
  For two column vectors $e$ and $a$, $\mathcal D(ea^T)$ has eigen-decomposition
  $$
  \mathcal D(e a^T)=\|e\|\|a\|\left(\frac{1}{\sqrt{2}}\left[\begin{array}{cc}
    \frac{e}{\|e\|}\\
    \frac{a}{\|a\|}
  \end{array}\right]\frac{1}{\sqrt{2}}\left[\begin{array}{cc}
    \frac{e}{\|e\|}\\
    \frac{a}{\|a\|}
  \end{array}\right]^T-\frac{1}{\sqrt{2}}\left[\begin{array}{cc}
    \frac{e}{\|e\|}\\
    -\frac{a}{\|a\|}
  \end{array}\right]\frac{1}{\sqrt{2}}\left[\begin{array}{cc}
    \frac{e}{\|e\|}\\
    -\frac{a}{\|a\|}
  \end{array}\right]^T\right)
  $$
and for each integer $k\ge 2$
$$
\left|\mathcal D(e a^T)^k\right|\preceq \|e\|^{k-2}\|a\|^{k-2} \left[\begin{array}{cc}
  \|a\|^2ee^T & 0\\ 0  & \|e\|^2 a a^T
\end{array}\right]\,.
$$
\end{lemma}

\subsection{Proof of \Cref{lem:mat_prod}}
\begin{proof}[Proof of \Cref{lem:mat_prod}]
  We will prove the asymmetric case first. The symmetric case follows by consider upper and lower diagonal of $X_\ell$ separately and use union bound.
  
  First consider the case of a single pair of $(X,H)$, where $X$ is $n\times r$ with independent $(v_1,R_1)$-Bernstein entries, and $H$ is $r\times m$.
  Then
  \begin{align*}
    XH = \sum_{1\le i\le n,1\le j\le r} X_{ij} e_i H_{j\cdot}^T\,.
  \end{align*}
By \Cref{lem:rank1_dil} we have
\begin{align*}
\left|\mathbb E\left[\mathcal D(X_{ij}e_i H_{j\cdot}^T)\right]^k\right|\preceq & \mathbb E|X_{ij}|^k \left|\left[\mathcal D(e_i H_{j\cdot}^T)\right]^k\right|\\
\preceq& (v_1/2) k! R_1^{k-2} \|H_{j\cdot}\|^{k-2} \left[\begin{array}{cc}
  \|H_{j\cdot}\|^2 e_i e_i^T & 0 \\ 0 & H_{j\cdot} H_{j\cdot}^T
\end{array}\right]
\end{align*}
Now take the sum over $i \in \{1,\ldots,n\}$ and $j \in \{1,\ldots,r\}$.
\begin{align}\label{eq:dilation_ineq_gen}
\sum_{1\le i\le n, 1\le j\le r} \left[\begin{array}{cc}
  \|H_{j\cdot}\|^2 e_i e_i^T & 0 \\ 0 & H_{j\cdot} H_{j\cdot}^T
\end{array}\right]
= \left[\begin{array}{cc}\|H\|_F^2 I_n & 0 \\ 0 & n H^T H \end{array}\right]\,.
\end{align}
By Theorem 6.2 of \cite{Tropp12}, 
we have
\begin{align*}
  \mathbb P\big[\left\|\mathcal D\left(XH\right)\right\|\ge t\big]
  \le 2(m+n)\exp\left(-\frac{t^2/2}{v_1 (n \|H^T H\|\vee \|H\|_F^2)+R_1\|H\|_{2,\infty}t}\right)\,.
\end{align*}
 The proof for the case of sum $\sum_\ell X_\ell H_\ell$ follows by modifying the above argument where the summation in \eqref{eq:dilation_ineq_gen} takes another outer layer of summation over $\ell$ and becomes
$$
\sum_{\ell=1}^{L}\sum_{1\le i\le n, 1\le j\le r} \left[\begin{array}{cc}
  \|H_{j\cdot}\|^2 e_i e_i^T & 0 \\ 0 & H_{j\cdot} H_{j\cdot}^T
\end{array}\right]
\preceq \left[\begin{array}{cc}\sum_\ell\|H_\ell\|_F^2 I_n & 0 \\ 0 & n \sum_\ell H_\ell^T H_\ell \end{array}\right]\,.
$$

To prove the result for the symmetric case, let $X_\ell^{(u)}$ be the diagonal and upper-diagonal part of $X_\ell$, and $X_{\ell}^{(l)}=X_\ell-X_\ell^{(u)}$. The claim follows by upper bounding $\mathbb P(\|\sum_\ell X_\ell^{(u)} H_\ell\|\ge t/2)$ and $\mathbb P(\|\sum_\ell X_\ell^{(l)} H_\ell\|\ge t/2)$ using the asymmetric result, and combining with union bound.
\end{proof}

\begin{proof}[Proof of \Cref{thm:quad_sum}]
The proof uses decoupling. Let  $\tilde X_\ell$ be an independent copy of $X_\ell$.  Define
$$
\tilde S = \sum_{\ell=1}^{L} X_\ell G_\ell \tilde X_\ell^T\,,
$$
\begin{align*}
\tilde S_2=&\Big[\sum_{\ell=1}^{L} \sum_{1\le i < j\le n} X_{\ell,ij}\tilde X_{\ell,ij}\left(e_ie_i^TG_{\ell,jj}+e_je_j^T G_{\ell,ii}+e_ie_j^TG_{\ell,ji}+e_je_i^T G_{\ell,ij}\right)\Big]\\
&+\Big[\sum_{\ell=1}^{L}\sum_{1\le i\le n}X_{\ell,ii}\tilde X_{\ell,ii} e_i e_i^T G_{\ell,ii}\Big]\,,
\end{align*}
and
$$
\tilde S_1=\tilde S-\tilde S_2\,.
$$
Now we expand $S_1$:
\begin{align*}
  S_1= & S-S_2 \\
  =&
  \Big[\sum_{\ell=1}^{L}\sum_{\substack{1\le i < j\le n \\ 1\le i'<j'\le n \\ (i,j)\neq (i',j')}}
   X_{\ell,ij}X_{\ell,i'j'}
   \left(e_i e_{i'}^T G_{\ell,jj'}+e_je_{j'}^TG_{\ell,ii'}+e_ie_{j'}^TG_{\ell,ji'}+e_je_{i'}^TG_{\ell,ij'}\right)\Big]\\
    &+ \Big[\sum_{\ell=1}^{L}\sum_{\substack{1\le i\le n \\ 1\le i' < j'\le n}} X_{\ell,ii} X_{\ell,i'j'}(e_i e_{i'}^T G_{\ell,ij'}+e_i e_{j'}^T G_{\ell,ii'})\Big]\\
    &+\Big[\sum_{\ell=1}^{L}\sum_{\substack{1\le i<j\le n \\ 1\le i'\le n}} X_{\ell,ij} X_{\ell,i'i'} (e_i e_{i'}^T G_{\ell,ji'}+e_j e_{i'}^T G_{\ell,ii'})\Big]\\
    &+ \Big[\sum_{\ell=1}^{L}\sum_{1\le i\neq i'\le n} X_{\ell,ii} X_{\ell,i'i'} e_i e_{i'} G_{\ell,ii'}\Big]\,,
\end{align*}
which can be viewed as a matrix-valued U-statistic defined on the vectors $X_{1,ij}, \ldots,X_{L,ij}$ indexed by pairs $(i,j)$ such that $1\le i < j\le n$.  Using the decoupling inequality (Theorem 1 of \cite{PenaM95}) we have
\begin{equation}\label{eq:decoupling}
\mathbb P(\|S_1\|\ge t)\le C_2 \mathbb P(\|\tilde S_1\|\ge t/C_2)  
\end{equation}
for some universal constant $C_2$ and all $t>0$.

The plan is to control $\|\tilde S_1\|$ by
$\|\tilde S_1\|\le \|\tilde S\| + \|\tilde S_2\|$, where we analyze each of the left-hand terms separately in the following two steps.

\paragraph{Step 1: Controlling $\tilde S$.}

Let $H_\ell=G_\ell\tilde X_\ell$.
In order to apply \Cref{lem:mat_prod} to control $\tilde S_2=\sum_\ell X_\ell H_\ell$ conditioning on $H_\ell$, we need to upper bound $$\left\|\sum_\ell H_\ell^T H_\ell\right\|=\left\|\sum_\ell \tilde X_\ell G_\ell^T G_\ell\tilde X_\ell^T\right\|$$ and $$\max_\ell \|H_\ell\|_{2,\infty}=\max_{\ell,j} \|\tilde X_\ell G_{\ell,j\cdot}\|\,.$$

We first consider $\left\|\sum_\ell H_\ell^T H_\ell\right\|$.
With high probability over $\tilde X_\ell$, we have
\begin{align}
  \left\|\sum_{\ell=1}^{L} H_\ell^T H_\ell\right\|=&\left\|\sum_{\ell=1}^{L} \tilde X_\ell G_\ell^T G_\ell\tilde X_\ell^T\right\|\nonumber\\
  \le & \sum_{\ell=1}^{L} \|\tilde X_\ell G_\ell^T G_\ell \tilde X_\ell^T\|\nonumber\\
  \le & \sum_{\ell=1}^{L} \|\tilde X_\ell G_\ell^T\|^2\nonumber\\
  \lesssim & \sum_{\ell=1}^{L}\left[v_1 n \log (L+n) \|G_\ell\|^2 + R_1^2 \|G_\ell^T\|_{2,\infty}^2\log^2(L+n)\right]\nonumber\\
  = & v_1 n \log (L+n)\sum_{\ell=1}^{L} \|G_\ell\|^2 + R_1^2\log^2(L+n)\sum_{\ell=1}^{L}\|G_\ell^T\|_{2,\infty}^2\nonumber\\
  \le & v_1 n  \log(L+n)\sigma_1^2+R_1^2 L \log^2(L+n)\sigma_2^2\,,\label{eq:HTH}
\end{align}
where the fourth line follows from applying \Cref{lem:mat_prod} to each individual $\tilde X_\ell G_\ell^T$ with union bound over $\ell$ and the fact that the entries of $\tilde X_\ell$ are $(v_1,R_1)$-Bernstein.
  
Now we turn to $\max_\ell \|H_\ell\|_{2,\infty}$.  
Applying \Cref{lem:mat_prod} to $\tilde X_\ell G_{\ell,j\cdot}$ and takinga  union bound over $j\in\{1,\ldots,n\}$ and $\ell \in \{1,\ldots,L\}$ we get with high probability,
\begin{align*}
\max_{\ell,j} \|H_{\ell}\|_{2,\infty} \lesssim  & \sqrt{v_1}\sqrt{n\log (L+n)} \max_\ell\|G_\ell\|_{2,\infty}+R_1\max_\ell\|G_\ell\|_\infty \log (L+n)\\
\lesssim &  \sqrt{v_1}\sqrt{n\log (L+n)} \sigma_2+R_1\log (L+n)\sigma_3\,.  
\end{align*}

Intersecting on these two events above and applying \Cref{lem:mat_prod}, we conclude with high probability
\begin{align}
\|\tilde S\|\lesssim &  \sqrt{v_1}\sqrt{n \log (L+n)}\left[v_1 n \log (L+n)\sigma_1^2 + R_1^2\log^2(L+n)L \sigma_2^2\right]^{1/2}\nonumber\\
&+ R_1\log (L+n)\left[\sqrt{v_1}\sqrt{n\log (L+n)} \sigma_2+R_1 \log (L+n)\sigma_3\right]\nonumber\\
\lesssim & v_1 n  \log (L+n)\sigma_1 + \sqrt{v_1} R_1 \sqrt{n}\sqrt{L}\log^{3/2}(L+n)\sigma_2+R_1^2 \log^2(L+n)\sigma_3\,.\label{eq:tilde_s_bound_proof}
\end{align}

\paragraph{Step 2: Controlling $\tilde S_2$.}

Let $Z_{\ell,ij}=X_{\ell,ij}\tilde X_{\ell,ij}$. 
By construction, the off-diagonal part of $\tilde S_2$ is
$$
\sum_{\ell=1}^{L}\sum_{1\le i<j\le n} Z_{\ell,ij}(e_i e_j^T G_{\ell,ji}+e_j e_i^T G_{\ell,ij})\,.
$$

Consider the first component $\sum_{\ell}\sum_{1\le i<j\le n} Z_{\ell,ij}G_{\ell,ji}e_i e_j^T$.  \Cref{lem:rank1_dil} implies that
$$
\mathbb E\left[\left|\mathcal D(Z_{\ell,ij}G_{\ell,ji}e_i e_j^T)\right|^k\right]
\preceq \frac{v_2 G_{\ell,ji}^2}{2} |R_2 G_{\ell,ji}|^{k-2}\left[\begin{array}{cc}
  e_i e_i^T & 0 \\ 0 & e_j e_j^T
\end{array}\right]
$$
provided that $Z_{\ell,ij}$'s are $(v_2',R_2')$-Bernstein.
Summing over $\ell \in \{1,\ldots,L\}$ and $1\leq i < j \leq n$, we obtain
\begin{align*}
\sum_{\ell,i<j} \mathbb E\left[\left|\mathcal D(Z_{\ell,ij}G_{\ell,ji}e_i e_j^T)\right|^k\right]
\preceq & \frac{v_2'}{2} |R_2' \max_{\ell,ij}G_{\ell,ji}|^{k-2}\max\left\{\max_i \sum_{\ell,j}G_{\ell,ji}^2,\max_{j}\sum_{\ell,i}G_{\ell,ji}^2\right\}I_{2n}\\
\preceq &\frac{v_2'}{2} |R_2' \sigma_3|^{k-2}L\sigma_2^2 I_{2n}\,.
\end{align*}
Then with high probability the off-diagonal part of $\tilde S_2$ is bounded by
\begin{equation}\label{eq:tilde_s2_main}
\sqrt{v_2'L\log(L+n)}\sigma_2+R_2'\log (L+n) \sigma_3\,.
\end{equation}

For the diagonal part of $\tilde S_2$, the $i$th diagonal entry is
$$
\sum_{\ell=1}^{L}\sum_{j=1}^{n} Z_{\ell,ij} G_{\ell,jj}.
$$
Then the operator norm of diagonal part of $\tilde S_2$ is bounded by its maximum entry, which is further bounded by, using standard Bernstein's inequality
\begin{align}
&\sqrt{v_2'\log (L+n)}\left(\sum_{\ell,j}G_{\ell,jj}^2\right)^{1/2}+R_2'\max_{\ell,j}|G_{\ell,jj}|\log (L+n)\nonumber\\
\le&\sqrt{v_2'\log(L+n)}\sigma_2'+R_2'\log(L+n)\sigma_3\,.\label{eq:tilde_s2_diag}
\end{align}

Now \eqref{eq:bound_on_S1} follows by combining \eqref{eq:tilde_s_bound_proof}, \eqref{eq:tilde_s2_main}, and \eqref{eq:tilde_s2_diag} together with the decoupling inequality \eqref{eq:decoupling}.

The claim regarding $S-\mathbb E S$ only requires an additional bound on
$\|S_2-\mathbb E S_2\|$, which can be obtained using an identical argument to that of $\tilde S_2$ with $(v_2, R_2)$ replacing $(v_2', R_2')$.
\end{proof}

\subsection{Matrix quadratic forms: The asymmetric case}\label{sec:asymmetric}
Let $X_1,\ldots,X_L$ be independent $n\times r$ matrices with independent zero mean entries. Let $G_\ell$ be each an $r\times r$ matrix.
The decomposition of the quadratic form now becomes simpler,
\begin{align*}
  S=\sum_{\ell=1}^{L} X_\ell G_\ell X_\ell^T=S_1+S_2\,,
\end{align*}
where
$$
S_1=\sum_{\ell=1}^{L}\sum_{(i,j)\neq (i',j')}X_{\ell,ij}X_{\ell,i'j'} e_i e_{i'}^T G_{\ell,jj'}
$$
is the mean-zero off-diagonal part and
$$
S_2=\sum_{\ell=1}^{L}\sum_{1\le i\le n,1\le j\le r} X_{\ell,ij}^2 e_i e_i^T G_{\ell,jj}\,,
$$
is the diagonal part with possibly non-zero expected values on the diagonal entries.

Define
$$
\sigma_1'=\left(\sum_{\ell=1}^{L} \|G_\ell\|_F^2\right)^{1/2}\,.
$$

\begin{theorem}\label{thm:asym_quad}
If $X_1,\ldots,X_L$ are $n\times r$ independent matrices with independent entries satisfying \Cref{ass:bern1} and \Cref{ass:bern2p}, then with probability at least $1-O((L+n)^{-1})$,
  \begin{align}
    \|S_1\|
    \le & C\bigg[
    v_1n \log(L+n)\sigma_1+v_1\sqrt{n}\log(L+n)\sigma_1'+\sqrt{v_1}R_1\sqrt{nL}\log^{3/2}(L+n)\sigma_2\nonumber\\
&\quad+\sqrt{v_2'}\log(L+n)\sigma_2'+(R_1^2+R_2')\log^2(L+n)\sigma_3
    \bigg]\,,\label{eq:S1_asym_small_n}
  \end{align}
  for some constant $C$.
If in addition \Cref{ass:bern2} holds, then with probability at least $1-O((L+n)^{-1})$,
\begin{align}
  \|S_2-\mathbb E S_2\|\le C\left[\sqrt{v_2\log(L+n)}\sigma_2'+R_2\log(L+n)\sigma_3\right]\,.
\end{align}
\end{theorem}

The proof follows largely the same scheme as in the symmetric case, with two notable differences.  First, in the asymmetric case $S_2$ only has diagonal entries. So the bounds for $S_2-\mathbb E S_2$ and $\tilde S_2$ only involve $\sigma_2'$ and $\sigma_3$.  Second, there is an additional term involving $\sigma_1'$ in the bound of $S_1$, which comes from the $n\|\sum_\ell H_\ell^T H_\ell\|\vee \sum_\ell \|H_\ell\|_F^2$ term in \Cref{lem:mat_prod}, because in the asymmetric case it is unclear whether the maximum is achieved by the operator norm part or the Frobenius norm part.

\begin{remark}\label{rem:asym_quad_n>r}
When $n\ge r$, we can drop the $\sigma_1'$ term and the high probability upper bound on $S_1$ can be reduced to 
 \begin{align}
    \|S_1\|
    \le & C\bigg[
    v_1n \log(L+n)\sigma_1+\sqrt{v_1}R_1\sqrt{nL}\log^{3/2}(L+n)\sigma_2\nonumber\\
&\quad+\sqrt{v_2'}\log(L+n)\sigma_2'+(R_1^2+R_2')\log^2(L+n)\sigma_3
    \bigg]\,.\label{eq:S1_asym_large_n}
  \end{align}  
\end{remark}

\begin{remark}\label{rem:comp_HansonWright}
In the special case of $L=1$, $n=1$, and $K$-sub-Gaussian entries, the proof of \Cref{thm:asym_quad} can be modified to show that
$$
\|S_1\|=O_P(K^2\|G\|_F)\,,
$$
which agrees with the Hanson--Wright inequality \citep{HansonW71,RudelsonV13}. 
\end{remark}

\begin{proof}[Proof of \Cref{thm:asym_quad}]
  Define $\tilde S$, $\tilde S_1$, $\tilde S_2$ accordingly. It is easy to check that $\tilde S_2$ only has diagonal entries and can be bounded by the same technique as in the symmetric case where
  \begin{equation}
  \|\tilde S_2\|\lesssim \sqrt{v_2'\log(L+n)}\sigma_2'+R_2'\log(L+n) \sigma_3\,.\label{eq:tilde_s2_asym}
  \end{equation}
  with high probability.

  For $\tilde S$, let $H_\ell=G_\ell \tilde X_\ell^T$, then
  with high probability
  $$
  \|H_\ell\|\lesssim \sqrt{v_1\log(L+n)}\left(\sqrt{n}\|G_\ell\|
  \vee \|G_\ell\|_F\right)+R_1\log(L+n)\|G_\ell\|_{2,\infty}\,.
  $$
and  \begin{align}
    \left\|\sum_{\ell=1}^{L} H_\ell^T H_\ell\right\|
    \lesssim & v_1\log(L+n) \left(n  \sigma_1^2 + (\sigma_1')^2\right)+R_1^2 L \log^2(L+n)\sigma_2^2\,,\label{eq:HTH_asym_small_n}
  \end{align}
The rest of the proof is the same as that of \Cref{thm:quad_sum}.
\end{proof}

\section{Proofs for the sparse Bernoulli case}\label{app:2}
The proof of \Cref{thm:s1_sparse_bernoulli} follows a similar idea to that of \Cref{thm:quad_sum}, which uses decoupling and reduces the problem to a linear combination in the form of $\sum_\ell X_\ell H_\ell$.  The proof here uses a refinement in constructing $H_\ell$ and controlling $\|\sum_\ell H_\ell^T H_\ell\|$ using properties of Bernoulli random variables.  The refinement involves carefully bounding the degrees of $A_\ell$, as well as $\sum_\ell A_\ell^2$, which is provided in \Cref{lem:counting}.

\begin{lemma} \label{lem:counting} Let $A_1,\ldots,A_L$ be independent adjacency matrices generated by a multi-layer SBM
satisfying the condition of \Cref{thm:consistency}. The following statements hold
simulatenously with probability at least $1-O((L+n)^{-1})$ for some universal constant $C$: 
  \begin{enumerate}
    \item $\max_{\ell,i}d_{\ell,i}\le C \log(L+n)$.
    \item $\max_{i}\sum_\ell d_{\ell,i}\le CLn\rho$.
    \item $\sum_{\ell,i}d_{\ell,i}\le C Ln^2\rho$.
    \item $\|\sum_{\ell}A_\ell^2\|\le C L n\rho$.
  \end{enumerate}
\end{lemma}
\begin{proof}
  Part 1 follows from direct application of Bernstein's inequality and union bound,
  $$
  \mathbb P\left[d_{\ell,i}-n\rho \ge t\right]\le \exp\left(-\frac{t^2/2}{4n\rho+2t}\right)\,,
  $$
and use the assumption that $n\rho \le C_2\log n$.

 For Part 2, observe that $\sum_{\ell=1}^{L} d_{\ell,i}$ has expected value at most $Ln\rho$. To control the deviation, Bernstein's inequality implies that
 $$
 \mathbb P\left[\sum_{\ell=1}^{L} (d_{\ell,i}-\mathbb E d_{\ell,i})\ge t\right]\le
 \exp\left(-\frac{t^2/2}{4\rho nL + 2t}\right)
 $$
  with probability at least $1-O((L+n)^{-1})$
 \begin{align*}
 \max_i \sum_{\ell=1}^{L} (d_{\ell,i}-\mathbb E d_{\ell,i}) \le &
C\left[ \rho^{1/2}n^{1/2}L^{1/2}\log^{1/2}(L+n)+\log (L+n)\right]\\
\le & C\rho^{1/2}n^{1/2}L^{1/2}\log^{1/2}(L+n)\,.
 \end{align*}

For Part 3, first we have $\mathbb E\sum_{\ell,i}d_{\ell,i}\le Ln^2\rho$, and
the deviation $X_{\ell,ij} = A_{\ell,ij} - P_{\ell,ij}$ satisfies
$$
\mathbb P\left[\sum_{\ell=1}^{L}\sum_{ij}X_{\ell,ij}\ge t\right]\le \exp\left(-\frac{t^2/2}{16\rho n^2 L + 4t}\right)\,.
$$
The claim follows from the assumption $\rho n^2 L \gtrsim \rho^{1/2}n L^{1/2}\sqrt{\log (L+n)}+\log(L+n)$.
  
  For Part 4, we first decompose 
  $$
  \sum_{\ell=1}^{L} A_\ell^2 = S_{1,A}+S_{2,A}.
  $$
  where $S_{2,A}$ is the diagonal part of $\sum_\ell A_\ell^2$, with
  $$
  (S_{2,A})_{ii}=\sum_{\ell=1}^{L} d_{\ell,i}\,.
  $$
  Using Part 2, we have with high probability
  \begin{equation}\label{eq:sum_A_square_diag}
    \|S_{2,A}\|=\max_{i}(S_{2,A})_{ii}\le CLn\rho\,.
  \end{equation}
For the off-diagonal part $S_{1,A}=\sum_\ell A_\ell^2-S_{2,A}$, we can obtain a high probability bound using decoupling. Let  $\tilde S_{1,A}$ be the corresponding version of $S_{1,A}$ for $\sum_\ell A_\ell \tilde A_\ell$.
For a matrix $M$, let $\|M\|_{1,\infty}=\max_i \sum_j|M_{ij}|$ be the maximum row-wise $\ell_1$ norm. Using symmetric dilation, Perron-Frobenius theorem and non-negativity of $A_\ell$, $\tilde A_\ell$ we have
$$
\|\tilde S_{1,A}\|\le \max\left\{\|\tilde S_{1,A}\|_{1,\infty},\|\tilde S_{1,A}^T\|_{1,\infty}\right\} \le\max\left\{ \left\|\sum_{\ell=1}^{L} A_\ell \tilde A_\ell\right\|_{1,\infty},\left\|\sum_{\ell=1}^{L} \tilde A_\ell  A_\ell\right\|_{1,\infty}\right\}\,.
$$
By symmetry, it suffices to upper bound the maximum row sum of $\sum_\ell A_\ell \tilde A_\ell$. 
The sum of the $i$th row is
$$
\sum_{\ell=1}^{L} \sum_{j,m}  A_{\ell,im}\tilde A_{\ell,jm}=\sum_{\ell=1}^{L} \sum_{m=1}^{n}A_{\ell,im} \tilde d_{\ell,m}
$$
whose expected value is upper bounded by $\rho^2 n^2 L$.

Intersecting on the event that $\max_{\ell,m}\tilde d_{\ell,m}\le C\log(L+n)$ and $\max_i\sum_\ell \tilde d_{\ell,i}\le CLn\rho$, the mean deviation
$$
\sum_{\ell=1}^{L} \sum_{m=1}^{n} X_{\ell,im}\tilde d_{\ell,m}
$$
can be bounded Bernstein's inequality
\begin{align*}
&\mathbb P\left[\sum_{\ell=1}^{L} \sum_{m=1}^{n}  X_{\ell,im} \tilde d_{\ell,m}\ge t\bigg| \tilde A_\ell\right]
\\\le &
\exp\left(-\frac{t^2/2}{4\rho\sum_{\ell,m}\tilde  d_{\ell,m}^2+2t\max_{\ell,m}\tilde d_{\ell,m}}\right)\\
\le & \exp\left(-\frac{t^2/2}{4C\rho\log(L+n)\sum_{\ell,m}\tilde  d_{\ell,m}+2Ct\log(L+n)}\right)\\
\le & \exp\left(-\frac{t^2/2}{4C\log(L+n)\rho^2 n^2 L +2Ct\log(L+n)}\right)\,. 
\end{align*}
Using a union bound over $i\in\{1,\ldots,n\}$ we conclude that with probability at least $1-O((L+n)^{-1})$
\begin{align*}
 \max_i \sum_{\ell=1}^{L}\sum_{m=1}^{n} X_{\ell,im}\tilde d_{\ell,m} \le C\rho nL^{1/2}\log(L+n)\le C \rho^2 n^2 L\,.
\end{align*}
Therefore we proved that with high probability $\|\tilde S_{1,A}\|\le C\rho^2 n^2 L$. Combining this with \eqref{eq:sum_A_square_diag} we have with high probability
$$
\left\|\sum_{\ell=1}^{L}A_\ell^2\right\|\le \|S_{1,A}\|+\|S_{2,A}\| \le C \rho n L \,. \qedhere
$$
\end{proof}

\begin{proof}[Proof of \Cref{thm:s1_sparse_bernoulli}]
By the sparse Bernoulli assumption, $X_\ell$ satisfy \Cref{ass:bern1} with $(v_1,R_1)=(2\rho,1)$ and \Cref{ass:bern2p} with $(v_2',R_2')=(2\rho^2,1)$.
Using the decoupling argument, we reduce the problem to controlling
$\tilde S$ and $\tilde S_2$ respectively.

First, for $\tilde S_2$, it is easy to verify that $\tilde S_2$ is a diagonal matrix with
$$
(\tilde S_2)_{ii}=\sum_{\ell=1}^{L}\sum_{j=1}^{n} X_{\ell,ij}\tilde X_{\ell,ij}\,,
$$
which is a sum of $nL$ independent zero-mean, $(2\rho^2,1)$-Bernstein random variables.
Using Bernstein's inequality and union bound over $i$, we have with probability at least $1-(L+n)^{-1}$
\begin{align}
  \|\tilde S_2\|=\max_i \left|(\tilde S_2)_{ii}\right|\le C\rho n^{1/2}L^{1/2}\sqrt{\log (L+n)}\,.\label{eq:tilde_s2_sb}
\end{align}

Second, we turn to $\tilde S$.  Recall that $X_\ell=A_\ell-P_\ell$, where $A_\ell$ consists of the uncentered versions of the corresponding entries of $X_\ell$, and $P_\ell=\mathbb E A_\ell$.
\begin{align*}
\tilde S=  \sum_{\ell=1}^{L} X_\ell \tilde X_\ell = \Big[\sum_{\ell=1}^{L} X_\ell \tilde A_\ell\Big] - \Big[\sum_{\ell=1}^{L} X_\ell P_\ell\Big].
\end{align*}

Using \Cref{lem:mat_prod} and the fact that $\|P_\ell\|\le n\rho$ and $\|P_\ell\|_{2,\infty}\le \rho n^{1/2}$, we have with probability at least $1-((L+n)^{-1})$ and universal constant $C$
\begin{align}
  \left\|\sum_{\ell=1}^{L} X_\ell P_\ell\right\| \le & C\left[\rho^{3/2}n^{3/2}L^{1/2}\sqrt{\log(L+n)}+\rho n^{1/2}\log (L+n)\right]\nonumber\\
  \le & C\rho^{3/2}n^{3/2}L^{1/2}\sqrt{\log(L+n)}\,.\label{eq:lin_part_tilde_s_sb}
\end{align}

Now we focus on $\sum_\ell X_\ell \tilde A_\ell$ by conditioning on $\tilde A_\ell$.
By \Cref{lem:counting}, with high probability $\|\sum_\ell \tilde A_\ell^2\|\le C \rho n L$ and $\max_\ell \|\tilde A_\ell\|_{2,\infty}=\max_{\ell,i} \tilde d_{\ell,i}^{1/2}\le C\log^{1/2}(L+n)$. Applying \Cref{lem:mat_prod}, intersecting on this event, we have with high probability
\begin{align}
 \left\|\sum_{\ell=1}^{L} X_\ell \tilde A_\ell\right\|\le & C\left[\rho^{1/2}n^{1/2}(\rho n L\log(L+n))^{1/2}+\log^{3/2}(L+n)\right]\nonumber\\
 \le & C \rho n L^{1/2}\log^{1/2}(L+n)\,.\label{eq:sum_X_A_sb}
  \end{align}
Combining \eqref{eq:lin_part_tilde_s_sb} and \eqref{eq:sum_X_A_sb} we obtain with high probability
\begin{equation}\label{eq:tilde_s_sb}
\|\tilde S\|\le C \rho n L^{1/2}\log^{1/2}(L+n)\,.
\end{equation}
The claimed bound holds for $\tilde S_1$ by combining \eqref{eq:tilde_s2_sb} and \eqref{eq:tilde_s_sb}, and hence holds for $S_1$ by decoupling.
\end{proof}

\section{Proof of consistency of bias-adjusted sum-of-squared spectral clustering}\label{app:consistency}
The plan is to decompose the matrix $S_0$ into the sum of a signal term and a noise term, where the signal term has a leading principal subspace with perfect clustering, and then apply matrix perturbation results (the Davis-Kahan $\sin\Theta$ theorem) combined with a standard error analysis of the K-means algorithm. We first introduce some notation a preliminary result for the K-means problem.

Given an $n\times d$ matrix $\hat U$, the K-means problem is an optimization problem
$$
\min_{\Theta,X} \|\hat U - \Theta X\|_F^2
$$
where the minimization is over all $\Theta \in\{0,1\}^{n\times K}$ with each row has exactly one ``$1$'', and all $X\in\mathbb R^{K\times d}$.
We say a pair $(\hat\Theta,\hat X)$ is an $(1+\epsilon)$-approximate solution if its objective function value is no larger than $(1+\epsilon)$ times the optimal value.
\begin{lemma}[Simplified from Lemma 5.3 of \cite{LeiR14}]\label{lem:k-means}
  Let $U$ be an $n\times d$ matrix with $K$ distinct rows with minimum pairwise Euclidean norm separation $\gamma$. Let $\hat U$ be another $n\times d$ matrix and $(\hat\Theta,\hat X)$ be an $(1+\epsilon)$-approximate solution to K-means problem with input $\hat U$, then the number of errors in $\hat\Theta$ as an estimate of the row clusters of $U$ is no larger than
  $$
  C_\epsilon\|\hat U-U\|_F^2\gamma^{-2}
  $$
  for some constant $C_\epsilon$ depending only on $\epsilon$.
\end{lemma}

\begin{proof}[Proof of \Cref{thm:consistency}]
Let $Q_\ell = \rho \Theta B_{\ell,0} \Theta^T$ then $P_\ell= Q_\ell-{\rm diag}(Q_\ell)$ and
\begin{align*}
\sum_{\ell=1}^{L} P_\ell^2=&\sum_{\ell=1}^{L} Q_\ell^2+[{\rm diag}(Q_\ell)]^2-Q_\ell{\rm diag}(Q_\ell)-{\rm diag}(Q_\ell) Q_\ell\\
=& \sum_{\ell=1}^{L} Q_\ell^2+E_1,
\end{align*}
where $E_1=\sum_\ell [{\rm diag}(Q_\ell)]^2-Q_\ell{\rm diag}(Q_\ell)-{\rm diag}(Q_\ell) Q_\ell$.
Furthermore, define the following additional error terms, 
\begin{align*}
  E_2=&\sum_{\ell=1}^{L} X_\ell P_\ell+P_\ell X_\ell\,,\\
  E_3=&S_2-\sum_{\ell=1}^{L} D_\ell\,,\\
  E_4=&S_1\,,
\end{align*}
where $S_1$, $S_2$ are defined as in \eqref{eq:S2} and \eqref{eq:S1} with $G_\ell=I_n$.
By the definition of $S_0$ in \eqref{eq:S0} and the decomposition \eqref{eq:decomp}, we have
$$
S_0=\sum_{\ell=1}^{L} Q_\ell^2 + E_1 + E_2 + E_3 + E_4\,.
$$

Let $\Theta=\tilde \Theta \Lambda$ where $\Lambda$ is a $K\times K$ diagonal matrix with $k$th diagonal entry being the $\ell_2$ norm of the $k$th column of $\Theta$. Then $\tilde \Theta$ is orthonormal. By the balanced community size assumption and $K$ assumed being consistent, the minimum eigenvalue of $\Lambda$ is lower bounded by $cn^{1/2}$ for some constant $c$.
We first lower-bound the signal term to be,
\begin{align}
\sum_{\ell=1}^{L} Q_\ell^2=&\rho^2 \tilde \Theta \Lambda \left(\sum_{\ell=1}^{L} B_{\ell,0}\Lambda^2 B_{\ell,0}\right)\Lambda\tilde \Theta^T\label{eq:signal_exp_1}\\
\succeq & cn\rho^2 \tilde \Theta\Lambda\left(\sum_{\ell=1}^{L} B_{\ell,0} B_{\ell,0}\right) \Lambda\tilde \Theta^T\nonumber\\
\succeq & cLn\rho^2 \tilde \Theta\Lambda^2\tilde \Theta^T\nonumber\\
\succeq & cLn^2\rho^2 \tilde \Theta\tilde \Theta^T\,,\label{eq:signal_exp_2}
\end{align}
where we used $\Lambda^2\succeq cn I_K$ and \Cref{ass:signal}.  Note that \eqref{eq:signal_exp_1} implies that the matrix $\sum_\ell Q_\ell^2$ is rank $K$ and the leading eigen-space is spanned by the columns of $\tilde \Theta$, and\eqref{eq:signal_exp_2} implies the smallest non-zero eigenvalue of $\sum_\ell Q_\ell^2$ is lower bounded by $cLn^2\rho^2$.

We now upper-bound the spectral norm of all the error terms. 
The first bias term $E_1$ is non-random and satisfies $\|E_1\|\le Ln\rho^2$.

For the noise term $E_2$, 
applying \Cref{lem:mat_prod} with $H_\ell=P_\ell$ and realizing that $\|P_\ell\|\le n\rho$ and $\|P_\ell\|_{2,\infty}\le \sqrt{n}\rho$,
we have with high probability
\begin{align}
\|E_2\|\le  & CL^{1/2}n^{3/2}\rho^{3/2}\log^{1/2}(L+n)\,.\label{eq:E2_op}
\end{align}

For $E_3=S_2-\sum_\ell D_\ell$, the decomposition \eqref{eq:S2_diag_decomp} implies that
$\|S_2-\sum_\ell D_\ell\|$ can be upper bounded deterministically by $Ln\rho^2$.  

Next we control $E_4=S_1$. 
Using \Cref{thm:s1_sparse_bernoulli},
we have
\begin{equation}\label{eq:E3_S1_op}
  \|S_1\|\le C L^{1/2}n\rho \log^{1/2}(L+n)
\end{equation}
with high probability.  Thus
\begin{align}
\frac{\|E_1+E_2+E_3+E_4\|}{\lambda_K(\sum_{\ell}Q_\ell^2)}\le & C\cdot \frac{Ln\rho^2+L^{1/2}n\rho\log^{1/2}(L+n)}{L n^2\rho^2}\nonumber\\
\le & \frac{C}{n}+\frac{C\log^{1/2}(L+n)}{L^{1/2}n\rho}\,,\label{eq:spec_err_bound}
\end{align}
where $\lambda_K(\sum_\ell Q_\ell^2)$ is the $K$th (and smallest) non-zero eigenvalue of $\sum_\ell Q_\ell^2$. Let $U $ and $\hat U$ be the $n\times K$ matrices consisting of the leading eigenvectors of $\sum_\ell Q_\ell^2$ and $S_0$, respectively. 
By the Davis-Kahan $\sin\Theta$ theorem, we have
 \[
 \|\hat U - U\|_F\le \sqrt{K}\|\hat U-U\|\le \frac{\sqrt{K}\|E_1+E_2+E_3+E_4\|}{\lambda_{K}(\sum_\ell Q_\ell^2)-\|E_1+E_2+E_3+E_4\|}\lesssim n^{-1}+\log^{1/2}(L+n)/(L^{1/2}n\rho).
 \]
 The rest proof follows from \Cref{lem:k-means} because Part 1 of \Cref{ass:signal} implies that the minimum separation of two distinct rows in $U$ is at least $C/\sqrt{n}$ for some constant $C$. (See Lemma 2.1 of \cite{LeiR14}.)
\end{proof}

\begin{proof}[Proof of \Cref{cor:varying_membership}]
  The proof of \Cref{cor:varying_membership} follows the same strategy as that of \Cref{thm:consistency}. Here we only decribe the differences.  We use the notation $\lesssim$ to denote ``bounded up to a universal constant factor''.

 Define $\tilde P_\ell =\rho\Psi_\ell B_\ell \Psi_\ell^T$, $\tilde P_\ell^*=\rho\Psi B_\ell \Psi^T$, and $P_\ell$, $P_\ell^*$ are the corresponding matrices with diagonal zeroed out.
Then the diagonal-removed squared adjacency matrices have the following decomposition,
\begin{align*}
  A_\ell^2-D_\ell = P_\ell^2 + P_\ell E_\ell+E_\ell P_\ell +(E_\ell^2-D_\ell)\,.
\end{align*}

It is easy to verify that the last three noise terms can be bounded using identical arguments, provided that the maximum $\ell_2$ norm of the rows of $P_\ell$ is uniformly bounded by a constant factor of the corresponding quantity of $P_\ell^*$, which is implied by \Cref{ass:signal}(a).
Therefore the only part that requires treatment is the squared signal term. It suffices to control the difference between the layer-wise squared signal and the common squared signal,
\begin{align*}
  P_\ell^2 - (\tilde P_\ell^*)^2 = & \big(P_\ell^2- \tilde P_\ell^2\big)+ \big(\tilde P_\ell^2-(\tilde P_\ell^*)^2\big)\,.
\end{align*}

The first term is bounded by $O(n\rho^2)$, using the same argument as upper bounding the ``$E_1$'' term in the proof of \Cref{thm:consistency}.

The second term is
\begin{align*}
  (1/2)(\tilde P_\ell-\tilde P_\ell^*)(\tilde P_\ell+\tilde P_\ell^*)+(1/2)(\tilde P_\ell+\tilde P_\ell^*)(\tilde P_\ell-\tilde P_\ell^*)\,.
\end{align*}
 With $\|\tilde P_\ell+\tilde P_\ell^*\|\lesssim n\rho$, we only need to control $\|\tilde P_\ell-\tilde P^*_\ell\|$. Observe,
 \begin{align*}
   \rho^{-1}(\tilde P_\ell-\tilde P^*_\ell) = & \Psi_\ell B_\ell\Psi_\ell^T-\Psi B_\ell \Psi^T\\
=& \Psi_\ell B_\ell (\Psi_\ell-\Psi)^T + (\Psi_\ell-\Psi) B_\ell \Psi
   \end{align*}
By construction, we have $\|\Psi\|\lesssim \sqrt{n}$, so
$$
\|\tilde P_\ell^2-(\tilde P_\ell^*)^2\|\le C (n\rho^2 + \epsilon_\ell n^2\rho^2)\,.
$$
Therefore, the additional spectral perturbation added to the common squared signal due to the varying membership is no larger than
$$
C L n\rho^2 + L \bar\epsilon n^2\rho^2\,.
$$
As a result, the right hand side of \eqref{eq:spec_err_bound} becomes
$$
\frac{C}{n}+C \bar\epsilon +\frac{C\log^{1/2}(L+n)}{L^{1/2}n\rho}\,,
$$
and the rest of the proof follows that of \Cref{thm:consistency}.
\end{proof}

\paragraph{Details about \Cref{rem:n_rho>1}.}  The proof of the claimed error bound in \Cref{rem:n_rho>1} follows the same steps as that of \Cref{thm:consistency}.  In this case \Cref{lem:counting} still holds without modification.  The only change in  the proof of \Cref{thm:consistency} is that \eqref{eq:E2_op} does not further simplify and will become the dominant term in the numerator of \eqref{eq:spec_err_bound}.  So the error bound $\|\hat{U} - U\|_F$ becomes, ignoring the constant factor,
$$
\frac{1}{n}+\frac{L^{1/2}n^{3/2}\rho^{3/2}\log^{1/2}(L+n)}{Ln^2\rho^2} = \frac{1}{n}+\frac{\log^{1/2}(L+n)}{L^{1/2}n^{1/2}\rho^{1/2}}\,.
$$

\section{Additional results for simulation}  \label{app:additional_simu}

In this appendix section, we provide additional results to the simulations
in \Cref{sec:simu}.

\paragraph{Computational time needed.}
When computing the simulation results shown in \Cref{fig:simulation_rho}, 
the ``Sum'',``SoS'', and ``SoS-Debias'' methods  complete in less than 0.2 seconds on average (over all trials and simulation setting). In comparison, the ``Tensor matrization'' and ``LMF'' methods complete in less than 2 seconds on average (i.e., 10 times slower), and the ``Co-reg'' method completes in 4 seconds on average (i.e., 20 times slower). We note that ``Co-reg'' 
(i.e., Co-regularized Spectral Clustering) is slower when compared to ``SoS-Debias'' since 
the former method solves a more nuanced statistical problem, where it estimates both   global embedding as well as a layer-specific embedding for each of the $L$ layers.

\paragraph{Nodes with exceptionally high degree in low-sparsity regimes.}
In \Cref{fig:simulation_degree}, we show that when the overall edge density parameter
$\rho$ is too small, the nodes within each community have a highly variable degree, whereas when $\rho$ is larger, the nodes' degrees are more concentrated. 
This is an alternative way to understand why our bias-adjustment method is
important when $\rho$ is small, as the spectrum of $\sum_{\ell}A^2_{\ell}$ could
dramatically change when the diagonal entries are zero-ed out when the values on the diagonal (i.e., the degree of each node) is highly variable (shown previously in \Cref{fig:simulation_bias}). Relating \Cref{fig:simulation_degree} to
the simulation results in \Cref{fig:simulation_rho}, we see that when $\rho = 0.025$, the Communities 1 and 3 have nodes whose degree deviate far from the empirical mean. This is a phenomenon described in many theoretical analyses of sparse SBM such as \cite{LeLV17}. This means the diagonal entries can heavily distort the spectrum of $\sum_{\ell}A^2_{\ell}$, which explains the sum-of-squared method's poor performance in \Cref{fig:simulation_rho} (orange triangles). However, when $\rho =0.2$,
all three communities have nodes whose degree concentrate tightly around the mean. This means the diagonal entries no longer distort the spectrum of $\sum_{\ell}A^2_{\ell}$, resulting in the sum-of-squared method achieving perfect cluster estimation.

\begin{figure}[H]
  \centering
   \includegraphics[width=200px]{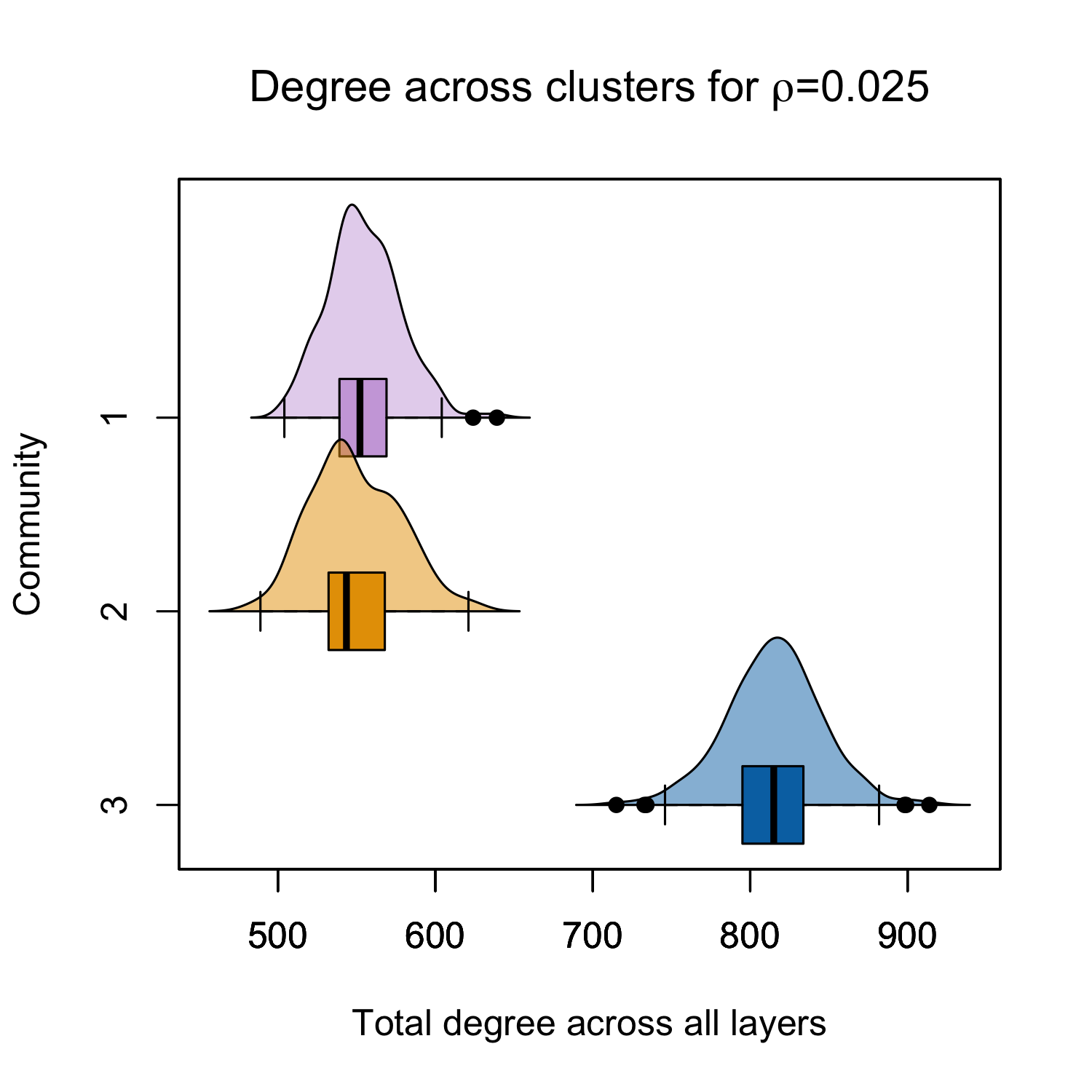}  
    \includegraphics[width=200px]{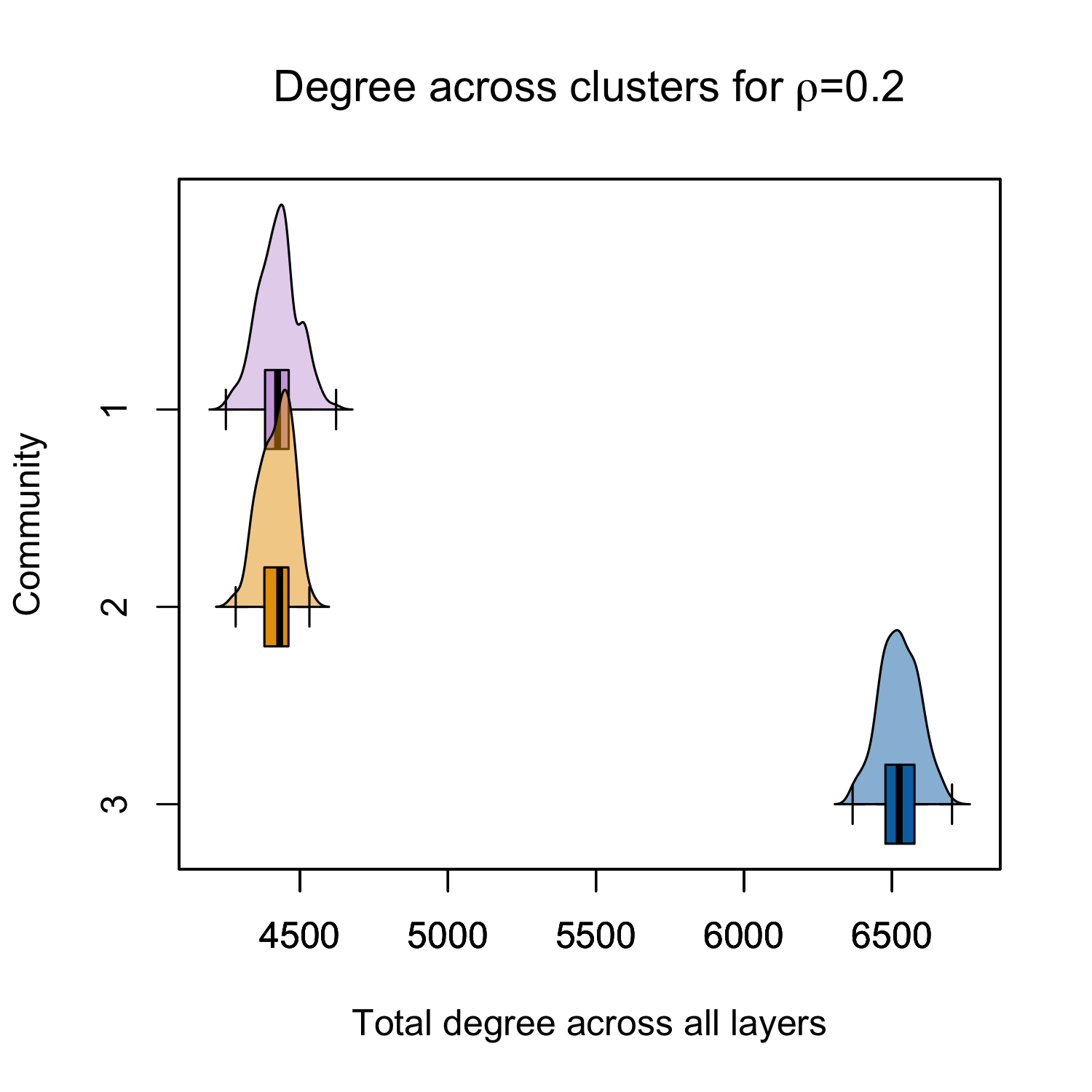}
       \vspace{-1em}  \caption
   { \small The distribution of diagonal entries of $\sum_{\ell}A_{\ell}^2$ (i.e., total degree among all $L$ layers) among nodes in the three communities based on one realization of $A_1,\ldots,A_L$ given the setup described in \Cref{sec:simu}, using $\rho = 0.025$ (left) or $\rho=0.2$ (right). Both the boxplot and estimate kernelized density are shown for each of the three communities, each colored differently. The black points depict nodes whose degree lie outside of the interquartile range of that communities' degrees. }
    \label{fig:simulation_degree}
\end{figure}

\paragraph{Effect of diagonal bias on the spectral embedding.} 
In \Cref{fig:simulation_3d}, we demonstrate how the effect of the diagonal bias
affects the spectral embedding (i.e., the leading $K=3$ eigenvectors of $\sum_{\ell}A_{\ell}^2$) when $\rho=0.15$. This is an important visualization since K-means
is performed on this spectral embedding. Specifically, \Cref{fig:simulation_3d} (top right) shows that when the diagonal entries are not zero-ed out, the nodes in the third community (blue) have an unusually high variance along the third dimension (i.e., corresponding to the third eigenvector). This matches our understanding from \Cref{fig:simulation_bias}, where at $\rho=0.15$, the third eigenvalue is indistinguishable from noise. This unusually-high variance among the third communities negatively impacts the resulting K-means clustering, shown in \Cref{fig:simulation_3d} (top left). However, when we consider the bias-adjusted variant of $\sum_{\ell}A_{\ell}^2$, the spectral embedding yields more-uniformly separated communities, shown in \Cref{fig:simulation_3d} (bottom right), which enables K-means clustering to recover the correct communities (bottom left).

\begin{figure}[H]
  \centering
   \includegraphics[width=350px]{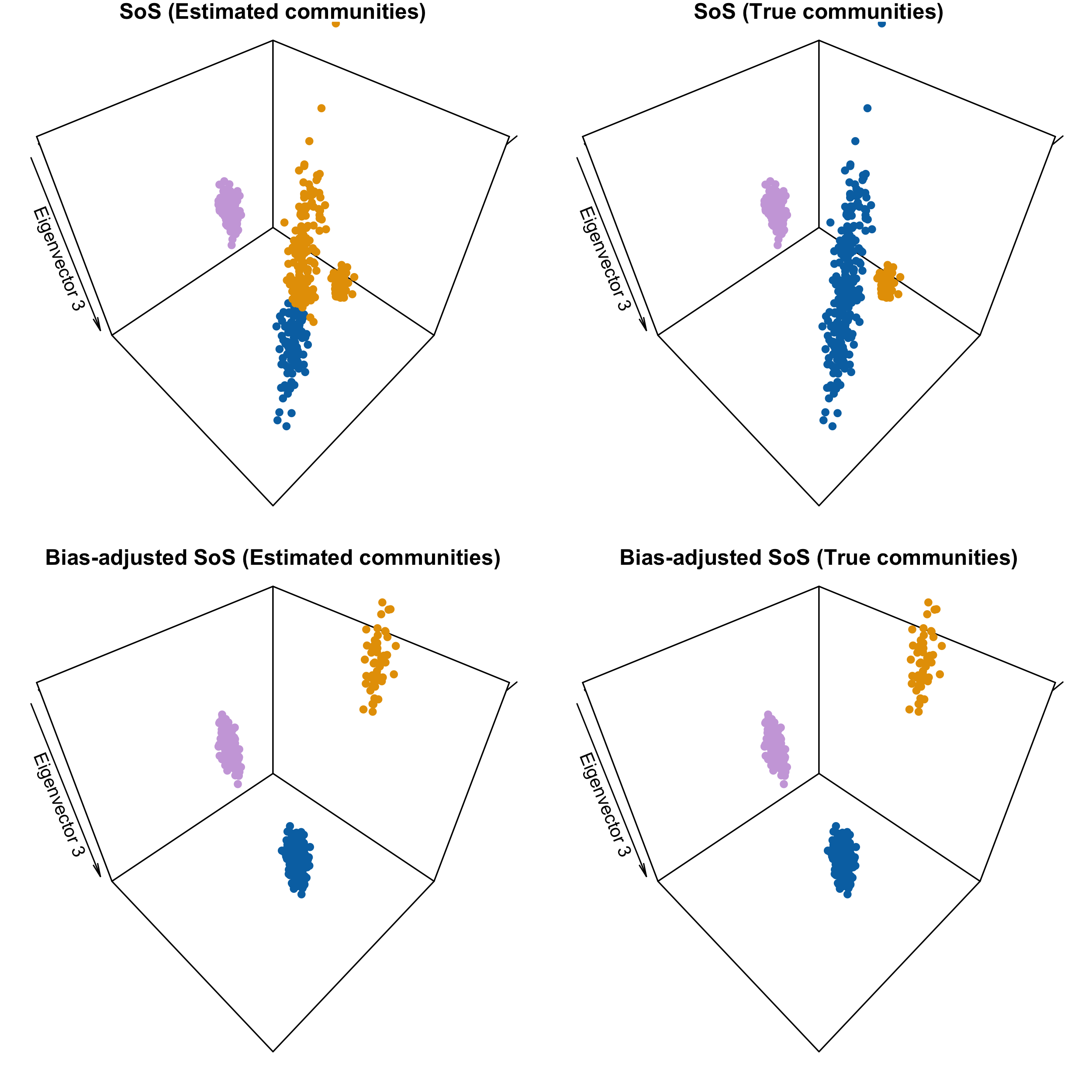}  
   \vspace{-2em}
   \caption
   { \small The visualization of all $n=500$ nodes based on the simulation setting in \Cref{sec:simu} for $\rho=0.15$, where the nodes have coordinates corresponding to the leading $K=3$ eigenvectors of $\sum_{\ell}A_{\ell}^2$ (top row) or the bias-adjusted variant of $\sum_{\ell}A_{\ell}^2$ (bottom row, i.e., zero-ing out the diagonal). The plots in the left column colors the nodes based estimated communities based on K-means of the shown spectral embedding, and the plots in the right column colors the nodes based on the true community.
   The coloring of the true communities correspond to the colors shown in \Cref{fig:simulation_degree}. }
    \label{fig:simulation_3d}
\end{figure}

\paragraph{Qualitative similarity of results under the varying-membership setting.}
As alluded to in \Cref{cor:varying_membership}, we consider a different data-generating process where the membership of each node can vary among the layers. 
Specifically, we call the original community structure among the $n$ nodes as the common community structure.
Then, 
separately for each of the $L$ layers,
among the $n=500$ nodes, we randomly reassign the membership of nodes in the following fashion, where $\delta \in [0,1)$ is a tunable parameter that dictates how often a node switches communities:
\begin{itemize}
\item The $n_1$ nodes in Community 1 stay in Community 1 with probability $1-\delta$
or switches to Community 3 with probability $\delta$.
\item The $n_2$ nodes in Community 2 stay in Community 2 with probability $1-\delta$
or switches to Community 3 with probability $\delta$.
\item The $n_3$ nodes in Community 3 stay in Community 3 with probability $1-\delta$
or switches to Community 1 or 2 with probability $4 \delta/5$ and $\delta/5$ respectively.
\end{itemize}
Note that while each node could belong to a different community among the $L$ layers,
this procedure keeps the relative sizes of the three communities the same.
Indeed, for any $\delta \in [0,1)$,
\[
\begin{bmatrix}
0.4 & 0.1 & 0.5
\end{bmatrix}
\begin{bmatrix}
1-\delta & 0 & \delta\\
0 & 1-\delta & \delta\\
4\delta/5 & \delta/5 & 1-\delta
\end{bmatrix}
=
\begin{bmatrix}
0.4 & 0.1 & 0.5
\end{bmatrix}
\]
Then, we generative the $L$ networks according to each layer's specific node-memberships as described in  \Cref{sec:simu}. Hence, our goal is to recover the common community structure despite observing $L$ networks where the memberships of each node slightly deviates from said community structure. 

In the first simulation for the varying-membership setting, we fix the overall edge density parameter $\rho=0.1$ and vary the membership-switching probability $\delta \in \{0, 0.05, \ldots, 0.5\}$. The results are shown in \Cref{fig:simulation_switching}. 
We see that when $\delta < 0.2$, our bias-adjusted sum-of-squared method is still able to estimate the common community structure better than the competing methods, but when
 $\delta \geq 0.2$, the variability of the node's membership more-or-less overwhelms all six methods. 
\begin{figure}[t]
  \centering
   \includegraphics[width=300px]{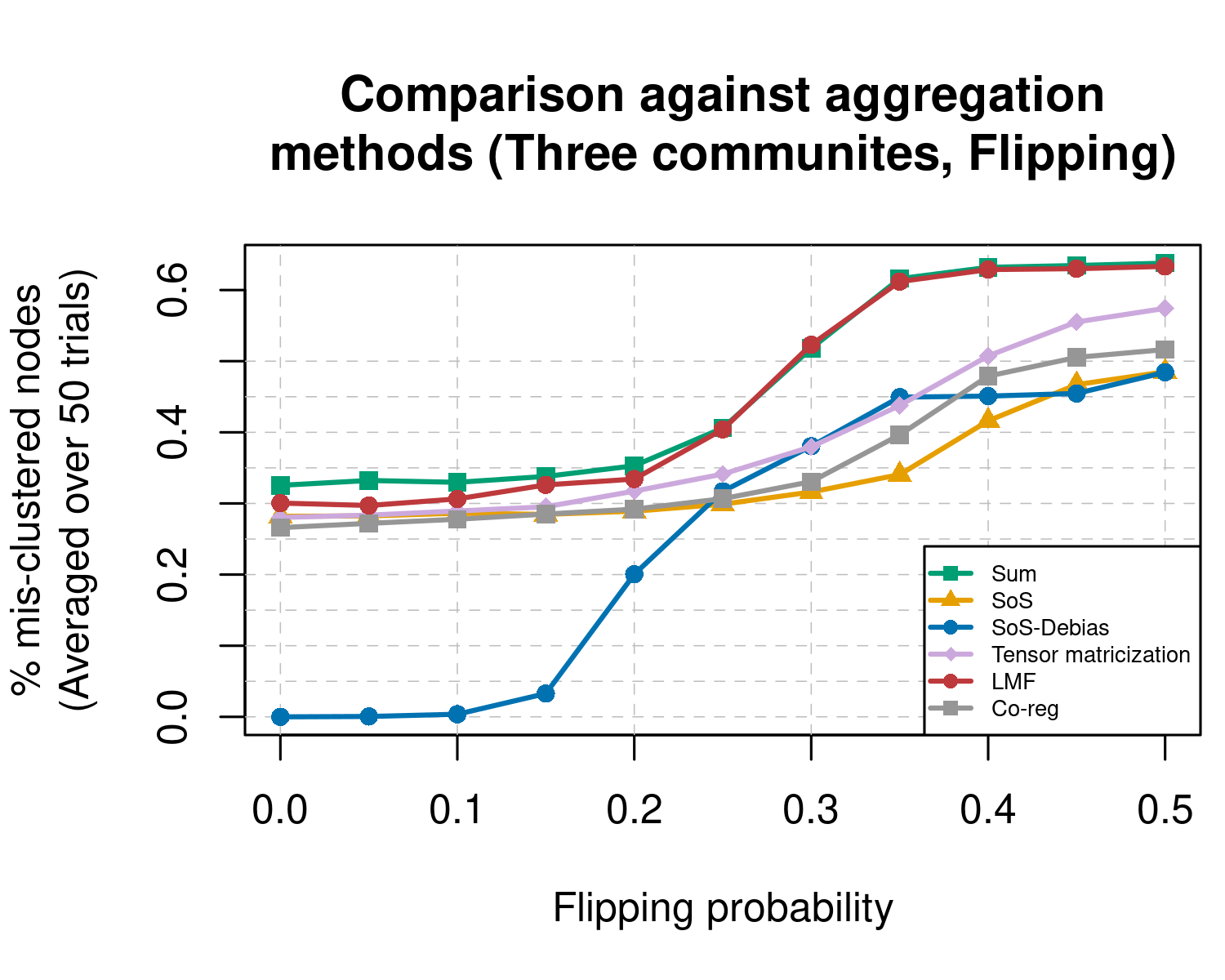} 
   \vspace{-2em}
    \caption{ \small 
 Results in the same layout as in \Cref{fig:simulation_rho}, but using the
 varying-membership data-generation process with $\rho=0.1$. When $\delta=0$, the results correspond to $\rho=0.1$ in \Cref{fig:simulation_rho}.
   }
    \label{fig:simulation_switching}
\end{figure}

In the second simulation for the varying-membership setting, we fix the overall edge density parameter $\rho=0.2$ and vary the membership-switching probability $\delta \in \{0, 0.05, \ldots, 0.5\}$. The results are shown in \Cref{fig:simulation_switching2}.
Interestingly, we see that when $\delta < 0.3$, our bias-adjusted sum-of-squared method is still able to estimate the common community structure better than the competing methods
despite many methods being able to estimate the communities reliably when $\delta=0$. 
We hope to investigate this phenomenon via statistical theory in future work.
\begin{figure}[t]
  \centering
   \includegraphics[width=300px]{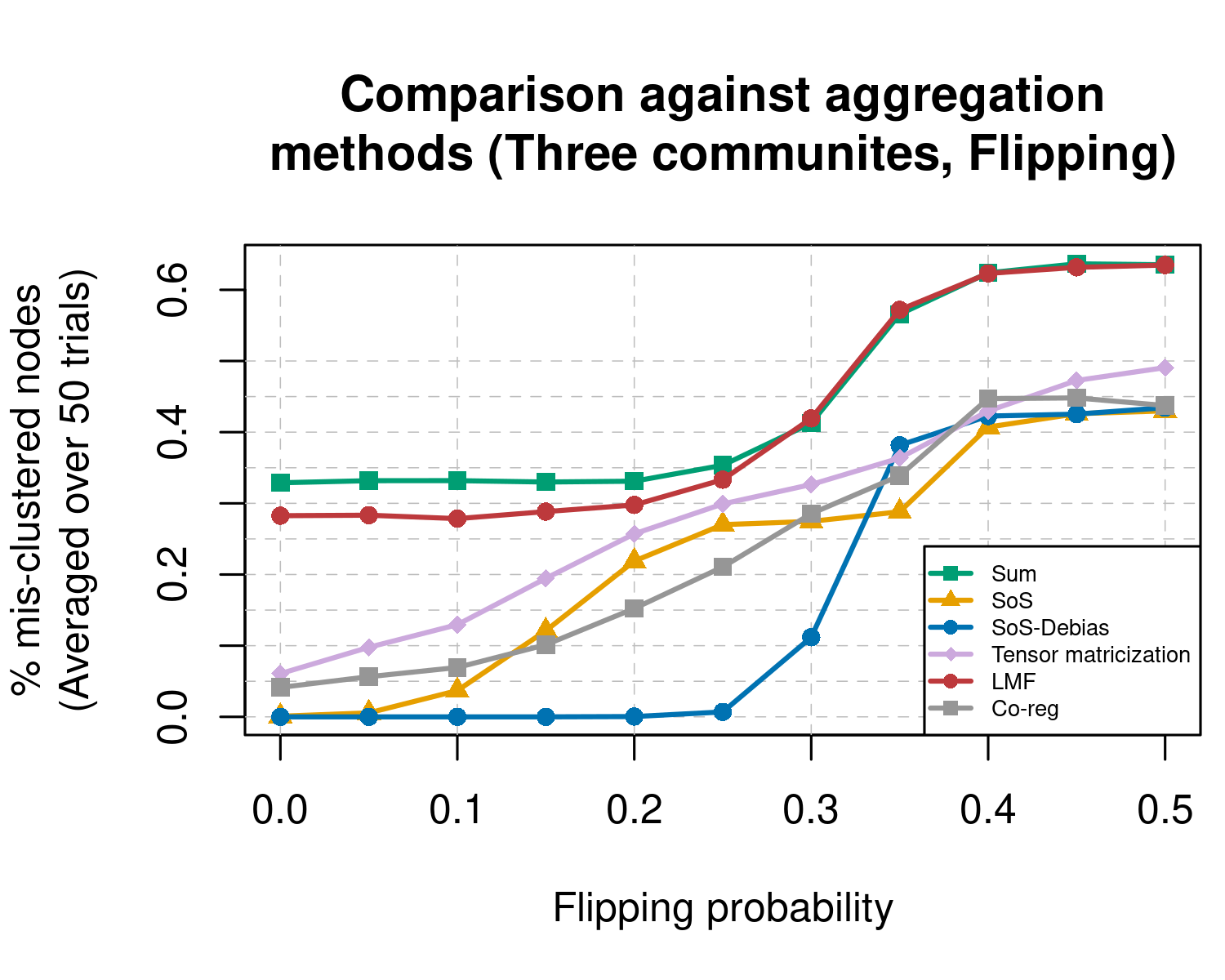} 
   \vspace{-2em}
    \caption{ \small 
Results in the same layout as in \Cref{fig:simulation_rho}, but using the
 varying-membership data-generation process with $\rho=0.2$. Observe that when $\delta=0$, the results correspond to $\rho=0.2$ in \Cref{fig:simulation_rho}.
   }
    \label{fig:simulation_switching2}
\end{figure}

\paragraph{Qualitative similarity of results when weighted eigenvectors are considered.}
In addition to the four methods considered in \Cref{sec:simu}, we can also consider another popular variant of spectral clustering where the eigenvectors are weighted by the eigenvalues, as mentioned in work like \cite{loffler2019optimality}.
Specifically, after we compute $M$ using one of the four aforementioned methods described in \Cref{sec:simu}, consider the SVD $M = UDV^\top$. Then, we perform K-means on $UD$.  The analogous results to \Cref{fig:simulation_rho} using this weighted spectral-clustering variant for the data-generation process described in \Cref{sec:simu} is shown in \Cref{fig:simulation_weighted}.

\begin{figure}[t]
  \centering
   \includegraphics[width=300px]{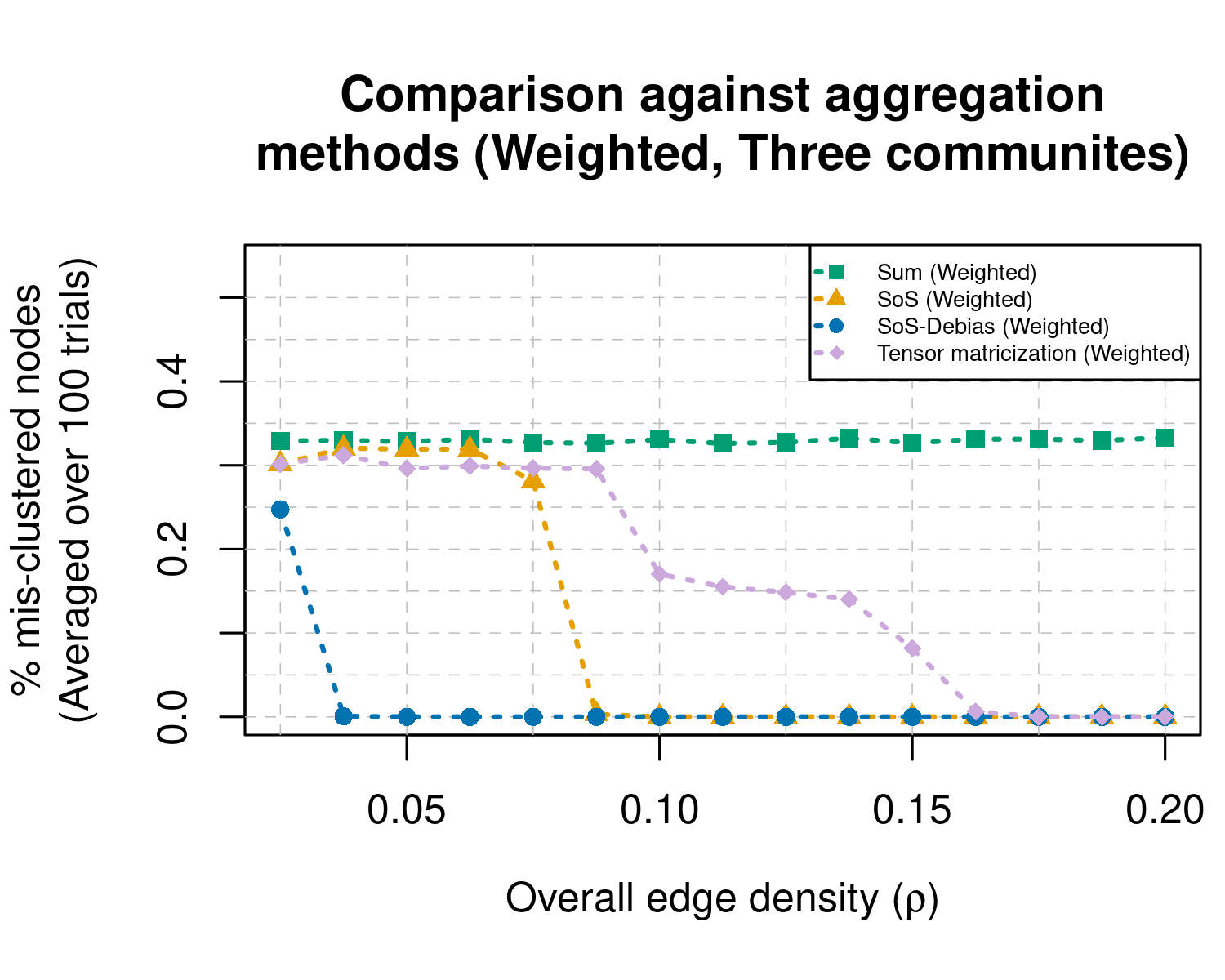} 
   \vspace{-2em}
    \caption{ \small 
 Results using the same exact setup as in \Cref{fig:simulation_rho}, but using the
 weighted spectral-clustering variant. 
   }
    \label{fig:simulation_weighted}
\end{figure}

Observe that reweighting the eigenvectors by the eigenvalues uniformly improves all procedures (i.e., comparing the solid line to its dotted line counterpart with the same color in \Cref{fig:simulation_rho}). This is analogous to clustering results shown in \cite{loffler2019optimality}. Lastly, comparing \Cref{fig:simulation_weighted} to \Cref{fig:simulation_rho}, we observe that using the weighted spectral-clustering on the bias-adjusted version of $\sum_{\ell}A_{\ell}^2$ (shown in blue circles) still enables a wide range of $\rho$ that yields perfect community estimation when compared to
those on $\sum_{\ell}A_{\ell}^2$ (shown in orange triangles). While we do not  prove results corresponding to this weighted spectral-clustering variant, we suspect the theoretical arguments in \cite{loffler2019optimality} would be applicable.

\section{Additional results for data analysis}  \label{app:additional_data}

In this appendix section, we provide additional results to the
data analysis in \Cref{sec:data}.

\paragraph{Additional summary statistics.}
We report the following summary statistics of all ten networks:
\begin{itemize}
\item \textbf{Degree}: The minimum degree across all networks is 0, the median degree ranged from 6 to 34 across all networks (i.e., $6/7836 \approx 0.07\%, 34/7836 \approx 0.4\%$), and the maximum degree
ranged from 406 to 1349 across all networks (i.e., $406/7836 \approx 5\%, 1349/7836 \approx 17\%$).

\item \textbf{Number of connected components}: It is well-documented that is impossible to non-trivially cluster between disconnected
components of a network. Hence, we check how many disconnected components reside within the ten networks.
The median number of connected components across all ten networks is 1605, where on average, 21\% among all 7836 nodes belong
to connected components with five or less nodes. In contrast, the number of connected components corresponding to $\sum_{\ell}A_{\ell}$ (where all non-zero entries are treated as 1) is one. This means one must consider some aggregated variant among $A_1,\ldots,A_{L}$ to meaningful perform clustering.

\item \textbf{Scale-free network}: Most networks observed in practice are scale-free, meaning the nodes' degrees follow the power-law distribution.
This means the empirical frequency of observing nodes with degree $k$ is inversely proportional to $k$ to a power (see \cite{ravasz2003hierarchical} and related work). Specifically,
if we let denote the empirical  frequency of observing nodes with degree $k$ as $p(k)$, then
\[
p(k) \propto k^{-\gamma},
\]
for some $\gamma$. Based on this model, we can measure how well a given network follows this power-law by measuring
\[
-\text{cor}\big(\log(p(k)), \log(k)\big).
\]
When we measure the above metric on all ten networks separately, the lowest squared correlation is 0.82, 
the median is 0.90, and the highest is 0.96, which means all ten networks have degree sequences that reasonably follow the power-law distribution.
\end{itemize}

\paragraph{Visualizing all adjacency matrices.}
Below, in \Cref{fig:app:adj_1to3} through \Cref{fig:app:adj_10}, we plot all ten adjacency matrices.

\begin{figure}[H]
  \centering
   \includegraphics[width=0.32\textwidth]{fig/networksosd/pnas_adj1}  
      \includegraphics[width=0.32\textwidth]{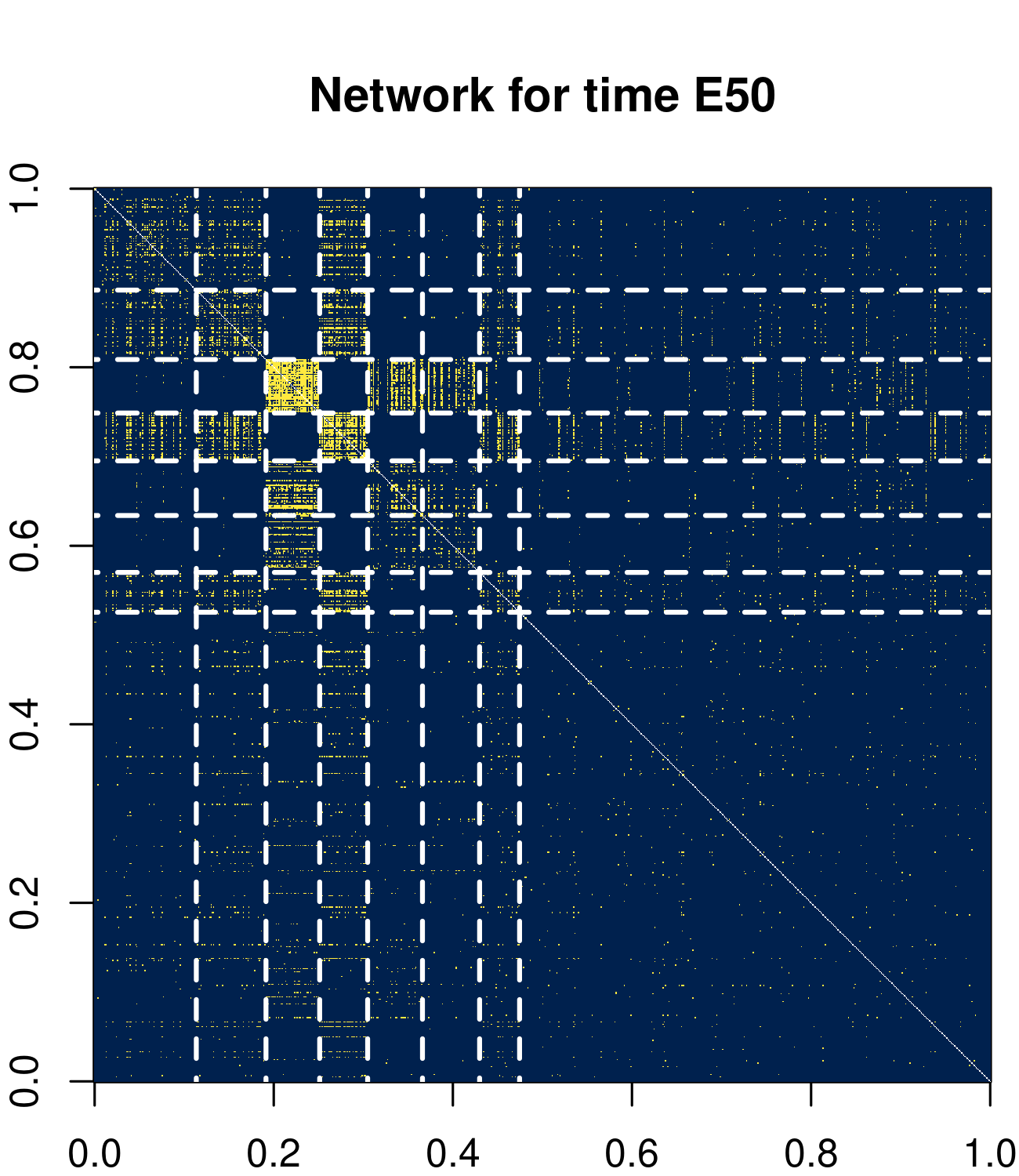}
         \includegraphics[width=0.32\textwidth]{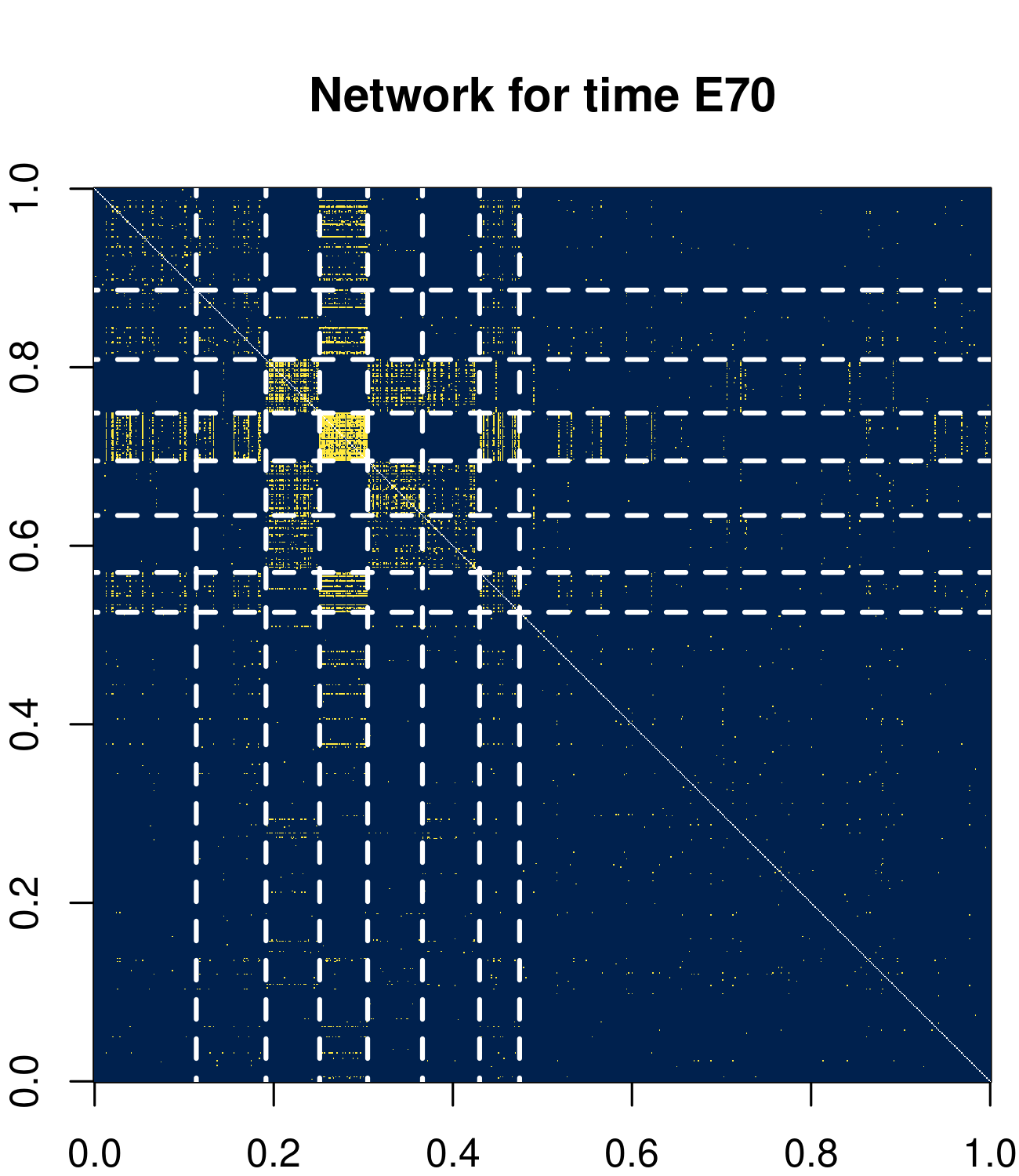}   \caption
   { \small Visualization similar to those in \Cref{fig:pnas_adj}, but for the adjacency matrices corresponding to the developmental times for 40, 50 and 70 days in the embryo (from left to right respectively).  }
    \label{fig:app:adj_1to3}
\end{figure}

\begin{figure}[H]
  \centering
   \includegraphics[width=0.32\textwidth]{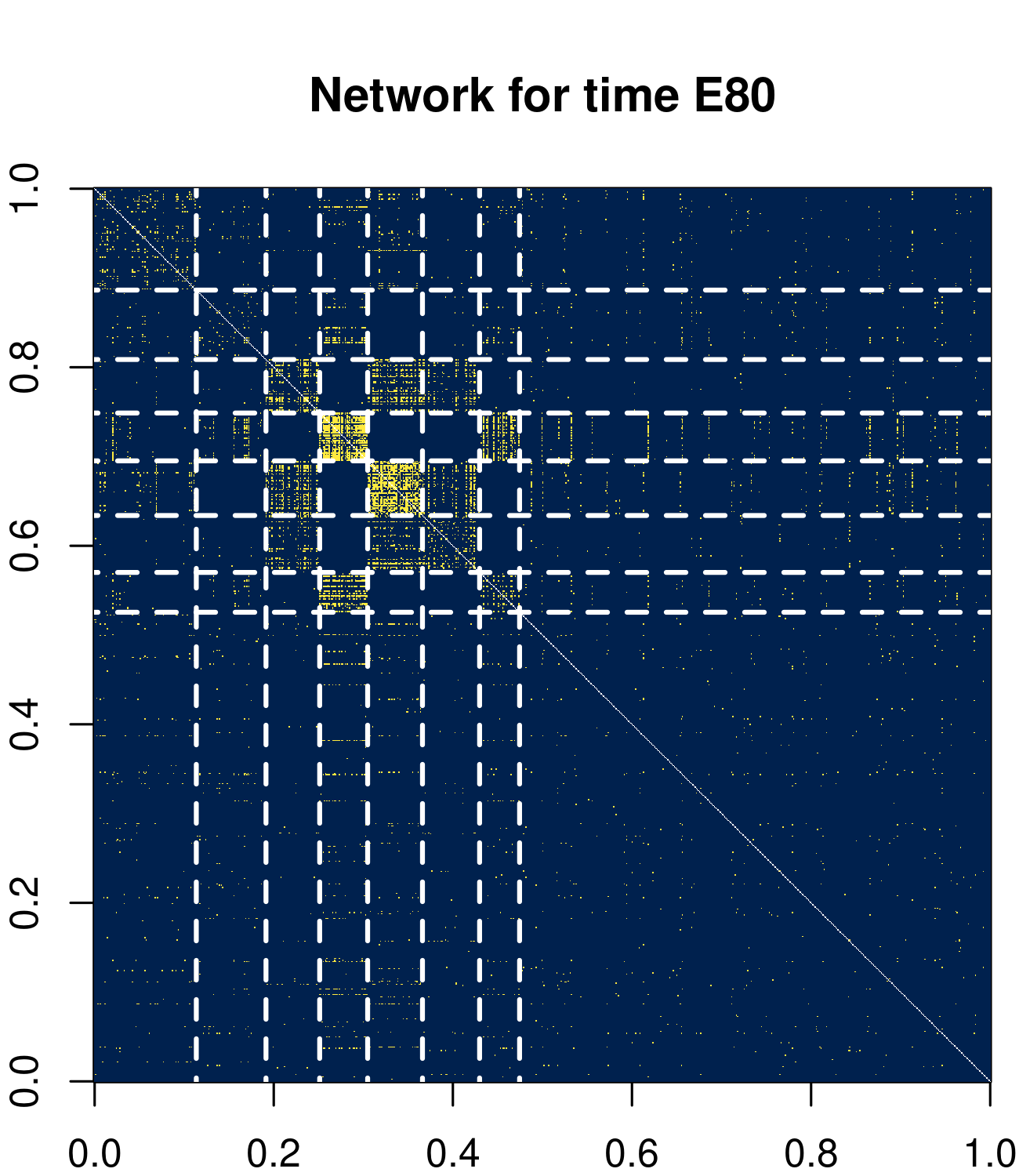}  
      \includegraphics[width=0.32\textwidth]{fig/networksosd/pnas_adj5}
         \includegraphics[width=0.32\textwidth]{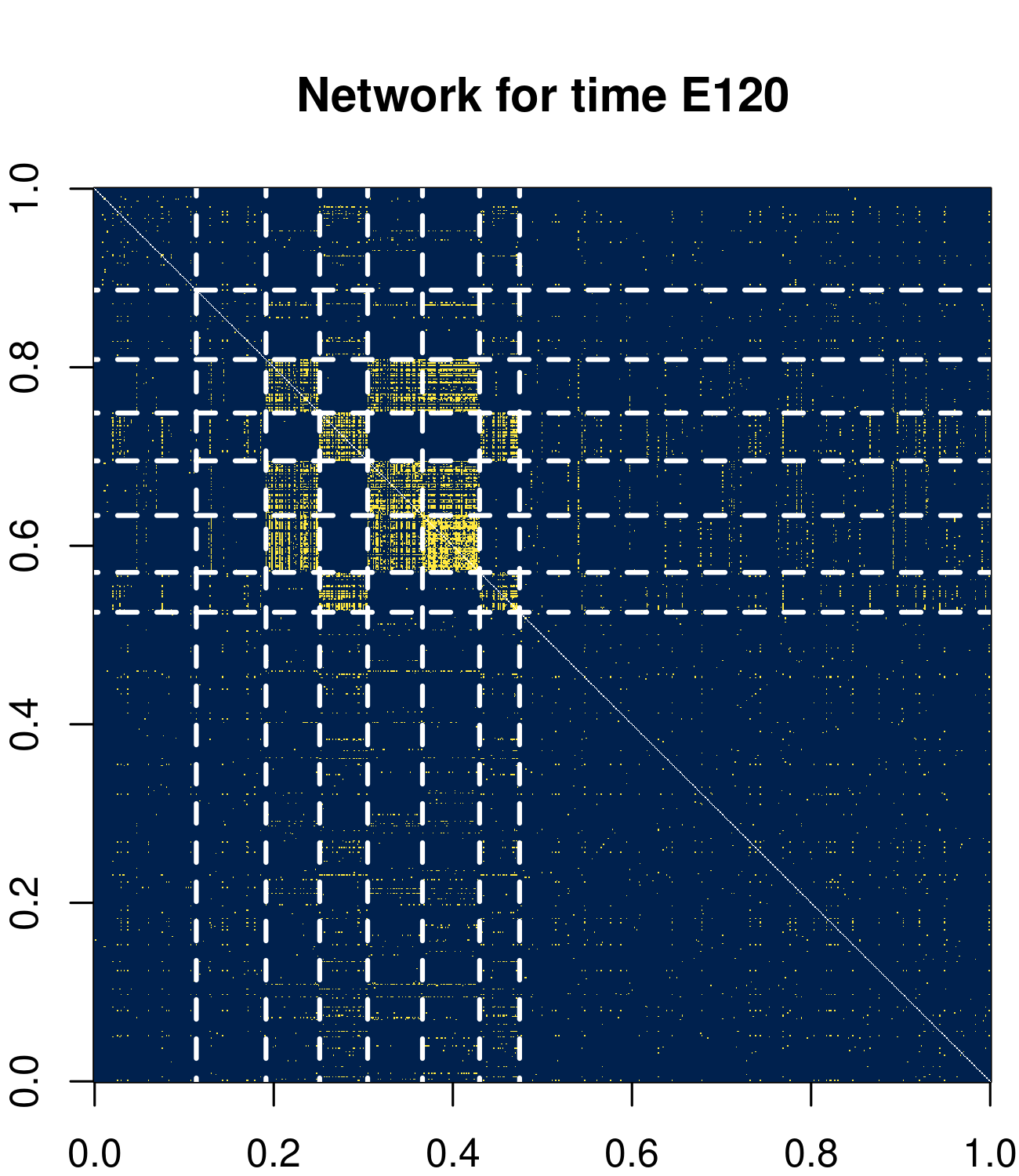}   \caption
   { \small Visualization similar to those in \Cref{fig:pnas_adj}, but for the adjacency matrices corresponding to the developmental times for 80, 90 and 120 days in the embryo (from left to right respectively).  }
    \label{fig:app:adj_4to6}
\end{figure}

\begin{figure}[H]
  \centering
   \includegraphics[width=0.32\textwidth]{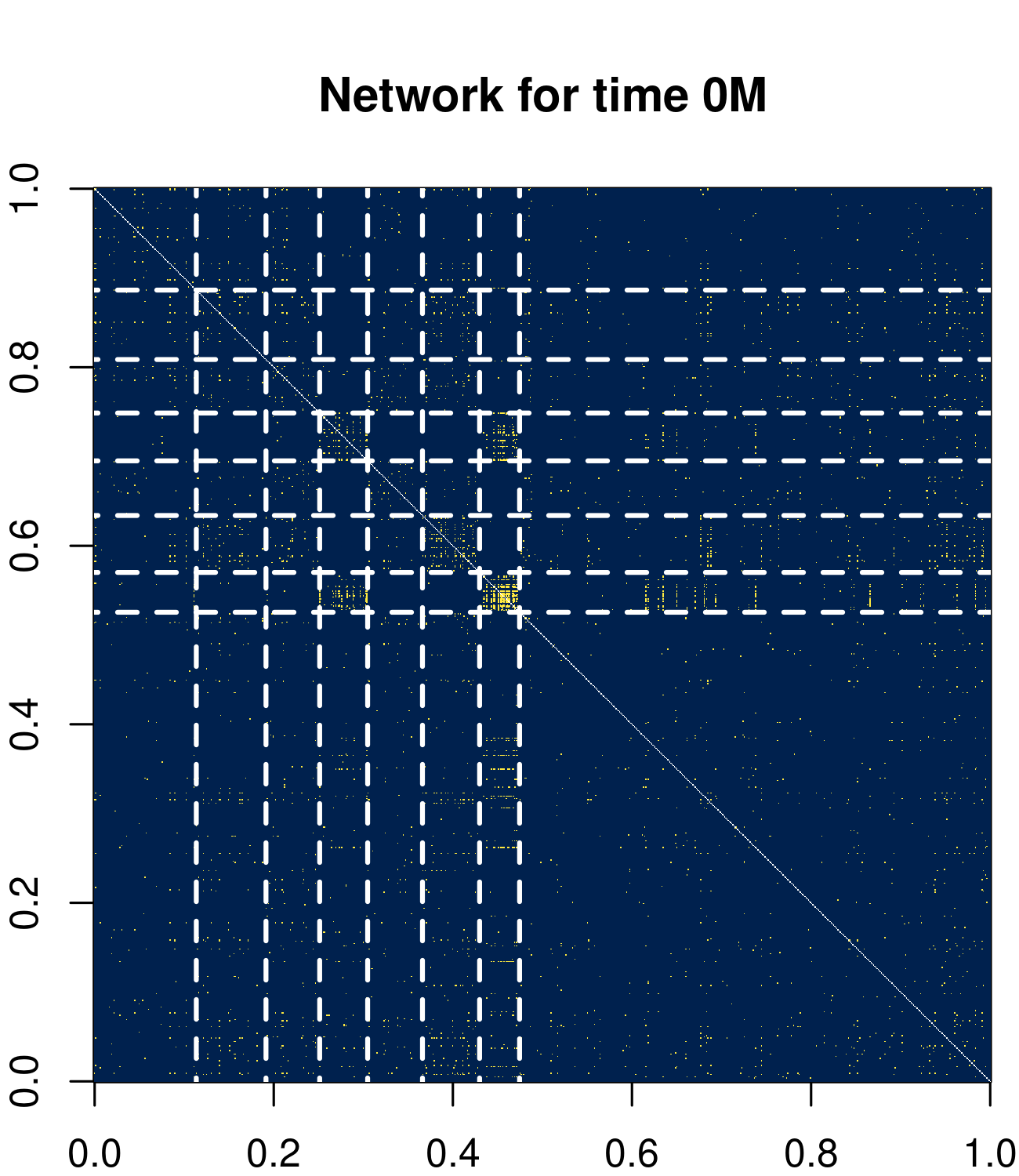}  
      \includegraphics[width=0.32\textwidth]{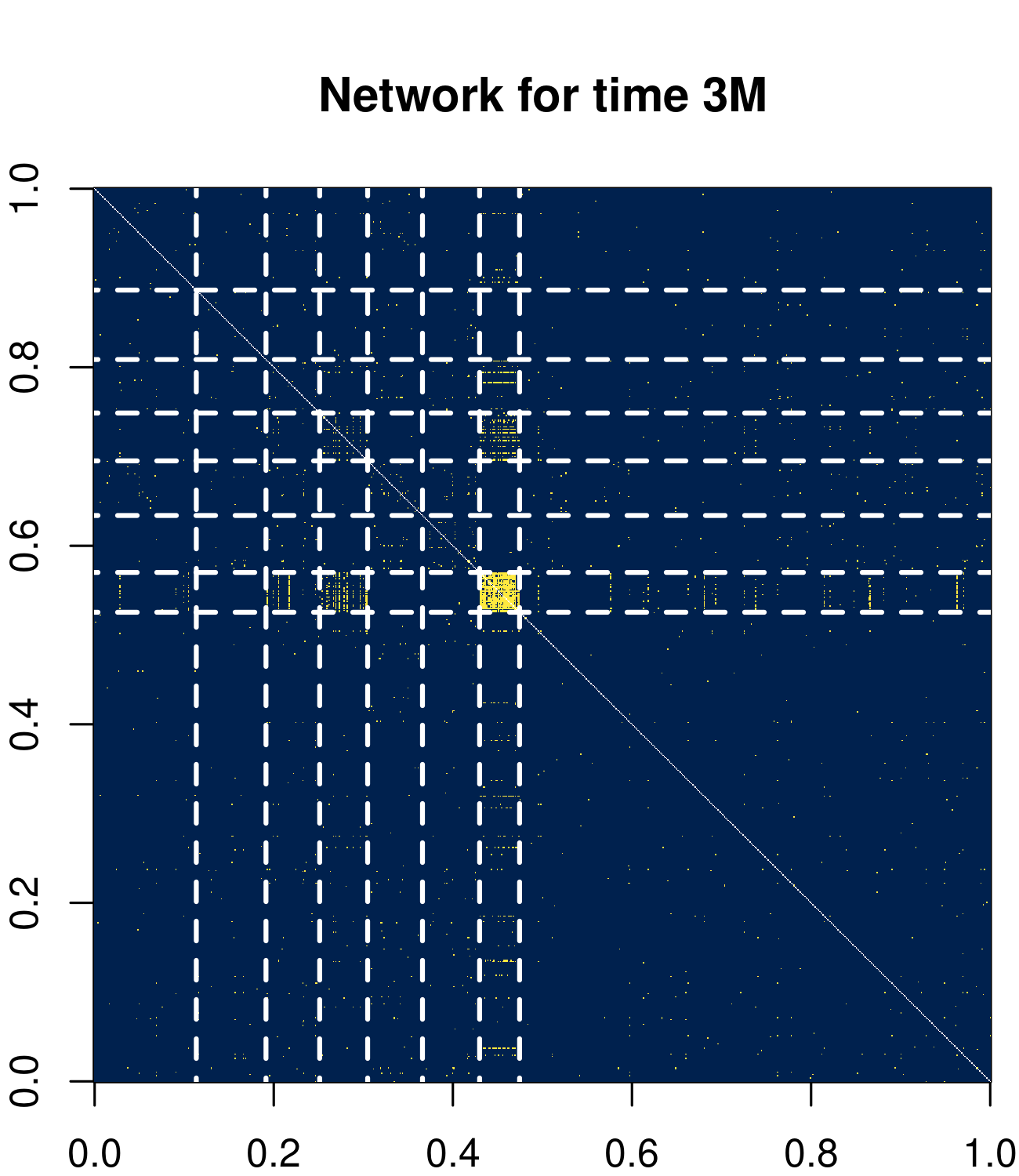}
         \includegraphics[width=0.32\textwidth]{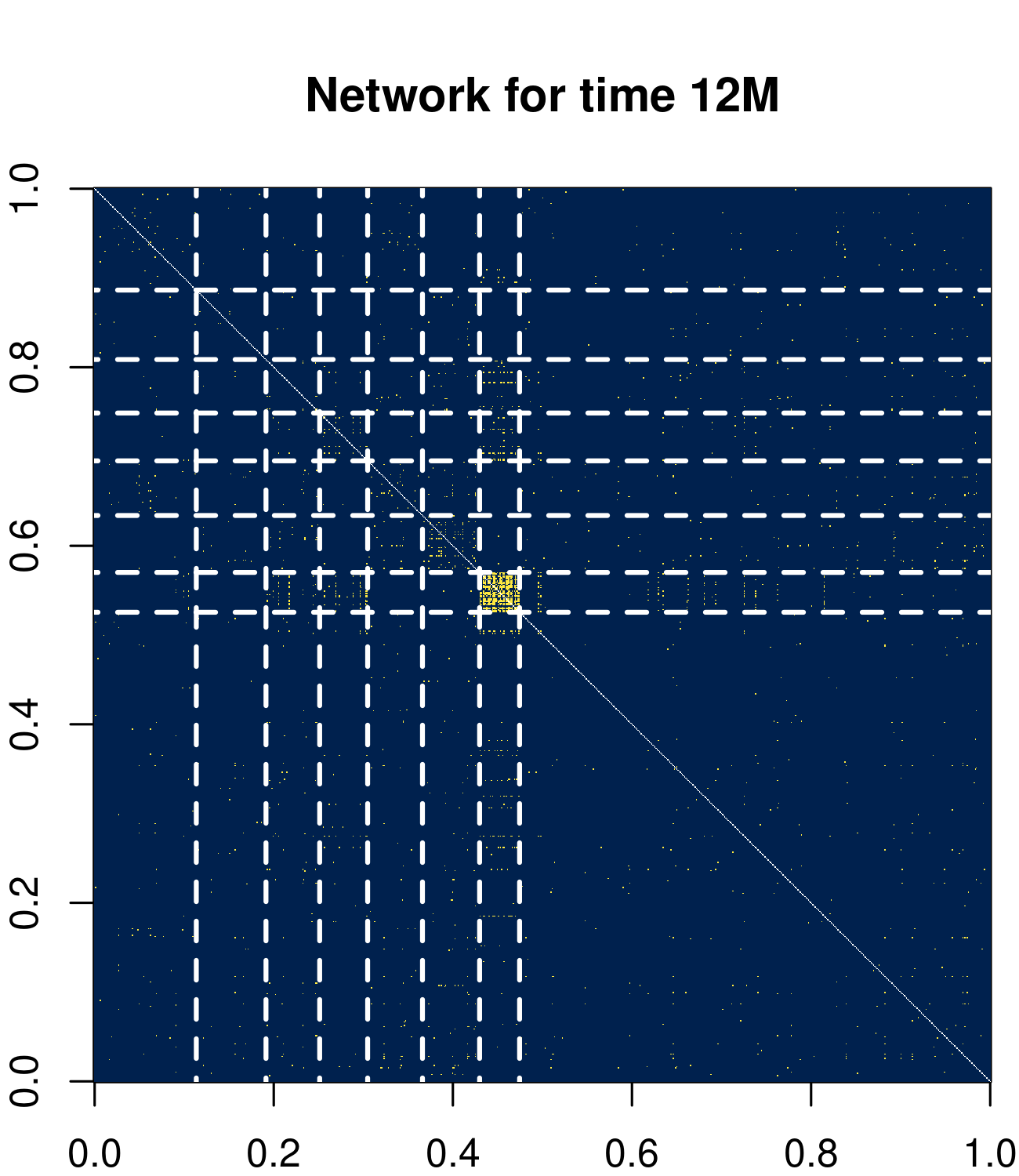}   \caption
   { \small Visualization similar to those in \Cref{fig:pnas_adj}, but for the adjacency matrices corresponding to the developmental times for 0, 3 and 12 months after birth (from left to right respectively).  }
    \label{fig:app:adj_7to9}
\end{figure}

\begin{figure}[H]
  \centering
   \includegraphics[width=0.32\textwidth]{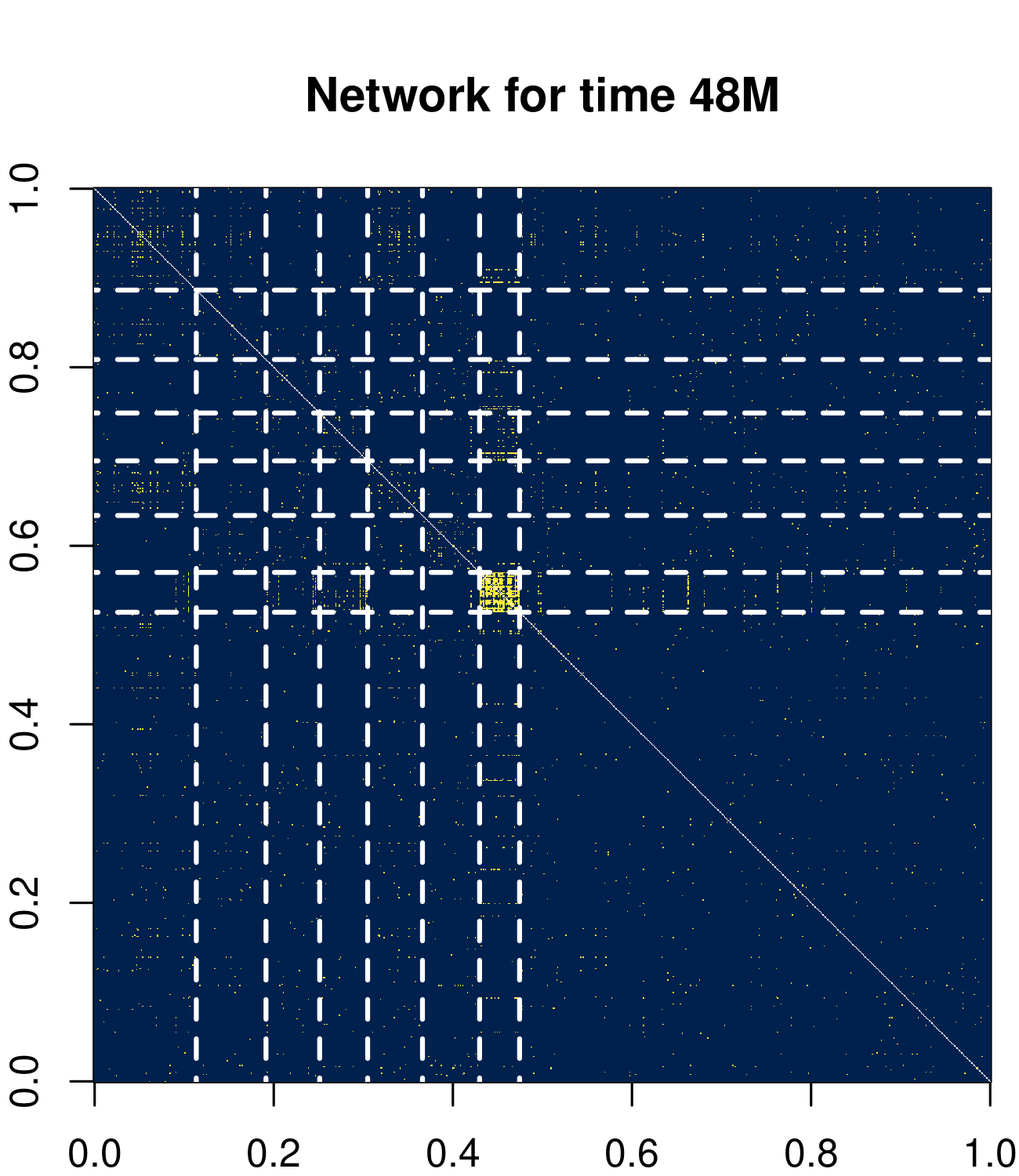}  
   \caption
   { \small Visualization similar to those in \Cref{fig:pnas_adj}, but for the adjacency matrix corresponding to the developmental times for 48 months after birth.  }
    \label{fig:app:adj_10}
\end{figure}

To better understand when each community has the highest connectivity in addition to \Cref{fig:app:adj_1to3} to \Cref{fig:app:adj_10}, we can also plot the within-community edge density for each of the eight communities across time, as shown in \Cref{fig:app:pnas_connectivity}. We see that the first seven communities are most dense at different stages of development, while the last community (the largest cluster) is always sparse
regardless of time. 
We do not plot the between-community edge densities in \Cref{fig:app:pnas_connectivity} 
since the average edge density between any community to all other communities is always less than 0.05 for any of the ten
networks.

\begin{figure}[H]
  \centering
   \includegraphics[width=0.9\textwidth]{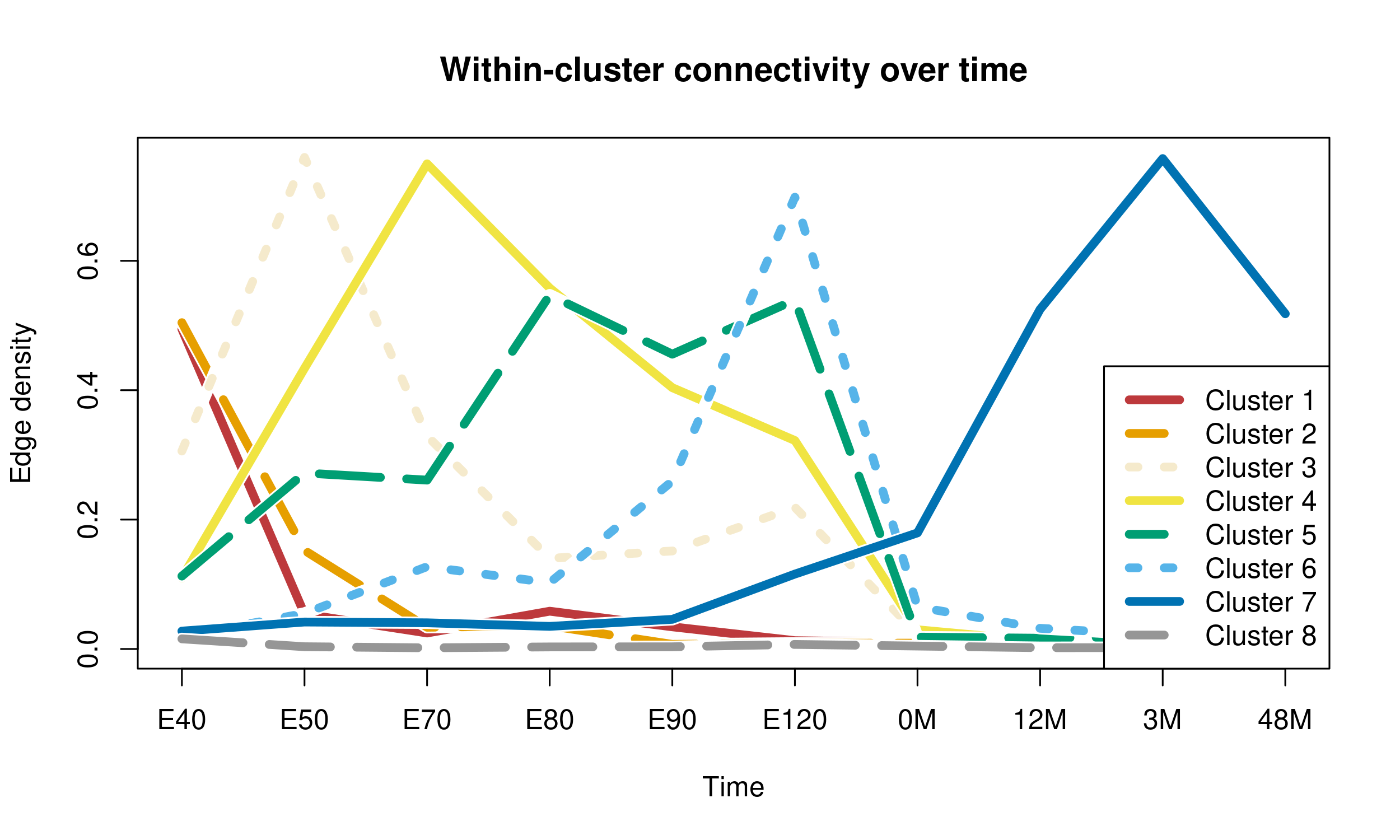}  
   \vspace{-0.5em}
   \caption
   { \small Edge density (calculated as number of observed edges divided by maximum number of possible edges) for each of the eight clusters across all time points. The clusters are colored by red, orange, pale, yellow, green, sky blue, dark blue, and gray in order from the first to last cluster, where the clusters are visualized by lines alternating between solid, dashed and dotted. }
    \label{fig:app:pnas_connectivity}
\end{figure}

\paragraph{Assessing the stability of the clustering with respect to tuning parameters.}
We assess the stability of our results by varying two important parameters
throughout our analysis, the correlation threshold needed when transforming the correlation matrix into an adjacency matrix and the number of clusters $K$.
We vary the correlation threshold among $\{0.68, 0.72, 0.77\}$ where a higher correlation threshold denotes fewer edges, and $K$ among $\{7,8,9\}$. 
We choose these particular correlation thresholds since the resulting networks still have desirable scale-free properties, and the total number of edges per network increases or decreases by 50\% on average when comparing the effect of changing the correlation
threshold from 0.72 to 0.68 or to 0.77 respectively.
Hence, we have a total of nine tuning-parameter-pairs, of which one yields the results shown in both \Cref{sec:data} and in \Cref{fig:app:adj_1to3} through \Cref{fig:app:adj_10} (i.e., the ``baseline communities'' corresponding to a correlation
threshold of $0.72$ and $K=8$).
Among the remaining eight tuning-parameter-pairs, we compare how each resulting communities of the genes compares to our baseline communities.

We focus on two of the eight tuning-parameter-pairs shown in \Cref{fig:app:pnas_sensitivity} which demonstrate communities that are most different from the baseline communities. Overall, these results show that our communities are stable across all tuning-parameter-pairs. On the left, we set the correlation threshold to be higher
than our baseline analysis in  \Cref{sec:data} (i.e., fewer edges) and set $K$ to be smaller.
On the right, we set the correlation threshold to be lower
than our baseline analysis in  \Cref{sec:data} (i.e., more edges) and set $K$ to be larger.
In both cases, we see an almost one-to-one mapping of baseline
communities to the new communities, where over 70\% of genes in each baseline community maps to a different new community. However, since we are comparing a baseline community with $K=8$ to either 
a community with $K=7$ or $K=9$, either one baseline community splits to multiple new communities, or multiple baseline communities merge to a particular new community. The results corresponding to the remaining six of eight tuning-parameter-pairs (not shown) are more stable than the ones shown in \Cref{fig:app:pnas_sensitivity} -- there, over 85\% of genes  in most baseline communities map to a different new community. 

\begin{figure}[H]
  \centering
   \includegraphics[width=0.48\textwidth]{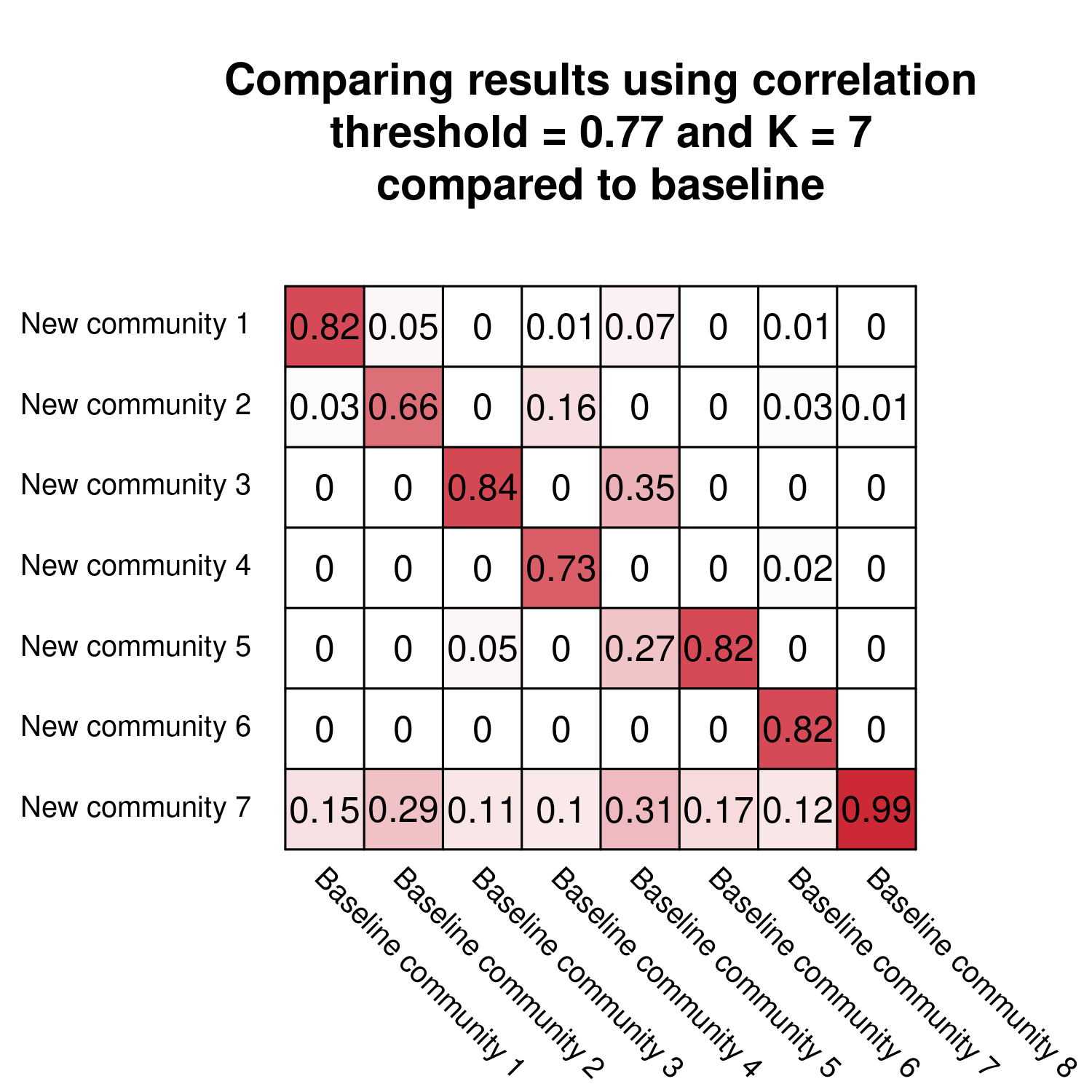}  
      \includegraphics[width=0.48\textwidth]{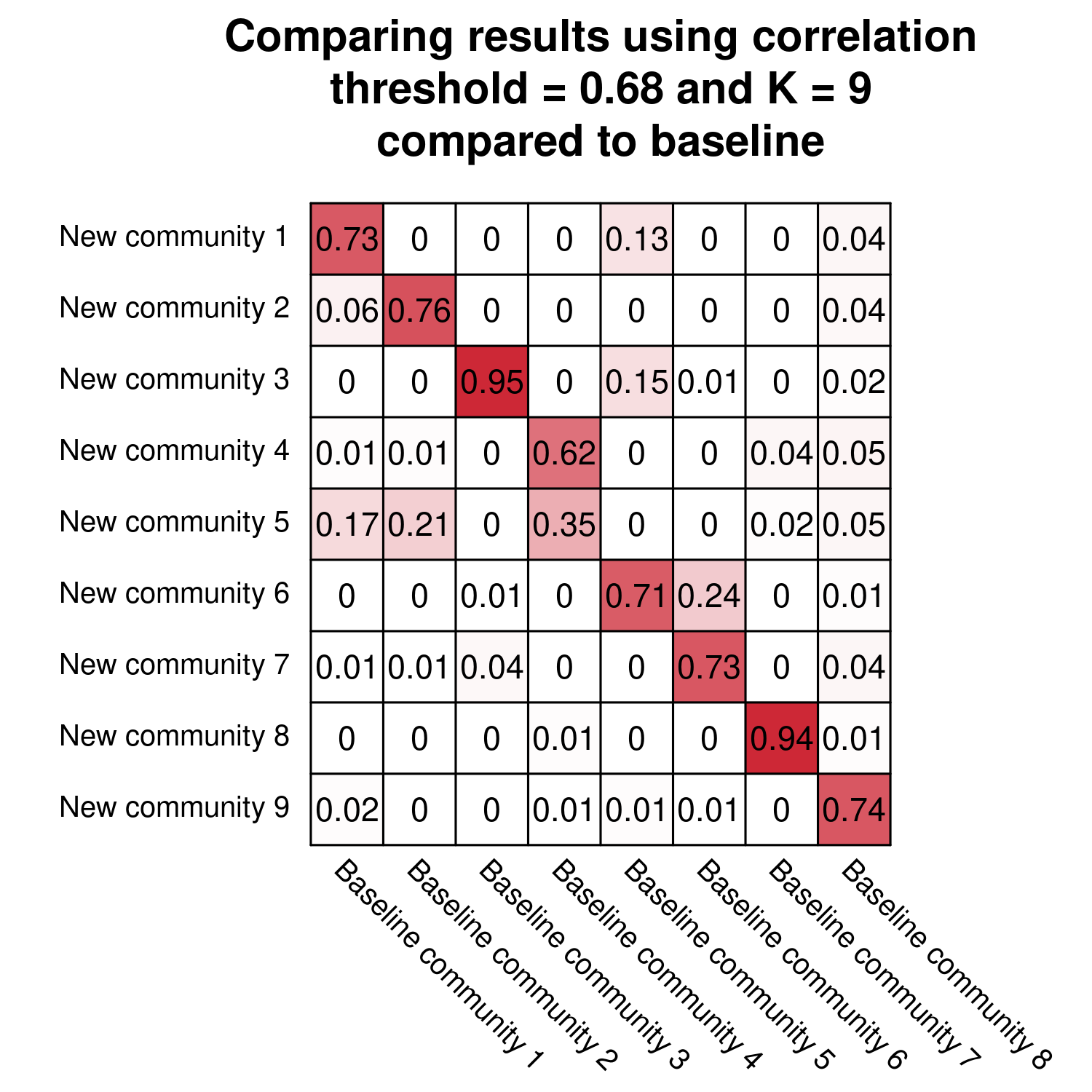}  \caption
   { \small Table displaying overlap between two sets of communities corresponding to different tuning-parameter-pairs and the baseline cluster. The value in each entry corresponds
   to the percentage of genes in a baseline community (i.e., column) belonging to a particular new community. The saturation of the red color corresponds to this value, where a higher saturation denotes values closer to 1. (Left): Comparison to results for a correlation threshold of 0.77 and $K=7$ (meaning there are fewer edges and communities). (Right): Comparison to results for a correlation threshold of 0.68 and $K=9$ (meaning there are more edges and communities).   }
    \label{fig:app:pnas_sensitivity}
\end{figure}

\end{document}